\documentclass[a4paper,10pt,draft]{amsart}

\usepackage{amssymb}
\usepackage{amsbsy}
\usepackage{amscd}
\usepackage[matrix,arrow]{xy}
\usepackage[mathscr]{eucal}
\usepackage{verbatim}
\usepackage{version}
\usepackage{color}

\newenvironment{NB}{
\color{red}{\bf NB}. \footnotesize 
}{}
\excludeversion{NB}
\newenvironment{NB2}{
\color{blue}%{\bf NB}. \footnotesize
}{}

\def\cal{\mathcal}
\def\Bbb{\mathbb}
\def\frak{\mathfrak}

\newcommand{\Amp}{\operatorname{Amp}}
\newcommand{\Coh}{\operatorname{Coh}}

\newcommand{\Ext}{\operatorname{Ext}}
\newcommand{\Hilb}{\operatorname{Hilb}}
\newcommand{\Hom}{\operatorname{Hom}}
\newcommand{\NS}{\operatorname{NS}}
\newcommand{\Pic}{\operatorname{Pic}}
\newcommand{\Quot}{\operatorname{Quot}}

\newcommand{\alg}{\operatorname{alg}}
\newcommand{\ch}{\operatorname{ch}}
\newcommand{\dslash}{/\!\!/} % algebro-geometric quot. (double slash)

\newcommand{\im}{\operatorname{im}}

\newcommand{\rk}{\operatorname{rk}}
\newcommand{\td}{\operatorname{td}}
\theoremstyle{plain}
 \newtheorem{thm}{Theorem}[subsection]
 \newtheorem{lem}[thm]{Lemma}
 \newtheorem{prop}[thm]{Proposition}
 \newtheorem{cor}[thm]{Corollary}
 
\theoremstyle{definition}
 \newtheorem{defn}[thm]{Definition}
 
\theoremstyle{remark}
 \newtheorem{rem}[thm]{Remark}

\numberwithin{equation}{section}
\begin{document}

%%%%%%%%%% authors & paper data %%%%%%%%%%

\title{Some moduli spaces of Bridgeland's stability conditions}
\author{Hiroki Minamide, Shintarou Yanagida, K\={o}ta Yoshioka}
\address{Department of Mathematics, Faculty of Science,
Kobe University,
Kobe, 657, Japan
}
\email{minamide@math.kobe-u.ac.jp, yanagida@math.kobe-u.ac.jp,
yoshioka@math.kobe-u.ac.jp}

\thanks{The second author is supported by JSPS Fellowships 
for Young Scientists (No.\ 21-2241). 
The third author is supported by the Grant-in-aid for 
Scientific Research (No.\ 22340010), JSPS}


\begin{abstract}
We shall study some moduli spaces of Bridgeland's semi-stable 
objects on abelian surfaces and K3 surfaces with Picard number 1.
Under some conditions,
we show that the moduli spaces are 
isomorphic to the moduli spaces of Gieseker semi-stable sheaves. 
We also study the ample cone of the moduli spaces.
\end{abstract}

\maketitle
%\tableofcontents

%%%%%%%%%%%%%%%%%%%%%%%%%%%%%%%%%%%%%%%%%%%%%%%%%%
%%%%%%%%%%%%%%%%%%%%%%%%%%%%%%%%%%%%%%%%%%%%%%%%%%
%%%%%%%%%%%%%%%%%%%%%%%%%%%%%%%%%%%%%%%%%%%%%%%%%%
\section{Introduction.}

Let $X$ be an abelian surface or a $K3$ surface over a field ${\frak k}$.
Denote by $\Coh(X)$ the category of coherent sheaves on $X$,
by ${\bf D}(X)$ the bounded derived category of $\Coh(X)$
and by $K(X)$ the Grothendieck group of ${\bf D}(X)$.

Let us fix an ample divisor $H$ on $X$.
For $\beta \in \NS(X)_{\Bbb Q}$ and $\omega \in {\Bbb Q}_{>0}H$,
Bridgeland \cite{Br:3} constructed a stability condition
$\sigma_{\beta,\omega}=({\frak A}_{(\beta,\omega)},Z_{(\beta,\omega)})$ 
on ${\bf D}(X)$.
Here ${\frak A}_{(\beta,\omega)}$ is a tilting of $\Coh(X)$, 
%and $Z_{(\beta,\omega)}:{\bf D}(X) \to {\Bbb C}$ is called 
%the stability function, 
and $Z_{(\beta,\omega)}:K(X) \to {\Bbb C}$ is a group homomorphism 
called the stability function.
In terms of the Mukai lattice 
$(H^*(X,{\Bbb Z})_{\alg}, \langle \cdot,\cdot \rangle)$, 
$Z_{(\beta,\omega)}$ is given by 
\begin{align*}
Z_{(\beta,\omega)}(E)=\langle e^{\beta+\sqrt{-1}\omega},v(E) \rangle,\quad
E \in K(X).
\end{align*}
Here $v(E):=\ch(E)\sqrt{\td_X}$ is the Mukai vector of $E$.
Hereafter for an object $E \in {\bf D}(X)$, we abbreviately write 
$Z_{(\beta,\omega)}(E) := Z_{(\beta,\omega)}([E])$, 
where $[E]$ is the class of $E$ in $K(X)$.
Let $\phi_{(\beta,\omega)}:{\frak A}_{(\beta,\omega)} \setminus \{ 0\}
\to (0,1]$ be the phase function, which is defined to be 
$Z_{(\beta,\omega)}(E)=|Z_{(\beta,\omega)}(E)|
e^{\pi \sqrt{-1}\phi_{(\beta,\omega)}(E)}$ 
for $0 \ne E \in {\frak A}_{(\beta,\omega)}$.

Let $w_1$ be a primitive isotropic Mukai vector of an object 
in ${\frak A}_{(\beta,\omega)}$.   
In this note, we shall study semi-stable objects $E$ with respect to
$(\beta,\omega)$ such that 
$\phi_{(\beta,\omega)}(E)=\phi_{(\beta,\omega)}(w_1)$.
Assume that there is a coarse moduli scheme $M_{(\beta,\omega)}(w_1)$ 
of stable objects and
$M_{(\beta,\omega)}(w_1)$ is projective.
In the case where $X$ is an abelian surface,
\cite{MYY} implies that this assumption is satisfied 
for any pair $(\beta,\omega)$.
Indeed $M_{(\beta,\omega)}(w_1)$ is the moduli space of
semi-homogeneous sheaves (up to shift). 

We set $X_1:=M_{(\beta,\omega)}(w_1)$.
Let ${\bf E}$ be a universal family 
as a complex of twisted sheaves on $X \times X_1$.
Let $\Phi_{X \to X_1}^{{\bf E}^{\vee}[1]}:
{\bf D}(X) \to {\bf D}^\alpha(X_1)$ be a twisted
Fourier-Mukai transform by ${\bf E}$, 
where 
$(\cdot)^\vee := {\bf R}{\cal H}om(\cdot,{\cal O})$ is the derived dual, 
and $\alpha$ is a representative of a suitable Brauer
class $[\alpha] \in H^2_{\text{\'{e}t}}(X_1,{\cal O}_{X_1}^{\times})$.
For simplicity,
we set $\Phi:=\Phi_{X \to X_1}^{{\bf E}^{\vee}[1]}$
and $\widehat{\Phi}:=\Phi_{X_1 \to X}^{{\bf E}[1]}$.

For $v \in H^*(X,{\Bbb Z})_{\alg}$ and $\beta \in \NS(X)_{\Bbb Q}$,
we set
\begin{equation}\label{eq:v:rda}
r_\beta(v):=-\langle v,\varrho_X \rangle,\quad
a_\beta(v):=-\langle v,e^\beta \rangle,\quad
d_\beta(v):=\frac{\langle v,H+(H,\beta)\varrho_X \rangle}{(H^2)}.
\end{equation}
Then
\begin{equation}\label{eq:v}
v=r_\beta(v)e^\beta+a_\beta(v) \varrho_X +
(d_\beta(v) H+D_\beta(v))+(d_\beta(v) H+D_\beta(v),\beta)\varrho_X,\;
D_\beta(v) \in H^{\perp} \cap \NS(X)_{\Bbb Q}.
\end{equation}

For a pair $(\beta,\omega)$,
${\cal M}_{(\beta,\omega)}(v)$ denotes the moduli stack of 
$\sigma_{(\beta,\omega)}$-semi-stable objects $E$ with $v(E)=v$.
Then we have the following result.
\begin{thm}[Theorem \ref{thm:isom}]\label{thm:main}
Let $X$ be an abelian surface, 
or a $K3$ surface with $\NS(X)={\Bbb Z}H$.
Assume the following conditions:
\begin{itemize}
\item[(1)]
There is a smooth projective surface 
$X_1$ which is the moduli space $M_{(\beta,\omega)}(w_1)$ 
of stable objects $E$ with $v(E)=w_1$.

\item[(2)]
$(\beta,\omega)$ satisfies
\begin{equation*}
\langle d_\beta(v)w_1-d_\beta(w_1)v,e^{\beta+\sqrt{-1}\omega} \rangle=0.
\end{equation*}

\item[(3)]
$(\beta,\omega)$ does not belong to any wall for $v$, 
or $w_1$ defines a wall $W_{w_1}$ for $v$ and $(\beta,\omega)$ 
belongs to exactly one wall $W_{w_1}$.
\end{itemize}
Then for a general $(\beta',\omega')$ in a neighborhood of
$(\beta,\omega)$ such that $(\beta',H)=(\beta,H)$, 
there is an ample 
divisor $H_1$ on $X_1$
such that 
${\cal M}_{(\beta',\omega')}(v)$ is
isomorphic to the moduli stack 
${\cal M}_{H_1}(u)^{ss}$ of Gieseker semi-stable
(twisted) sheaves with the Mukai vector $u$,
where $u=\Phi(v)$ or $u=\Phi(v)^{\vee}$.
In particular, there is a coarse moduli scheme
$M_{(\beta',\omega')}(v)$ which is isomorphic to
the moduli scheme of Gieseker semi-stable sheaves 
$\overline{M}_{H_1}(u)$. 
\end{thm}

As a corollary of this theorem, we get the following.
\begin{thm}[Theorem \ref{thm:projective}]
Let $X$ be an abelian surface or a $K3$ surface with $\Pic(X)={\Bbb Z}H$.
Assume that $(\beta,\omega)$ is general with respect to $v$.
Then there is a coarse moduli scheme $M_{(\beta,\omega)}(v)$
which is isomorphic to the projective scheme 
$\overline{M}_{\widehat{H}}(\Phi(v))$, where $\widehat{H}$ is a natural
ample class on $X_1$ associated to $H$.
\end{thm}

\begin{NB}
Assume that $X$ is an abelian surface. Then
every Fourier-Mukai transform $\Phi_{X \to X'}^{{\bf F}^{\vee}}
:{\bf D}(X) \to {\bf D}(X')$
is induced by
the moduli space $X'$ of stable sheaves on $X$ by Orlov.
Let $H'$ be the ample class of $X'$ corresponding to $H$
via $\Phi_{X \to X'}^{{\bf F}^{\vee}}$.  
Then $\Phi_{X \to X'}^{{\bf F}^{\vee}}$ induces an isomorphism
$M_{(\beta,\omega)}(v) \to M_{(\beta',\omega')}(v')$,
where $v'=\Phi_{X \to X'}^{{\bf F}^{\vee}}(v)$ and
$\beta' \in \NS(X')_{\Bbb Q}$ and $\omega' \in {\Bbb R}_{>0}H'$.
In particular, we may say that
Bridgeland's stability condition of this type is the
correct notion of stability in the relation of 
the symmetry of ${\bf D}(X)$.

\end{NB}

If $X$ is an abelian surface, we can also study the ample cone 
of $M_{(\beta,\omega)}(v)$ by using this theorem
(Corollary \ref{cor:ample}).

%%%%%%%%%%%%%%%%%%%%%%%%%%%%%%%%%%%%%%%%%%%%%%%%%%
%%%%%%%%%%%%%%%%%%%%%%%%%%%%%%%%%%%%%%%%%%%%%%%%%%
%%%%%%%%%%%%%%%%%%%%%%%%%%%%%%%%%%%%%%%%%%%%%%%%%%
\section{Preliminaries}

As in the introduction,
let $X$ be an abelian surface or a $K3$ surface over a field ${\frak k}$,
and fix an ample divisor $H$ on $X$.

\subsection{Notations for Mukai lattice}

We set $A^*_{\alg}(X)=\oplus_{i=0}^2 A^{i}_{\alg}(X)$ 
to be the quotient of the cycle group of $X$ 
by the algebraic equivalence.
Then we have $A^{0}_{\alg}(X)\cong {\Bbb Z}$,
$A^{1}_{\alg}(X)\cong \NS(X)$ and 
$A^{2}_{\alg}(X)\cong {\Bbb Z}$.
We denote the fundamental class of $A^{2}_{\alg}(X)$ by $\varrho_X$,
and express an element $x\in A^{*}_{\alg}(X)$ 
by $x=x_0+x_1+x_2 \varrho_X$ 
with $x_0 \in {\Bbb Z}$, $x_1 \in \NS(X)$ and $x_2 \in {\Bbb Z}$.
The lattice structure $\langle \cdot,\cdot\rangle$
of $A^*_{\alg}(X)$ is given 
by  
\begin{equation}\label{eq:mukai_pairing}
 \langle x,y \rangle:= (x_1,y_1)-(x_0 y_2+x_2 y_0),
\end{equation}
where $x=x_0+x_1+x_2 \varrho_X$ and $y=y_0+y_1+y_2 \varrho_X$. 
We will call $(A^*_{\alg}(X), \langle \cdot,\cdot \rangle)$ 
the Mukai lattice for $X$.
In the case of ${\frak k}={\Bbb C}$, this lattice is 
sometimes denoted by $H^{*}(X,{\Bbb Z})_{\text{alg}}$ in literature. 
In this paper, we will use the symbol $H^{*}(X,{\Bbb Z})_{\text{alg}}$ 
even when ${\frak k}$ is arbitrary.

The Mukai vector $v(E)\in H^*(X,{\Bbb Z})_{\alg}$ 
for $E \in \Coh(X)$ is defined by 
\begin{equation*}
\begin{split}
v(E):=&\ch(E) \sqrt{\td_X}\\
=&\rk E+c_1(E)+(\chi(E)-\varepsilon \rk E)\varrho_X \in H^*(X,{\Bbb Z})_{\alg}
\end{split}
\end{equation*}
where $\varepsilon=0,1$ 
according as $X$ is an abelian surface or a $K3$ surface.
For an object $E$ of ${\bf D}(X)$, 
$v(E)$ is defined by $\sum_{k}(-1)^k v(E^k)$,
where $(E^k)=(\cdots \to E^{-1} \to E^0 \to E^1 \to \cdots)$
is the bounded complex representing the object $E$.

For $\beta \in \NS(X)_{\Bbb Q}$,
we define the $\beta$-twisted semi-stability
replacing the usual Hilbert polynomial
$\chi(E(nH))$ by $\chi(E(-\beta+nH))$.

For a Mukai vector $v$, ${\cal M}_H^\beta(v)^{ss}$ denotes
the moduli stack of $\beta$-twisted semi-stable sheaves $E$ on $X$
with $v(E)=v$.
$\overline{M}_H^\beta(v)$ denotes
the moduli scheme of $S$-equivalence classes of
$\beta$-twisted semi-stable sheaves $E$ on $X$
with $v(E)=v$
and $M_H^\beta(v)$ denotes the open subscheme consisting of
$\beta$-twisted stable sheaves.
If $\beta=0$, then
we write $\overline{M}_H(v):=\overline{M}_H^\beta(v)$.

\subsection{Stability conditions and wall/chamber structure.}

Let us recall the stability conditions given in \cite[\S\,1]{MYY}.
For $E \in K(X)$ with \eqref{eq:v}, we have
\begin{equation}\label{eq:stab_func}
\begin{split}
Z_{(\beta,\omega)}(E)=&
\langle e^{\beta+\sqrt{-1}\omega},v(E) \rangle\\
=& -a_\beta(E)+\frac{(\omega^2)}{2}r_\beta(E)+d_\beta(E)(H,\omega)\sqrt{-1}.
\end{split}
\end{equation}
Then ${\frak A}_{(\beta,\omega)}$ is the tilt of $\Coh(X)$
with respect to
a torsion pair $({\frak T}_{(\beta,\omega)},{\frak F}_{(\beta,\omega)})$
defined by
\begin{enumerate}
\item
${\frak T}_{(\beta,\omega)}$ is generated by
$\beta$-twisted stable sheaves with $Z_{(\beta,\omega)}(E) \in {\Bbb H} \cup
{\Bbb R}_{<0}$.
\item
${\frak F}_{(\beta,\omega)}$ is generated by
$\beta$-twisted stable sheaves with $-Z_{(\beta,\omega)}(E) \in {\Bbb H} \cup
{\Bbb R}_{<0}$,
\end{enumerate}
where ${\Bbb H}:=\{z \in {\Bbb C} \mid \mathrm{Im}\, z>0 \}$
 is the upper half plane.

For a chosen $\beta \in \NS(X)_{\Bbb Q}$, 
let us write $b:=(\beta,H)/(H^2)\in{\Bbb Q}$.
Then $\beta=bH+\eta$ with $\eta \in H^{\perp} \cap \NS(X)_{\Bbb Q}$.
Now let us set 
\begin{align*}
{\frak H}:=
\{(\eta,\omega) \mid 
  \eta \in \NS(X)_{\Bbb Q},\ (\eta,H)=0,\ 
  \omega \in {\Bbb Q}_{>0}H \}.
\end{align*}

In \cite[\S\,1.4]{MYY}, we showed that 
the category ${\frak A}_{(bH+\eta,\omega)}$ changes only when 
$(\eta,\omega)$ moves across the \emph{wall for categories}.
Let us recall its definition.

\begin{defn}
Set
\begin{align*}
{\frak R}:=
\{u \in A^*_{\alg}(X) \mid 
  u \in (H+(H,bH)\varrho_X)^{\perp},\ 
  \langle u^2 \rangle=-2 \}.
\end{align*}
\begin{enumerate}
\item[(1)]
For $u \in {\frak R}$, 
we define a \emph{wall $W_u$ for categories} by
\begin{align*}
W_u := 
\{(\eta,\omega) \in {\frak H}_{\Bbb R} \mid
  \rk u \cdot (\omega^2)= -2 \langle  e^{bH+\eta},u \rangle \},
\end{align*}
where ${\frak H}_{\Bbb R}$ is the enlarged parameter space defined by
\begin{align*}
{\frak H}_{\Bbb R}:=
\{(\eta,\omega) \mid 
  \eta \in \NS(X)_{\Bbb R},\ (\eta,H)=0,\ 
  \omega \in {\Bbb R}_{>0}H \}.
\end{align*}

\item[(2)]
A connected component of
${\frak H}_{\Bbb R} \setminus \cup_{u \in {\frak R}} W_u$
is called a \emph{chamber for categories}.
\end{enumerate}
\end{defn}
If $X$ is an abelian surface or $(\omega^2)>2$, then
${\frak A}_{(bH+\eta,\omega)}$ does not depend on the choice of
$\omega$.

Now, the pair 
$\sigma_{(\beta,\omega)}=({\frak A}_{(\beta,\omega)},Z_{(\beta,\omega)})$ 
satisfies the requirement of stability conditions on ${\bf D}(X)$,
as mentioned in the introduction 
(see \cite[\S\,1.3]{MYY} for the proof).
In particular, the (semi-)stability of objects in 
${\frak A}_{(\beta,\omega)}$ with respect to $Z_{(\beta,\omega)}$
is well-defined.

\begin{NB}
If $(\eta,\omega),(\eta',\omega') \in {\frak H}_{\Bbb R}$
belong to the same chamber for categories, then
${\frak A}_{(bH+\eta,\omega)}={\frak A}_{(bH+\eta',\omega')}$.
If $X$ is an abelian surface or $(\omega^2)>2$, then
\begin{align*}
{\frak A}_{(bH+\eta,\omega)}=
\{E \in {\bf D}(X) \mid H^i(E)=0 \,  (i \ne -1, 0),\ 
\mu_{\max}(H^{-1}(E)) \leq b(H^2),\ \mu_{\min}(H^0(E))>b(H^2) \}
\end{align*}
for any $(\eta,\omega) \in {\frak H}_{\Bbb R}$.
Here $H^*$ is the cohomology with respect to ${\frak C}$. 
Thus ${\frak A}_{(bH+\eta,\omega)}$ does not depend on the choice
of $(\eta,\omega)$.
\end{NB}

\begin{defn}
$E \in {\bf D}(X)$ is called \emph{semi-stable} of phase $\phi$,
if there is an integer $n$ such that
$E[-n]$ is a semi-stable object of ${\frak A}_{(\beta,\omega)}$
with $\phi_{(\beta,\omega)}(E[-n])=\phi-n$. 
If we want to emphasize the dependence on the stability condition,
we say that $E$ is \emph{$\sigma_{(\beta,\omega)}$-semi-stable}.
\end{defn}

\begin{defn}\label{defn:phase(v)}
For a non-zero Mukai vector $v \in H^*(X,{\Bbb Z})_{\alg}$, 
we define $Z_{(\beta,\omega)}(v) \in {\Bbb C}$ and 
$\phi_{(\beta,\omega)}(v) \in (0,2]$ by 
\begin{equation}
\begin{split}
Z_{(\beta,\omega)}(v):=& \langle e^{\beta+\sqrt{-1} \omega},v \rangle 
= |Z_{(\beta,\omega)}(v)|e^{\pi \sqrt{-1}\phi_{(\beta,\omega)}(v)}.
\end{split}
\end{equation}
Then
$$
\phi_{(\beta,\omega)}(v(E))
=\phi_{(\beta,\omega)}(E) 
$$
for $0 \ne 
E \in {\frak A}_{(\beta,\omega)} \cup {\frak A}_{(\beta,\omega)}[1]$. 
\end{defn}

\begin{defn}
For a Mukai vector $v$,
${\cal M}_{(\beta,\omega)}(v)$ denotes the moduli stack of 
$\sigma_{(\beta,\omega)}$-semi-stable objects $E$ of 
${\frak A}_{(\beta,\omega)}$ with $v(E)=v$.
$M_{(\beta,\omega)}(v)$ denotes the moduli scheme of
the $S$-equivalence classes of  
$\sigma_{(\beta,\omega)}$-semi-stable objects $E$ of 
${\frak A}_{(\beta,\omega)}$ with $v(E)=v$, if it exists.
\end{defn}

Next we recall the wall/chamber structure 
\emph{for stabilities} (see \cite[\S\,3.1]{MYY} for details).
Let us set the rational number
\begin{equation*}
\begin{split}
d_{\beta,\min}:= 
\frac{1}{(H^2)}\min \{ \deg(E(-\beta))>0 \mid E \in K(X) \}
\in \frac{1}{\mathrm{d}[(\beta,H)] (H^2)}{\Bbb Z},\\
\end{split}
\end{equation*}
where $\mathrm{d}[x]$ is the denominator of $x \in {\Bbb Q}$.
Then $d_{\beta}(E) \in {\Bbb Z}d_{\beta,\min}$ for any $E \in K(X)$.

\begin{defn}\label{defn:wall:stability}
Let ${\cal C}$ be a chamber for categories, that is,
${\frak A}_{(bH+\eta,\omega)}$ is constant for 
$(\eta,\omega) \in {\cal C} \cap {\frak H}$.
For a Mukai vector $v$, 
let us set $r:=r_{bH+\eta}(v)$, $d:=d_{bH+\eta}(v)$ and 
$a:=a_{bH+\eta}(v)$ using \eqref{eq:v:rda}. 
\begin{enumerate}
\item[(1)]
Let $v_1$ be a Mukai vector, and set 
$r_1:=r_{bH+\eta}(v_1)$, $d_1:=d_{bH+\eta}(v_1)$ and
$a_1:=a_{bH+\eta}(v_1)$. 
For $v_1$ satisfying
\begin{enumerate}
\item
$0<d_1 <d$,
\item
$\langle v_1^2 \rangle<(d_1/d)\langle v^2 \rangle+
2d d_1 \varepsilon/ d_{b H+\eta,\min}^2$,
\item
$\langle v_1^2 \rangle \geq 
-2d_1^2 \varepsilon/d_{bH+\eta, \min}^2$,
\end{enumerate}
we define the \emph{wall for stabilities of type $v_1$} 
%(or \emph{wall for $v_1$} for simplicity)
as the set of 
$$
W_{v_1}:=\{(\eta,\omega) \in {\frak H}_{\Bbb R} \mid
(\omega^2)(dr_1-d_1 r)=2(-d \langle e^{bH+\eta},v_1 \rangle
+d_1 \langle e^{bH+\eta},v \rangle) \}. 
$$ 
\item[(2)]
A \emph{chamber for stabilities} is a connected component of 
${\cal C} \setminus \cup_{v_1} W_{v_1}$.  
\end{enumerate}
\end{defn}

If it is necessary to emphasize the dependence on $v$, then
we call $W_{v_1}$ by the \emph{wall for $v$}.
By \cite[Lemma~3.1.6]{MYY}, 
if $(\eta,\omega)$ and $(\eta',\omega')$ 
belong to the same chamber, then
${\cal M}_{(bH+\eta,\omega)}(v)=
{\cal M}_{(bH+\eta',\omega')}(v)$.
As we explained in \cite{MYY}, the above conditions (a),(b),(c) are
necessary numerical conditions for the walls of stability conditions.
Thus there may exist $W_{v_1}$ such that the stability condition does
not change by crossing $W_{v_1}$.
For an abelian surface,
we will give a necessary and sufficient condition of the wall 
where the stability condition does change
in \S \ref{subsect:dependence}. 

For the later discussions, we prepare  
\begin{defn}\label{defn:BBD}
For a complex $F \in {\bf D}(X)$,
let $^\beta H^p(F) \in {\frak A}_{(\beta,\omega)}$
denote the $p$-th cohomology group of $F$
with respect to the $t$-structure of ${\bf D}(X)$ associated to
${\frak A}_{(\beta,\omega)}$. 
\end{defn}
By \cite[Thm.~1.3.6]{BBD}, 
$^\beta H^p$ is a cohomological functor.

\subsection{Some calculation of Mukai vector}

As mentioned at \eqref{eq:v} in the introduction,
if $\beta \in \NS(X)_{\Bbb Q}$ is chosen, 
then any $v \in H^*(X,{\Bbb Z})_{\alg}$ can be expressed as 
\begin{equation*}
v=r_\beta(v)e^\beta+a_\beta(v) \varrho_X +
(d_\beta(v) H+D_\beta(v))+(d_\beta(v) H+D_\beta(v),\beta)\varrho_X,\quad
D_\beta(v) \in H^{\perp} \cap \NS(X)_{\Bbb Q}.
\end{equation*}
with $r_\beta(v),a_\beta(v),d_\beta(v)$ given by \eqref{eq:v:rda}.
The next lemma calculates the dependence of $a_\beta(v),d_\beta(v)$
on $\beta$, 
which will be repeatedly used in our discussion.

\begin{lem}\label{lem:beta-gamma}
For $v \in H^*(X,{\Bbb Z})_{\alg}$ and
$\beta, \gamma \in \NS(X)_{\Bbb Q}$,
\begin{equation*}
\begin{split}
d_\beta(v)H+D_\beta(v)
&=r_\gamma(v)(\gamma-\beta)+d_\gamma(v)H+D_\gamma(v),
\\
a_\beta(v)
&=a_\gamma(v)+(d_\gamma(v)H+D_\gamma(v),\gamma-\beta)
  +\frac{r_\gamma(v)}{2}((\beta-\gamma)^2).
\end{split}
\end{equation*}
In particular,
\begin{equation*}
d_\beta(v)=d_\gamma(v)+r_\gamma(v)\frac{\deg(\gamma-\beta)}{(H^2)}.
\end{equation*}
\end{lem}

\begin{proof}
We note that
\begin{equation*}
\begin{split}
e^\gamma
&=e^\beta e^{\gamma-\beta}
 =e^\beta+(\gamma-\beta+(\gamma-\beta,\beta)\varrho_X)
  +\frac{(\gamma-\beta)^2}{2}\varrho_X, 
\\
(d_\gamma(v)H+D_\gamma(v),\gamma)
&=(d_\gamma(v)H+D_\gamma(v),\beta)+(d_\gamma(v)H+D_\gamma(v),\gamma-\beta).
\end{split}
\end{equation*}
Hence we get
\begin{equation*}
\begin{split}
v&=r_\gamma(v) e^{\gamma}+a_\gamma(v) \varrho_X 
   +\left(d_\gamma(v)H+D_\gamma(v)+\bigl(d_\gamma(v)H+D_\gamma(v),\gamma\bigr)
    \varrho_X\right)
\\
&=r_\gamma(v) e^{\beta}
  +\left(a_\gamma(v)+(d_\gamma(v)H+D_\gamma(v),\gamma-\beta)
         +\frac{r_\gamma(v)}{2}((\beta-\gamma)^2)
   \right)\varrho_X 
\\
&\phantom{=}
  +\left(r_\gamma(v)(\gamma-\beta)+d_\gamma(v)H+D_\gamma(v)+
    \bigl(r_\gamma(v)(\gamma-\beta)+d_\gamma(v)H+D_\gamma(v),\beta\bigr)
   \varrho_X\right),
\end{split}
\end{equation*}
which implies the claims.
\end{proof}

%\begin{NB}
%\begin{lem}
%For $v,v_1 \in H^*(X,{\Bbb Z})_{\alg}$ and 
%$\beta,\gamma \in \NS(X)_{\Bbb Q}$,
%\begin{equation}
%\begin{split}
%d_\beta(v_1)r_\beta(v)-d_\beta(v)r_\beta(v_1)=&
%d_\gamma(v_1)r_\gamma(v)-d_\gamma(v)r_\gamma(v_1),\\
% d_\beta(v_1)a_\beta(v)-d_\beta(v)a_\beta(v_1)
%=& (r_\gamma(v_1)a_\gamma(v)-r_\gamma(v)a_\gamma(v_1))
%\frac{\deg(\gamma-\beta)}{(H^2)}\\
%& +(d_\gamma(v_1)a_\gamma(v)-d_\gamma(v)a_\gamma(v_1)).
%\end{split}
%\end{equation}
%
%\end{lem}
%
%\begin{proof}
%The claims follow from the following computations:
%\begin{equation}
%\begin{split}
%& \deg_\beta(v_1)v-\deg_\beta(v)v_1\\
%=&
%(r_\gamma(v_1)\deg(\gamma-\beta)+\deg_\gamma(v_1))
%(r_\gamma(v)e^\gamma+a_\gamma(v)\varrho_X+\xi+(\xi,\gamma)\varrho_X)\\
%&-
%(r_\gamma(v)\deg(\gamma-\beta)+\deg_\gamma(v))
%(r_\gamma(v_1)e^\gamma+a_\gamma(v_1)\varrho_X+\xi_1+(\xi_1,\gamma)\varrho_X)\\
%=& (\deg_\gamma(v_1)r_\gamma(v)-\deg_\gamma(v)r_\gamma(v_1))e^\gamma+
%(r_\gamma(v_1)a_\gamma(v)-r_\gamma(v)a_\gamma(v_1))
%\deg(\gamma-\beta)\varrho_X\\
%& +
%(\deg_\gamma(v_1)a_\gamma(v)-\deg_\gamma(v)a_\gamma(v_1))\varrho_X\\
%& +e^\gamma[(r_\gamma(v_1)\xi-r_\gamma(v)\xi_1)\deg(\gamma-\beta)
%+(\deg_\gamma(v_1)\xi-\deg_\gamma(v)\xi_1)].
%\end{split}
%\end{equation} 
%
%\end{proof}
%
%\end{NB}

\subsection{The homological correspondence.}

For $\gamma \in \NS(X)_{\Bbb Q}$,
let
\begin{equation}
\begin{split}
w_1:=& r_1 e^{\gamma}\\
=& r_1 \left(e^{\beta}+\frac{(\beta-\gamma)^2}{2} \varrho_X
+(\gamma-\beta+(\gamma-\beta,\beta)\varrho_X)\right)
\in H^*(X,{\Bbb Z})_{\alg}
\end{split}
\end{equation}
be a primitive isotropic Mukai vector 
such that $r_1 (\gamma-\beta,H)>0$.

Assume that there is a coarse moduli scheme 
$X_1:=M_{(\beta,\omega)}(w_1)$ of
stable objects and
$X_1$ is projective.
Let ${\bf E}$ be a universal object on $X \times X_1$
as a complex of 
twisted sheaves. 
We have $v({\bf E}_{|X \times \{x_1 \}})=w_1 $, $x_1 \in X_1$.
We set $v({\bf E}_{|\{x \} \times X_1}^{\vee})=r_1 e^{\gamma'}$,
$x \in X$.
\begin{defn}
For $C \in \NS(X)_{\Bbb Q}$,
we define $\widehat{C} \in \NS(X_1)_{\Bbb Q}$ by 
\begin{equation}
\Phi(C+(C,\gamma)\varrho_X)=
\begin{cases}
\widehat{C}+(\widehat{C},\gamma')\varrho_{X_1}, & r_1>0, \\
-(\widehat{C}+(\widehat{C},\gamma')\varrho_{X_1}), & r_1<0. 
\end{cases}
\end{equation} 
\end{defn}

\begin{lem}\label{lem:FM-w_1}
\begin{enumerate}
\item[(1)]
If $C$ belongs to the positive cone, then $\widehat{C}$ belongs to
the positive cone.
\item[(2)]
Let $H$ be an ample divisor on $X$.
Assume that one of the following conditions holds:
\begin{enumerate}
\item 
$X$ is an abelian surface,
\item
$\NS(X)={\Bbb Z}H$, 
\item
$M_{(\beta,\omega)}(w_1)$ is the moduli of $\mu$-stable vector bundles, 
that is,
$$
M_{(\beta,\omega)}(w_1)=
\begin{cases}
M_H(w_1),& \rk w_1>0,\\
M_H(-w_1), & \rk w_1<0.
\end{cases}
$$
\end{enumerate}
Then
$\widehat{H}$ is ample.
Thus we have the following.  
\begin{enumerate}
\item[1.]
If $r_1>0$, then
\begin{align*}
&\Phi(e^{\gamma})=-\frac{1}{r_1}\varrho_{X_1},\quad 
 \Phi(\varrho_X)=-r_1 e^{\gamma'},\\
&\Phi(dH+D+(dH+D,\gamma)\varrho_X)=
d\widehat{H}+\widehat{D}+(d \widehat{H}+\widehat{D},\gamma')\varrho_{X_1},
\end{align*}
where $D \in \NS(X)_{\Bbb Q} \cap H^\perp$.
\item[2.]
If $r_1<0$, then
\begin{align*}
&\Phi(e^{\gamma})=-\frac{1}{r_1}\varrho_{X_1},\quad 
 \Phi(\varrho_X)=-r_1 e^{\gamma'},\\
&\Phi(dH+D+(dH+D,\gamma)\varrho_X)=
-\left(d\widehat{H}+\widehat{D}+(d \widehat{H}+\widehat{D},\gamma')
       \varrho_{X_1}\right),
\end{align*}
where $D \in \NS(X)_{\Bbb Q} \cap H^\perp$.
\end{enumerate}
\end{enumerate}
\end{lem}

\begin{proof}
(1) 
If $r_1>0$, then
\begin{equation*}
\Phi(e^{\gamma})=-\frac{1}{r_1}\varrho_{X_1},\quad
\Phi(\varrho_X)=-r_1 e^{\gamma'},\quad
\Phi(C+(C,\gamma)\varrho_X)=
\widehat{C}+(\widehat{C},\gamma')\varrho_{X_1}.
\end{equation*}
If $r_1<0$, then
\begin{equation*}
\Phi(e^{\gamma})=-\frac{1}{r_1}\varrho_{X_1},\quad
\Phi(\varrho_X)=-r_1 e^{\gamma'},\quad
\Phi(C+(C,\gamma)\varrho_X)=
-(\widehat{C}+(\widehat{C},\gamma')\varrho_{X_1}).
\end{equation*}
Since $\Phi$ preserves the orientation of the Mukai lattice
(\cite{HMS} or Proposition \ref{prop:isotropic-K3} for a K3 surface, 
and \cite{Y:birational} for an abelian surface),
if $C$ belongs to the positive cone, then
$\widehat{C}$ also belongs to the positive cone.

(2) Assume that (a) or (b) holds. Then
$\widehat{C} \in \NS(X_1)_{\Bbb Q}$ is ample if and only if
$\widehat{C}$ belongs to the positive cone.
Since $H$ is ample, (1) implies the claim.
If (c) holds, then it is known that $\widehat{H}$ is ample
by the construction of the moduli space.  
\end{proof}

%We set $w_1:=v({\bf E}_{X \times \{x_1 \}})$.

\begin{NB}
\begin{lem}\label{lem:FM-w_1:2}
For $v \in H^*(X,{\Bbb Z})_{\alg}$,
$$
\Phi(v)=-r_1 a_\gamma(v) e^{\gamma'}-\frac{r_\gamma(v)}{r_1}\varrho_{X_1}
+\frac{r_1}{|r_1|}(d_\gamma(v) \widehat{H}+\widehat{D_\gamma(v)}+
(d_\gamma(v) \widehat{H}+\widehat{D_\gamma(v)},\gamma')\varrho_{X_1}).
$$
In  particular,
\begin{equation}
\begin{split}
r_{\gamma'}(\Phi(v))=& -r_1 a_\gamma(v)\\
a_{\gamma'}(\Phi(v))=& -\frac{r_\gamma(v)}{r_1}\\
d_{\gamma'}(\Phi(v))=& \frac{r_1}{|r_1|}d_\gamma(v)\\
D_{\gamma'}(\Phi(v))=& \frac{r_1}{|r_1|}\widehat{D_\gamma(v)}.
\end{split}
\end{equation}
\end{lem}
\end{NB}

\subsection{A lemma on angles of stability functions}

In this paper, we often compare the phases 
$\phi_{(\beta,\omega)}(E)$ and $\phi_{(\beta,\omega)}(F)$ of two objects 
$E,F \in{\bf D}(X)$.
For that, it is convenient to use the next function,
which was introduced in \cite[\S\,1.3]{MYY}.

\begin{defn}\label{defn:area}
For $E, E' \in K(X)$,
we set
\begin{equation*}
\Sigma_{(\beta,\omega)}(E',E):=\det
\begin{pmatrix}
\mathrm{Re} Z_{(\beta,\omega)}(E') & \mathrm{Re}Z_{(\beta,\omega)}(E)\\
\mathrm{Im} Z_{(\beta,\omega)}(E') & \mathrm{Im} Z_{(\beta,\omega)}(E)
\end{pmatrix}.
\end{equation*}
We also set
$\Sigma_{(\beta,\omega)}(v',v):=\Sigma_{(\beta,\omega)}(E',E)$
for $v(E)=v,v(E')=v'$.
\end{defn}

Then we have $\Sigma_{(\beta,\omega)}(E',E) \ge 0$ 
if and only if $\phi_{(\beta,\omega)}(E)-\phi_{(\beta,\omega)}(E')\ge0$
(see \cite[Remark~1.3.5]{MYY}).

Next, let us prepare some notations for the phase of the stability
function.  
%We denote the phase $\phi$ of $E$ by $\phi_{(\beta,\omega)}(E)$. 
For $\phi \in {\Bbb R}$,
$P(\phi)$ denotes the category of semi-stable objects
$E \in {\bf D}(X)$ with $\phi_{(\beta,\omega)}(E)=\phi$.

By the Harder-Narasimhan property of 
the stability function $Z_{(\beta,\omega)}$,
for any $0 \neq E \in {\bf D}(X)$ we have a collection of triangles
\begin{align*}%\label{diag:HN}
 \xymatrix{
    0=E_0   \ar[rr]         &
  & E_1     \ar[dl] \ar[rr] &
  & E_2     \ar[r]  \ar[dl] & \cdots \ar[r]     
  & E_{n-1} \ar[rr]         &
  & E_n =E  \ar[dl]
 \\
                            & A_1 \ar[ul]^{[1]} 
  &                         & A_2 \ar[ul]^{[1]}
  &                         &                     
  &                         & A_n \ar[ul]^{[1]}
 }
\end{align*}
such that $A_i \in P(\phi_i)$ with $\phi_1 > \phi_2 > \cdots >\phi_n$. 
Let us denote $\phi_{\max}(E) := \phi_1$ and $\phi_{\min}(E) := \phi_n$.
%If the pair $(\beta,\omega)$ is understood, 
%then we write $\phi(\cdot):=\phi_{(\beta,\omega)}(\cdot)$.

%\begin{defn}
%For $E, E' \in {\frak A}_{(\beta,\omega)}$, we set
%\begin{equation*}
%\Sigma_{(\beta,\omega)}(E',E):=\det
%\begin{pmatrix}
%\mathrm{Re} Z_{(\beta,\omega)}(E') & \mathrm{Re}Z_{(\beta,\omega)}(E)\\
%\mathrm{Im} Z_{(\beta,\omega)}(E') & \mathrm{Im} Z_{(\beta,\omega)}(E)
%\end{pmatrix}.
%\end{equation*}
%\end{defn}

%Note that $Z_{(\beta,\omega)}(E)$, 
%$\phi_{(\beta,\omega)}(E)$ and $\Sigma_{(\beta,\omega)}(E,F)$ 
%are determined by the 
%Mukai vectors $v(E)$ and $v(F)$.
%Let us denote by
%$\phi_{(\beta,\omega)}(v)$ and 
%$\Sigma_{(\beta,\omega)}(v,w)$ 
%the corresponding functions 
%of Mukai vectors $v,w \in H^*(X,{\Bbb Z})_{\alg}$.

Now we want to state the main Lemma~\ref{lem:range-of-phi} in this subsection.
Let us recall the notations in the introduction:
$\Phi:=\Phi_{X \to X_1}^{{\bf E}^{\vee}[1]}$
and $\widehat{\Phi}:=\Phi_{X_1 \to X}^{{\bf E}[1]}$.

\begin{lem}\label{lem:range-of-phi}
We set $\phi:=\phi_{(\beta,\omega)}$.
For a torsion free (twisted) sheaf $E$ on $X_1$,
$F:=\widehat{\Phi}(E)$ satisfies the following properties.
\begin{enumerate}
\item[(1)]
$\Hom({\bf E}_{|X \times \{ x_1 \}},F[k])=0$ for $k \ne 0,1$
and $\Hom({\bf E}_{|X \times \{ x_1 \}},F)=0$ except
finitely many points $x_1 \in X_1$.
\item[(2)]
$\phi(w_1)-1<\phi_{\min}(F) \leq \phi_{\max}(F)<\phi(w_1)+1$.
\item[(3)]
Assume that $Z_{(\beta,\omega)}(F) \in 
{\Bbb R}Z_{(\beta,\omega)}(w_1)$, i.e.,
$\Sigma_{(\beta,\omega)}(F,w_1)=0$. 
If $F$ is not a $\sigma_{(\beta,\omega)}$-semi-stable object of
${\frak A}_{(\beta,\omega)}$,
then there is an exact sequence of torsion free sheaves
$$
0 \to E_1 \to E \to E_2 \to 0
$$
such that 
$$
\phi(w_1)+1>\phi_{\max}(\widehat{\Phi}(E_1))
\geq \phi_{\min}(\widehat{\Phi}(E_1)) >\phi(w_1)
$$
and
$$
\phi(w_1) \geq \phi_{\max}(\widehat{\Phi}(E_2))
\geq \phi_{\min}(\widehat{\Phi}(E_2))>\phi(w_1)-1.
$$
In particular, $\Sigma_{(\beta,\omega)}(w_1,\widehat{\Phi}(E_1))>0$ and
$\Sigma_{(\beta,\omega)}(\widehat{\Phi}(E_2),w_1)>0$.
%\begin{NB}
%Old version:
%Assume that $\phi(w_1)=\phi(F)$. 
%If $F$ is not semi-stable with respect to $(\beta,\omega)$,
%then there is an exact sequence of torsion free sheaves
%$$
%0 \to E_1 \to E \to E_2 \to 0
%$$
%such that $\phi(w_1)+1>\phi(\widehat{\Phi}(E_1))>\phi(w_1)$
%and
%$\phi(w_1)>\phi(\widehat{\Phi}(E_2))>\phi(w_1)-1$.
%\end{NB}
\end{enumerate}
\end{lem}

\begin{proof}
(1)
We note that 
\begin{equation*}
\begin{split}
\Hom({\bf E}_{|X \times \{ x_1 \}},F[k])=%& 
\Hom(\Phi({\bf E}_{|X \times \{ x_1 \}}),\Phi(F)[k])%\\
=%& 
\Hom({\frak k}_{x_1}[-1],E[k]).
\end{split}
\end{equation*}
Hence $\Hom({\bf E}_{|X \times \{ x_1 \}},F[k])=0$ for
$k \ne 0,1$.
Since $E$ is torsion free, we also have 
$\Hom({\bf E}_{|X \times \{ x_1 \}},F)=0$
except for finitely many points $x_1 \in X_1$.
Thus (1) holds.

(2)
We first prove that
$\phi_{\max}(F)<\phi(w_1)+1$.
Assume that there is an exact triangle
\begin{equation}\label{seq:lem:range-of-phi}
F_1 \to F \to F_2 \to F_1[1]
\end{equation} 
with $\phi_{\min}(F_1) \geq \phi(w_1)+1$
and $\phi_{\max}(F_2) <\phi(w_1)+1$.

Since 
\begin{align*}
\Hom({\bf E}_{|X \times \{ x_1 \}},F_1[k])
=\Hom(F_1,{\bf E}_{|X \times \{ x_1 \}}[2-k])^{\vee}
\end{align*}
and $\phi({\bf E}_{|X \times \{ x_1 \}}[2-k])-\phi_{\min}(F_1) \leq 1-k$,
we have 
\begin{align*}
&\Hom({\bf E}_{|X \times \{ x_1 \}},F_1[k])=0\ \text{ for } k\ge 2, 
\\
&\Hom({\bf E}_{|X \times \{ x_1 \}},F_1[1])=0
\text{ except for finitely many } x_1 \in X_1.
\end{align*}
We also have  
$\Hom({\bf E}_{|X \times \{ x_1 \}},F_2[k])=0$ for $k \leq -1$,
since $\phi_{\max}(F_2[k])-\phi({\bf E}_{|X \times \{ x_1 \}}) < 1+k \leq 0$.
Then taking $\Hom({\bf E}_{|X \times \{ x_1 \}},\bullet)$ 
of \eqref{seq:lem:range-of-phi} and using (1),
we find that 
\begin{align*}
&\Hom({\bf E}_{|X \times \{ x_1 \}},F_1[k])=0 \ \text{ for } k \leq -1,
\\
&\Hom({\bf E}_{|X \times \{ x_1 \}},F_1)=0 \ 
\text{ except for finitely many } x_1 \in X_1.
\end{align*}

By $\Hom({\bf E}_{|X \times \{ x_1 \}},F_1[k])=0$ for $k \ne 0,1$
and all $x_1 \in X_1$,
$\Phi(F_1)$ is represented by a complex $V_{-1} \to V_0$
of locally free sheaves.
Since $\Hom({\bf E}_{|X \times \{ x_1 \}},F_1[k])=0$, $k=0,1$
except for finitely many 
$x_1 \in X_1$, $\Phi(F_1)=0$.
Thus $F_1=0$.

We next prove that
$\phi_{\min}(F)>\phi(w_1)-1$.
Assume that there is an exact triangle
\begin{equation}\label{seq:lem:range-of-phi2}
F_1 \to F \to F_2 \to F_1[1]
\end{equation} 
such that $\phi_{\max}(F_2) \leq \phi(w_1)-1$
and $\phi_{\min}(F_1) >\phi(w_1)-1$.
Then 
$$
\Hom({\bf E}_{|X \times \{ x_1 \}},F_1[k])=
\Hom(F_1,{\bf E}_{|X \times \{ x_1 \}}[2-k])^{\vee}=0
$$ 
for $k \geq 3$.
We also have  
\begin{align*}
&\Hom({\bf E}_{|X \times \{ x_1 \}},F_2[k])=0\ \text{ for } k \leq 0,
\\ 
&\Hom({\bf E}_{|X \times \{ x_1 \}},F_2[1])=0\  
\text{ except for finitely many } x_1 \in X_1.
\end{align*}
Then by \eqref{seq:lem:range-of-phi2},
$\Hom({\bf E}_{|X \times \{ x_1 \}},F_2[k])=0$
for $k \ne 1$, which implies that
$\Phi(F_2)$ is locally free.
On the other hand,
$\Hom({\bf E}_{|X \times \{ x_1 \}},F_2[1])=0$ except for finitely many
$x_1 \in X_1$.
Therefore $F_2=0$.

(3)
We first prove that
$\phi_{\max}(F)>\phi(w_1)$ and $\phi_{\min}(F)<\phi(w_1)$.
If $\phi_{\max}(F) \leq \phi(w_1)$, then
since $Z_{(\beta,\omega)}(F) \in {\Bbb R}Z_{(\beta,\omega)}(w_1)$, 
we see that
$\phi_{\min}(F)=\phi_{\max}(F)=\phi(w_1)$.
Thus $F$ is a $\sigma_{(\beta,\omega)}$-semi-stable object of
${\frak A}_{(\beta,\omega)}$, 
which is a contradiction.
If $\phi_{\min}(F) \geq \phi(w_1)$, then
we also get that
$\phi_{\min}(F)=\phi_{\max}(F)=\phi(w_1)$,
which contradict
our assumption on $F$.
Therefore the claims hold.

Then
we have an exact triangle
\begin{equation}
F_1 \to F \to F_2 \to F_1[1]
\end{equation} 
such that $\phi_{\min}(F_1)>\phi(w_1)$ and
$\phi_{\max}(F_2) \leq \phi(w_1)$.
%Since $\phi(w_1)-1<\phi_{\min}(F) \leq \phi_{\max}(F)<\phi(w_1)+1$,
By (2), we have
$\phi_{\max}(F_1)<\phi(w_1)+1$ and 
$\phi_{\min}(F_2)>\phi(w_1)-1$.
Since $\phi_{\min}(F_1)>\phi(w_1)$ and
$\phi_{\max}(F_2) \leq \phi(w_1)$,
we have $\Hom({\bf E}_{|X \times \{ x_1 \}},F_1[k])=0$
for $k \geq 2$ and
$\Hom({\bf E}_{|X \times \{ x_1 \}},F_2[k])=0$ for $k<0$.
Moreover $\Hom({\bf E}_{|X \times \{ x_1 \}},F_2)=0$ except
finitely many $x_1 \in X_1$.
We set $E_i:=\Phi_{X \to X_1}^{{\bf E}^{\vee}[1]}(F_i)$.
Then $E_i$ are torsion free sheaves fitting in an exact sequence
$$
0 \to E_1 \to E \to E_2 \to 0.
$$ 
Since $\widehat{\Phi}(E_i)=F_i$, $i=1,2$,
we get the claim.
\end{proof}

\section{Relation with $\mu$-semi-stability.}\label{sect:mu-semi-stable}

In this section, we fix the pair $(\beta,\omega)$ and set
$\phi:=\phi_{(\beta,\omega)}$.
We shall study the relation of Bridgeland stability with $\mu$-semi-stability.
The main statement is given in Theorem~\ref{thm:mu-semi-stable}.
We shall freely use the notations 
$\Phi:=\Phi_{X \to X_1}^{{\bf E}^{\vee}[1]}$
and $\widehat{\Phi}:=\Phi_{X_1 \to X}^{{\bf E}[1]}$.

\subsection{A polarization of $X_1$.}

In this subsection, we introduce a ${\Bbb Q}$-divisor 
$\widehat{L}$ on $X_1$ and
show that it is ample under suitable assumptions.
\begin{lem}\label{lem:JHF}
For a $\sigma_{(\beta,\omega)}$-semi-stable object $F$ 
of ${\frak A}_{(\beta,\omega)}$
with
$\phi(F)=\phi(w_1)$,
we have an exact sequence in ${\frak A}_{(\beta,\omega)}$
\begin{equation}\label{eq:JHF}
0 \to F_1 \to F \to F_2 \to 0
\end{equation}
such that $F_1$ is a $\sigma_{(\beta,\omega)}$-semi-stable 
object of ${\frak A}_{(\beta,\omega)}$
satisfying
$\phi(F_1)=\phi(w_1)$ and 
$\Hom(F_1,{\bf E}_{|X \times \{x_1 \}})=0$
for all $x_1 \in X_1$, and
$F_2$ is $S$-equivalent to 
$\oplus_i {\bf E}_{|X \times \{x_i \}}$, $x_i \in X_1$.
Since $\Hom(F_1,F_2)=0$, \eqref{eq:JHF} is uniquely determined by
$F$.
\end{lem}

\begin{proof}
For a non-zero morphism
$\varphi:F \to {\bf E}_{|X \times \{x_1 \}}$,
$\phi(F)=
\phi({\bf E}_{|X \times \{x_1 \}})$ implies
that $\varphi$ is surjective and $\ker \varphi$
is a $\sigma_{(\beta,\omega)}$-semi-stable 
object of ${\frak A}_{(\beta,\omega)}$
with $\phi(\ker \varphi)=
\phi(w_1)$.
Apply this procedure successively, we finally obtain
a subobject $F_1$ of $F$ such that
$\phi(F_1)=\phi(w_1)$ and 
$\Hom(F_1,{\bf E}_{|X \times \{x_1 \}})=0$
for all $x_1 \in X_1$.
Then $F_2:=F/F_1$ is $S$-equivalent to 
$\oplus_i {\bf E}_{|X \times \{x_i \}}$, $x_i \in X_1$.
\end{proof}

We set $\gamma-\beta:=\lambda H+\nu$, $(\nu,H)=0$.
Then we see that
\begin{equation}
\begin{split}
& (a_\beta(v)d_\beta(w_1)-a_\beta(w_1)d_\beta(v))(H^2)\\
=& r_1 a_\gamma(v)\deg(\gamma-\beta)+
r_1(d_\gamma(v)H+D_\gamma(v),\gamma-\beta)
\deg(\gamma-\beta)-r_1 d_\gamma(v) (H^2)\frac{((\gamma-\beta)^2)}{2}\\
=& r_1 a_\gamma(v)\lambda(H^2)+r_1 \lambda^2 \frac{(H^2)^2}{2}d_\gamma(v)
+r_1(D_\gamma(v),\nu)\lambda(H^2)-r_1 d_\gamma(v)(H^2)\frac{(\nu^2)}{2}
\end{split}
\end{equation}
and
\begin{equation}
r_\beta(v)d_\beta(w_1)-r_\beta(w_1)d_\beta(v)=
r_\gamma(v)d_\gamma(w_1)-r_\gamma(w_1)d_\gamma(v)=-r_\gamma(w_1)d_\gamma(v).
\end{equation}
Then
\begin{equation}\label{eq:Sigma}
\begin{split}
\frac{\Sigma_{(\beta,\omega)}(v,w_1)}{(H,\omega)}
=&
(r_\beta(v)d_\beta(w_1)-r_\beta(w_1)d_\beta(v))\frac{(\omega^2)}{2}-
(a_\beta(v)d_\beta(w_1)-a_\beta(w_1)d_\beta(v))\\
=& 
-\frac{r_1}{2}\left((\omega^2)+\lambda^2(H^2)-(\nu^2) \right)d_\gamma(v)
-r_1 \lambda(a_\gamma(v)+(D_\gamma(v),\nu)). 
\end{split}
\end{equation}

\begin{defn}\label{defn:L}
For $\gamma-\beta:=\lambda H+\nu$, $\nu \in H^{\perp}$,
we set
$$
L:=\frac{(\omega^2)+\lambda^2(H^2)-(\nu^2)}{2(H^2)}H
+\lambda \nu \in \NS(X)_{\Bbb Q}.
$$
\end{defn}
By \eqref{eq:Sigma}, we get the following lemma.
\begin{lem}\label{lem:Sigma-L}
\begin{equation}
\begin{split}
\frac{\Sigma_{(\beta,\omega)}(v,w_1)}{(H,\omega)}
=
-r_1(c_1(v)-(\rk v)\gamma,L)-r_1 \lambda a_\gamma(v).
 \end{split}
\end{equation}

\end{lem}

Let us study the properties of $L$.
We note that 
\begin{equation}
\begin{split}
(L^2)=& \frac{1}{4(H^2)}
\left[\left((\omega^2)+\lambda^2(H^2)-(\nu^2) \right)^2+4\lambda^2(H^2)(\nu^2)
\right]\\
>& \frac{1}{4(H^2)}
\left(\lambda^2(H^2)+(\nu^2) \right)^2 \geq 0.
\end{split}
\end{equation}
\begin{NB}
$L$ depends on $(\omega^2)$.
\end{NB}
Since $(L,H)>0$, $L$ belongs to the positive cone.
\begin{lem}
Assume that one of the following conditions holds:
\begin{enumerate}
\item
$X$ is an abelian surface.
\item
$\NS(X)={\Bbb Z}H$.
\item
$\nu=0$.
\end{enumerate}
Then $L$ is an ample ${\Bbb Q}$-divisor.
\end{lem}

\begin{proof}
Since $L$ belongs to the positive cone,
$L$ is an ample divisor, if $X$ is an abelian surface or
$\NS(X)={\Bbb Z}H$.
If $\nu=0$, then 
$L \in {\Bbb Q}_{>0} H$.
Thus $L$ is also an ample divisor.
\end{proof}

\begin{rem}
\begin{enumerate}
\item[(1)]
Assume that $X$ is an abelian surface.
Then since $(\widehat{L}^2)=(L^2)>0$ and
$(\widehat{L},\widehat{H})=(L,H)>0$,
$\widehat{L}$ is ample.
\item[(2)]
Assume that $X$ is a $K3$ surface and
one of the conditions (b), (c) of Lemma \ref{lem:FM-w_1} (2) holds.
If $\nu=0$, then $\widehat{L} \in {\Bbb Q}_{>0}\widehat{H}$ is also
ample.  
\end{enumerate}
\end{rem}

\begin{NB}
We set 
$$
\frac{c_1(F)-\rk F \gamma}{\chi(F(-\gamma))}=xH+y, y \in H^{\perp}.
$$
Then 
$$
x\frac{((\omega^2)+\lambda^2 (H^2)-(\nu^2))}{2}+\lambda(\nu,y)=-\lambda.
$$
As an equation of $(\omega,\nu) \in {\frak H}_{\Bbb R}$,
it defines a half sphere ${\frak H}_{\Bbb R}$.
For a given half sphere, $xH+y$ is uniquely determined.
\end{NB}

\begin{lem}\label{lem:slope-condition}
For $E \in K(X_1)$, 
$Z_{(\beta,\omega)}(\widehat{\Phi}(E))
\in {\Bbb R}Z_{(\beta,\omega)}(w_1)$
if and only if 
\begin{equation}\label{eq:slope}
%\frac{(c_1(E)-\rk E \gamma',\widehat{L})}{\rk E}=
(c_1(E)-\rk E \gamma',\widehat{L})=\frac{\lambda}{|r_1|}\rk E.
\end{equation}
\begin{NB}
Moreover if $\langle v(E)^2 \rangle \geq 0$ and
$v(\Phi_{X_1 \to X}^{{\bf E}[1]}(E)) \not \in {\Bbb Z}w_1$, 
then \eqref{eq:slope}
implies that
$\rk E>0$ if and only if $d_\beta(\Phi_{X_1 \to X}^{{\bf E}[1]}(E))>0$.
Thus $Z_{(\beta,\omega)}(\Phi_{X_1 \to X}^{{\bf E}[1]}(E))
\in {\Bbb R}_{>0} Z_{(\beta,\omega)}(w_1)$.
\end{NB}
\end{lem}

\begin{proof}
We set $F:=\widehat{\Phi}(E)$.
By Lemma \ref{lem:FM-w_1}, we have
$$
\widehat{c_1(F(-\gamma))}=\frac{r_1}{|r_1|}
c_1(\Phi(F)(-\gamma'))=\frac{r_1}{|r_1|} c_1(E(-\gamma')).
$$
By using Lemma \ref{lem:Sigma-L}, we see that
\begin{equation}\label{eq:Sigma-Lhat}
\begin{split}
\frac{\Sigma_{(\beta,\omega)}(F,w_1)}{(H,\omega)}=&
-r_1(c_1(F(-\gamma)),L)-r_1 \lambda a_\gamma(F)\\
= &-r_1 (\widehat{c_1(F(-\gamma))},\widehat{L})+\lambda \rk E \\
=& -|r_1|(c_1(E(-\gamma')),\widehat{L})+\lambda \rk E.
\end{split}
\end{equation}
Therefore \eqref{eq:slope} holds.
\begin{NB}
We take $v=\pm v(\Phi_{X_1 \to X}^{{\bf E}[1]}(E))$
with $d_\beta(v)>0$.
Since $\langle v^2 \rangle \geq 0$,
there is a $\sigma_{(\beta,\omega)}$-semi-stable 
object $F \in {\frak A}_{(\beta,\omega)}$ 
with $v(F)=v$. 
Since $v \not \in {\Bbb Z}w_1$,
we can find a $\sigma_{(\beta,\omega)}$-semi-stable object $F'$ 
such that
$\Hom({\bf E}_{|X \times \{x_1 \}},F')=
\Hom(F',{\bf E}_{|X \times \{x_1 \}},F')=0$
for all $x_1 \in X_1$ and $v(F')=v-mw_1$, $m \in {\Bbb Z}_{ \geq 0}$.
Then
$\Phi_{X \to X_1}^{{\bf E}^{\vee}[1]}(F')$ is a locally free sheaf on $X_1$
with $v(\Phi_{X \to X_1}^{{\bf E}^{\vee}[1]}(F'))=\pm v(E)+m \varrho_{X_1}$. 
Hence $\rk E>0$ if and only if 
$d_\beta(\Phi_{X_1 \to X}^{{\bf E}[1]}(E))>0$.
%The last claim is a consequence of Lemma \ref{lem:range-of-phi}.
\end{NB}
\end{proof}

\begin{NB}
Assume that $\lambda \ne 0$.
Then $r_1 \lambda>0$.
If $Z_{(\beta,\omega)}(v) \in {\Bbb R}Z_{(\beta,\omega)}(w_1)$,
then $d_\beta(v) \ne 0$.
If $d_\gamma(v)=d_\beta(v)-r \lambda=0$, then
$0<d_\beta(v)=\frac{r}{r_1}r_1 \lambda$.
Since $\langle v^2 \rangle=(D^2)-2\frac{r}{r_1}r_1 a_\gamma(v)$,
$r_1 a_\gamma(v) \leq 0$ and the equality holds if
$D=0$ and $\langle v^2 \rangle=0$, which is equivalent to
$v \in {\Bbb R} w_1$.
Hence we assume that $d_\gamma(v) \ne 0$.

\begin{equation}
\begin{split}
-\frac{r_1^2}{2}d_\gamma(v)^2(\omega^2)=&
\frac{r_1^2}{2}d_\gamma(v)^2(\lambda^2(H^2)-(\nu^2))
+r_1^2 \lambda d_\gamma(v) a_\gamma(v)
+r_1^2 \lambda d_\gamma(v)(D,\nu)\\
\geq & \frac{r_1^2}{2}d_\gamma(v)^2(\lambda^2(H^2)
-(\nu^2))+r_1^2 \lambda d_\gamma(v) a_\gamma(v)
-r_1^2 \lambda d_\gamma(v) \sqrt{-(D^2)}\sqrt{-(\nu^2)}\\
=& \frac{r_1^2}{2}d_\gamma(v)^2\lambda^2(H^2)+
\frac{r_1^2}{2}
\left(d_\gamma(v) \sqrt{-(\nu^2)}-\lambda \sqrt{-(D^2)}\right)^2
+\frac{r_1^2}{2} \lambda^2(D^2)+r_1^2 \lambda d_\gamma(v) a_\gamma(v)\\
=& \frac{r_1^2}{2}\lambda^2(d_\gamma(v)^2(H^2)+(D^2))
+\frac{r_1^2}{2}
\left(d_\gamma(v) \sqrt{-(\nu^2)}-\lambda \sqrt{-(D^2)}\right)^2
+r_1^2 \lambda d_\gamma(v) a_\gamma(v)\\
=& \frac{r_1^2}{2}\lambda^2 \langle v^2 \rangle+
r_1^2\lambda^2 r a_\gamma(v)
+\frac{r_1^2}{2}
\left(d_\gamma(v) \sqrt{-(\nu^2)}-\lambda \sqrt{-(D^2)}\right)^2
+r_1^2 \lambda d_\gamma(v) a_\gamma(v)\\
 \geq & r_1 \lambda(\lambda r+d_\gamma(v))r_1 a_\gamma(v)
= r_1 \lambda d_\beta(v) r_1 a_\gamma(v). 
\end{split}
\end{equation}
Hence $d_\beta(v) (r_1 a_\gamma(v))< 0$.
\end{NB}

We also have the following as a consequence of \eqref{eq:Sigma-Lhat}.

\begin{lem}\label{lem:slope-phase}
For $w \in H^*(X,{\Bbb Z})_{\alg}$ with $\rk \Phi(w)=-r_1 a_\gamma(w)>0$,
$$
\Sigma_{(\beta,\omega)}(w,w_1) \underset{(<)}{>}0
\Longleftrightarrow
\frac{(c_1(\Phi(w)(-\gamma')),\widehat{L})}{\rk \Phi(w)}
\underset{(>)}{<}
\frac{\lambda}{|r_1|}.
$$
\end{lem}

\subsection{$\mu$-semi-stability and 
$\sigma_{(\beta,\omega)}$-semi-stability.}

From now on, we assume that $\widehat{L}$ is ample.
\begin{thm}\label{thm:mu-semi-stable}
Assume that $\widehat{L}$ is ample.
Let $F$ be an object of
${\bf D}(X)$
such that
$Z_{(\beta,\omega)}(F) \in {\Bbb R}Z_{(\beta,\omega)}(w_1)$. 
\begin{enumerate}
\item[(1)]
$F$ is a $\sigma_{(\beta,\omega)}$-semi-stable object 
of ${\frak A}_{(\beta,\omega)}$
if and only if
$\Phi(F)$ fits in an exact triangle
$$
E_1 \to \Phi(F) \to E_2[-1] \to E_1[1],
$$
where $E_1$ is a $\mu$-semi-stable torsion free sheaf
on $X_1$ and $E_2$ is a 0-dimensional sheaf on $X_1$. 
\item[(2)]
Assume that $F$ is a $\sigma_{(\beta,\omega)}$-semi-stable 
object of ${\frak A}_{(\beta,\omega)}$.
\begin{enumerate}
\item
$\Phi(E)$ is a torsion free sheaf on $X_1$ if and only if
$\Hom(F,{\bf E}_{|X \times \{ x_1 \}})=0$ for all $x_1 \in X_1$.    
\item
$\Phi(E)^{\vee}$ is a torsion free sheaf on $X_1$ if and only if
$\Hom({\bf E}_{|X \times \{ x_1 \}},F)=0$ for all $x_1 \in X_1$.    
\end{enumerate}
\end{enumerate}
\end{thm}

\begin{rem}\label{rem:mu-semi-stable}
We take $\beta' \in \NS(X_1)_{\Bbb Q}$ with
$(\beta'-\gamma',\widehat{L})=\frac{\lambda}{|r_1|}$
and $\omega' \in {\Bbb R}_{>0}\widehat{L}$ with
$(\widehat{L}^2)>2$.
Let $F$ be an object of ${\bf D}(X)$ in Theorem \ref{thm:mu-semi-stable}.
Then Theorem \ref{thm:mu-semi-stable} implies that
$F$ is a 
$\sigma_{(\beta,\omega)}$-semi-stable object of
${\frak A}_{(\beta,\omega)}$
if and only if
$\Phi(F)[1]$ is a $\sigma_{(\beta',\omega')}$-semi-stable
object of ${\frak A}_{(\beta',\omega')}$ with 
$\phi_{(\beta',\omega')}(\Phi(F)[1])=1$. 
In \S \ref{subsect:another}, we give a more precise relation
with Bridgeland stability on $X_1$.  
\end{rem}

\begin{cor}\label{cor:mu-semi-stable}
Assume that $w_1$ does not define a wall for $v$.
\begin{enumerate}
\item[(1)]
Every $\sigma_{(\beta,\omega)}$-semi-stable object $F$ with $v(F)=v$
satisfies
$$
\Hom({\bf E}_{|X \times \{ x_1 \}},F)=
\Hom(F,{\bf E}_{|X \times \{ x_1 \}})=0
$$
 for all $x_1$.
Hence $\Phi(F)$ is a locally free $\mu$-semi-stable sheaf.
\item[(2)]
Moreover if $(\beta,\omega)$ does not belong any wall for $v$, then
$\Phi(F)$ is $\delta$-twisted semi-stable for all $\delta$.
\end{enumerate}
\begin{NB}
For a $\mu$-semi-stable sheaf $E$ with $v(E)=\Phi(v)$,
we have an exact triangle
$$
A[-1] \to E \to E^{**} \to A,
$$
where $A$ is a 0-dimensional sheaf and $E^{**}$
is a locally free $\mu$-semi-stable sheaf.
Then we have an exact sequence
$$
0 \to \Phi_{X_1 \to X}^{\bf E}(A) \to \widehat{\Phi}(E) \to
\widehat{\Phi}(E^{**}) \to 0
$$
of $\sigma_{(\beta,\omega)}$-semi-stable objects.
Since $w_1$ does not define a wall for $v$,
we get $A=0$. Thus $E$ is locally free.
\end{NB}
\end{cor}

\begin{proof}
(2) Let $E$ be a $\mu$-semi-stable sheaf with $v(E)=\Phi(v)$.
For an exact sequence 
$$
0 \to E_1 \to E \to E_2 \to 0
$$
such that $E_i$ are torsion free sheaves with the same slope,
if $v(E_i) \not \in {\Bbb Q}v(E)$, then
$\widehat{\Phi}(E_i)$ defines a wall for $v$.
Hence $v(E_i) \in {\Bbb Q}v(E)$. Then $E$ is $\delta$-twisted 
semi-stable for any $\delta$.
\end{proof}

We divide the proof of Theorem \ref{thm:mu-semi-stable}
into the proofs of Proposition \ref{prop:mu-semi-stable}
and Proposition \ref{prop:mu-semi-stable2} below. 

\begin{prop}\label{prop:mu-semi-stable}
Assume that $\widehat{L}$ is ample.
Let $E$ be a torsion free sheaf on $X_1$
with
\begin{equation}\label{eq:slope-gamma'}
\frac{(c_1(E)-\rk E \gamma',\widehat{L})}{\rk E}=
\frac{\lambda}{|r_1|}.
\end{equation}
If $E$ is $\mu$-semi-stable with respect to $\widehat{L}$, then
$F:=\widehat{\Phi}(E) 
\in {\frak A}_{(\beta,\omega)}$ and $F$ is 
$\sigma_{(\beta,\omega)}$-semi-stable
with
$\phi(F)=\phi(w_1)$.
\end{prop}

\begin{proof}
Let $E$ be a torsion free sheaf 
on $X_1$ with \eqref{eq:slope-gamma'}
such that $E$ is $\mu$-semi-stable with respect to
$\widehat{L}$.
By Lemma \ref{lem:slope-condition},
\begin{equation}\label{eq:Z(F)}
Z_{(\beta,\omega)}(\widehat{\Phi}(E)) 
\in {\Bbb R}Z_{(\beta,\omega)}(w_1).
\end{equation}
Assume that $F:=\widehat{\Phi}(E)$ 
is not $\sigma_{(\beta,\omega)}$-semi-stable.
By Lemma \ref{lem:range-of-phi} (3),
we have an exact sequence of torsion free sheaves
\begin{equation}
0 \to E_1 \to E \to E_2 \to 0
\end{equation}
such that 
$\Sigma_{(\beta,\omega)}(w_1,F_1)>0$ and 
$\Sigma_{(\beta,\omega)}(F_2,w_1)>0$,
\begin{NB}
$\phi(w_1)+1>\phi(F_1)>\phi(F)=\phi(w_1)>\phi(F_2)>\phi(w_1)-1$,
\end{NB}
where $F_1=\widehat{\Phi}(E_1)$ and $F_2=\widehat{\Phi}(E_2)$.

Applying Lemma \ref{lem:slope-phase} to $w=v(F_2)$ and
$w=v(F)$, we get
\begin{equation}
\begin{split}
\frac{(c_1(E_2(-\gamma')),\widehat{L})}
{\rk E_2} < \frac{\lambda}{|r_1|}= 
\frac{(c_1(E(-\gamma')),\widehat{L})}{\rk E},
\end{split}
\end{equation}
which implies that $E$ is not $\mu$-semi-stable with respect to
$\widehat{L}$.
Therefore $F$ is a $\sigma_{(\beta,\omega)}$-semi-stable 
object of ${\frak A}_{(\beta,\omega)}$.
Then by \eqref{eq:Z(F)},
$\phi(F)=\phi(w_1)$.
\end{proof}

\begin{NB}
Then \eqref{eq:ineq-phi} implies that
\begin{equation}
\begin{split}
-r_1 (c_1(v)-\rk v \gamma,L)
= & r_1 \lambda a_\gamma(v), \\
-r_1 (c_1(F_2)-\rk F_2 \gamma,L)
> & r_1 \lambda a_\gamma(F_2). 
\end{split}
\end{equation}
Hence
\begin{equation}
-r_1 (c_1(F_2)-\rk F_2 \gamma,L)
>  a_\gamma(F_2) \frac{-r_1 (c_1(v)-\rk v \gamma,L)}{a_\gamma(v)}.
\end{equation}
If $r_1>0$, then
\begin{equation}
\frac{r_1^2 (c_1(\Phi(F_2))-\rk \Phi(F_2) \gamma',\widehat{L})}
{\rk \Phi(F_2)}=
\frac{-r_1 (c_1(F_2)-\rk F_2 \gamma,L)}{a_\gamma(F_2)}<
 \frac{-r_1 (c_1(v)-\rk v \gamma,L)}{a_\gamma(v)},
\end{equation}
which implies that $E$ is not $\mu$-semi-stable
with respect to $\widehat{L}$.
If $r_1<0$, then
$c_1(\Phi(F_2))-\rk \Phi(F_2) \gamma'=
-(\widehat{c_1(F_2)}-\rk F_2 \widehat{\gamma})$.
Hence
\begin{equation}
\begin{split}
-\frac{r_1^2 (c_1(\Phi(F_2))-\rk \Phi(F_2) \gamma',\widehat{L})}
{\rk \Phi(F_2)}=&
\frac{-r_1 (c_1(F_2)-\rk F_2 \gamma,L)}{a_\gamma(F_2)}\\
>&
 \frac{-r_1 (c_1(v)-\rk v \gamma,L)}{a_\gamma(v)}\\
=& -\frac{r_1^2 (c_1(E)-\rk E \gamma',\widehat{L})}
{\rk E},
\end{split}
\end{equation}
which implies that $E$ is not $\mu$-semi-stable with
respect to $\widehat{L}$.
\end{NB}

\begin{prop}\label{prop:mu-semi-stable2}
Let $F$ be an object of
${\frak A}_{(\beta,\omega)}$
such that
$\phi(F)=\phi(w_1)$ 
and $F$ is $\sigma_{(\beta,\omega)}$-semi-stable.
Then $\Phi(F)$ fits in an exact triangle
$$
E_1 \to \Phi(F) \to E_2[-1] \to E_1[1],
$$
where $E_1$ is a $\mu$-semi-stable torsion free sheaf
on $X_1$ and $E_2$ is a 0-dimensional sheaf on $X_1$. 
Moreover $E_1$ is $\Phi(F)$ is a $\mu$-stable sheaf, if
$F$ is a $\sigma_{(\beta,\omega)}$-stable object. 
\end{prop}

\begin{proof}
By Lemma \ref{lem:JHF},
we have an exact sequence
\begin{equation}\label{eq:JHF2}
0 \to F_1 \to F \to F_2 \to 0
\end{equation}
in ${\frak A}_{(\beta,\omega)}$, where
$F_1$ is a $\sigma_{(\beta,\omega)}$-semi-stable 
object of ${\frak A}_{(\beta,\omega)}$
such that
$\phi(F_1)=\phi(w_1)$ and 
$\Hom(F_1,{\bf E}_{|X \times \{x_1 \}})=0$
for all $x_1 \in X_1$, and
$F_2$ is $S$-equivalent to 
$\oplus_i {\bf E}_{|X \times \{x_i \}}$, $x_i \in X_1$.
Applying $\Phi_{X \to X_1}^{{\bf E}^{\vee}[1]}$,
we have an exact triangle
$$
\Phi_{X \to X_1}^{{\bf E}^{\vee}[1]}(F_1) \to 
\Phi_{X \to X_1}^{{\bf E}^{\vee}[1]}(F) \to
\Phi_{X \to X_1}^{{\bf E}^{\vee}[2]}(F_2)[-1] \to
\Phi_{X \to X_1}^{{\bf E}^{\vee}[1]}(F_1)[1].
$$
Obviously 
$E_2:=
\Phi_{X \to X_1}^{{\bf E}^{\vee}[2]}(F_2)$ is a 0-dimensional sheaf on $X_1$.

We shall prove that 
$E_1:=\Phi_{X \to X_1}^{{\bf E}^{\vee}[1]}(F_1)$ 
is a $\mu$-semi-stable
sheaf on $X_1$.
Since $\Hom(F_1,{\bf E}_{|X \times \{x_1 \}})=0$ for all $x_1 \in X_1$,
$E_1$ is a torsion free 
sheaf on $X_1$.
Assume that $E_1$ is not $\mu$-semi-stable 
with respect to $\widehat{L}$.
Then we have an exact sequence
\begin{equation}\label{eq:destab1}
0 \to E_1' \to E_1 \to E_1'' \to 0
\end{equation}
such that $E_1'$ and $E_1''$ are torsion free sheaves.
Then we have an exact triangle
$$
\widehat{\Phi}(E_1') \to \widehat{\Phi}(E_1) \to \widehat{\Phi}(E_1'')
\to \widehat{\Phi}(E_1')[1].
$$
By Lemma \ref{lem:range-of-phi},
we get
\begin{equation}
\begin{split}
\phi(w_1)-1<
\phi_{\min}(\widehat{\Phi}(E_1')) \leq \phi_{\max}(\widehat{\Phi}(E_1'))
<\phi(w_1)+1\\
\phi(w_1)-1<
\phi_{\min}(\widehat{\Phi}(E_1'')) \leq \phi_{\max}(\widehat{\Phi}(E_1''))
<\phi(w_1)+1.
\end{split}
\end{equation}
In particular, using $^\beta H^p$ in Definition \ref{defn:BBD},
$^\beta H^p(\widehat{\Phi}(E_1'))
={^\beta H^p}(\widehat{\Phi}(E_1''))=0$ except for
$p=-1,0,1$.
Since $\widehat{\Phi}(E_1)=F_1 \in {\frak A}_{(\beta,\omega)}$, 
we have $^\beta H^{-1}(\widehat{\Phi}(E_1'))=
{^\beta H^1}(\widehat{\Phi}(E_1''))=0$ and 
an exact sequence
$$
0 \to {^\beta H^{-1}}(\widehat{\Phi}(E_1'')) \to
{^\beta H^0}(\widehat{\Phi}(E_1')) \overset{\psi}{\to} F_1 \to
{^\beta H^0}(\widehat{\Phi}(E_1''))\to {^\beta H^1}(\widehat{\Phi}(E_1'))
\to 0
$$
 in ${\frak A}_{(\beta,\omega)}$.
Then $\phi(^\beta H^{-1}(\widehat{\Phi}(E_1'')))<\phi(w_1)$.
By the $\sigma_{(\beta,\omega)}$-semi-stability of $F_1$, 
$\phi(\im \psi) \leq \phi(w_1)$.
Hence 
$$
\phi(^\beta H^0(\widehat{\Phi}(E_1'))) \leq \phi(w_1).
$$
Since $0 \geq \phi(^\beta H^1(\widehat{\Phi}(E_1'))[-1])>\phi(w_1)-1$,
we have
$$
\phi(w_1)-1<\phi(\widehat{\Phi}(E_1')) \leq \phi(w_1).
$$
Assume that $\phi(\widehat{\Phi}(E_1'))<\phi(w_1)$.
By Lemma \ref{lem:slope-phase}, we have
$$
\frac{(c_1(E_1'),\widehat{L})}{\rk E_1'}<
\frac{(c_1(E_1),\widehat{L})}{\rk E_1}.
$$
If $\phi(\widehat{\Phi}(E_1'))=\phi(w_1)$,
then we have $^\beta H^{-1}(\widehat{\Phi}(E_1''))=
{^\beta H^1}(\widehat{\Phi}(E_1'))=0$.
In this case, we have
\begin{equation}\label{eq:properly-E_1}
\frac{(c_1(E_1'),\widehat{L})}{\rk E_1'}=
\frac{(c_1(E_1),\widehat{L})}{\rk E_1}.
\end{equation}
Hence $E_1$ is $\mu$-semi-stable with respect to 
$\widehat{L}$.
Moreover 
if $F$ is $\sigma_{(\beta,\omega)}$-stable, then
$\Phi(F)$ is a $\mu$-stable torsion free sheaf on $X_1$. 
\end{proof}

\section{Relation with Gieseker semi-stability.}

We shall study the relation of Bridgeland stability 
with Gieseker semi-stability
by refining the arguments in the last subsection.
After we prepare some calculations 
on the properly semi-stable objects in \S\,\ref{subsect:proper_ss},
we introduce 
the adjacent chambers ${\cal C}_{\pm}$ to a wall for stabilities 
in \S\,\ref{subsect:C+}.
The wall crossing behavior for these chambers is studied 
in \S\,\ref{subsect:C+} and \S\,\ref{subsect:C-},
and the main theorem (Theorem \ref{thm:isom}) will be obtained. 
In the final \S\,\ref{subsect:proj}, 
we study the assumption in Theorem \ref{thm:isom}
for $K3$ surfaces.

\subsection{Properly semi-stable objects.}\label{subsect:proper_ss}

\begin{lem}\label{lem:wall-general}
Let $v \in H^*(X,{\Bbb Z})_{\alg}$ be a Mukai vector with
$d_\beta(v)>0$. 
Assume that $(\beta,\omega)$ belongs to exactly one wall
$W_{w_1}$ for $v$.
For a subobject $F_1$ of $F$ with $v(F)=v$ and 
$\phi_{(\beta,\omega)}(F_1)=\phi_{(\beta,\omega)}(F)$,
$$
\frac{c_1(F_1(-\gamma))}{a_\gamma(F_1)}=\frac{c_1(F(-\gamma))}{a_\gamma(F)}.
$$
\end{lem}

\begin{proof}
We set
\begin{equation}
\begin{split}
v=& v(F)=r_\gamma(v)e^\gamma+a_\gamma(v) \varrho_X+
(\xi+(\xi,\gamma)\varrho_X),\\
= & e^\gamma(r_\gamma(v)+a_\gamma(v) \varrho_X+\xi)\\
v_1=& v(F_1)=r_\gamma(v_1)e^\gamma+a_\gamma(v_1) \varrho_X+
(\xi_1+(\xi_1,\gamma)\varrho_X)\\
=& e^\gamma(r_\gamma(v_1)+a_\gamma(v_1) \varrho_X+\xi_1).
\end{split}
\end{equation}
Since 
$d_\beta(v_1)v-d_\beta(v)v_1$ and
$d_\beta(w_1)v-d_\beta(v)w_1$ define the same wall,
%$d_\beta(v_1)v-d_\beta(v)v_1$ and
%$d_\beta(w_1)v-d_\beta(v)w_1$
they are linearly dependent.

We note that
\begin{equation}
\begin{split}
& d_\beta(v_1)v-d_\beta(v)v_1\\
=&
e^\gamma[(d_\beta(v_1)r_\gamma(v)-d_\beta(v)r_\gamma(v_1))
+(d_\beta(v_1)a_\gamma(v)-d_\beta(v)a_\gamma(v_1))\varrho_X
+(d_\beta(v_1)\xi-d_\beta(v)\xi_1)]
\end{split}
\end{equation} 
and
\begin{equation}
\begin{split} d_\beta(w_1) v-d_\beta(v)w_1
=
e^\gamma[
(d_\beta(w_1)r_\gamma(v)-d_\beta(v)r_1)
+d_\beta(w_1)a_\gamma(v) \varrho_X+d_\beta(w_1)\xi].
\end{split}
\end{equation} 
Since $d_\beta(w_1)>0$ and $a_\gamma(v) \ne 0$,
we have
\begin{equation}
\begin{split}
0 = & d_\beta(w_1)a_\gamma(v)(d_\beta(v_1)\xi-d_\beta(v)\xi_1)-
(d_\beta(v_1)a_\gamma(v)-d_\beta(v)a_\gamma(v_1))d_\beta(w_1)\xi \\
= & d_\beta(w_1)d_\beta(v)(a_\gamma(v_1)\xi-a_\gamma(v)\xi_1).
\end{split}
\end{equation}
Hence the claim holds.
\end{proof}

\begin{NB}
An old proof:
\begin{proof}
We set
\begin{equation}
\begin{split}
v=& v(F)=r_\gamma(v)e^\gamma+a_\gamma(v) \varrho_X+
(\xi+(\xi,\gamma)\varrho_X),\\
v_1=& v(F_1)=r_\gamma(v_1)e^\gamma+a_\gamma(v_1) \varrho_X+
(\xi_1+(\xi_1,\gamma)\varrho_X).
\end{split}
\end{equation}
Since 
$d_\beta(v_1)v-d_\beta(v)v_1$ and
$d_\beta(w_1)v-d_\beta(v)w_1$ define the same wall,
$d_\beta(v_1)v-d_\beta(v)v_1$ and
$d_\beta(w_1)v-d_\beta(v)w_1$
are linearly dependent.

We note that
\begin{equation}
\begin{split}
& \deg_\beta(v_1)v-\deg_\beta(v)v_1\\
=&
(r_\gamma(v_1)\deg(\gamma-\beta)+\deg_\gamma(v_1))
(r_\gamma(v)e^\gamma+a_\gamma(v)\varrho_X+\xi+(\xi,\gamma)\varrho_X)\\
&-
(r_\gamma(v)\deg(\gamma-\beta)+\deg_\gamma(v))
(r_\gamma(v_1)e^\gamma+a_\gamma(v_1)\varrho_X+\xi_1+(\xi_1,\gamma)\varrho_X)\\
=& (\deg_\gamma(v_1)r_\gamma(v)-\deg_\gamma(v)r_\gamma(v_1))e^\gamma+
(r_\gamma(v_1)a_\gamma(v)-r_\gamma(v)a_\gamma(v_1))
\deg(\gamma-\beta)\varrho_X\\
& +
(\deg_\gamma(v_1)a_\gamma(v)-\deg_\gamma(v)a_\gamma(v_1))\varrho_X\\
& +(r_\gamma(v_1)\xi-r_\gamma(v)\xi_1)\deg(\gamma-\beta)
+(\deg_\gamma(v_1)\xi-\deg_\gamma(v)\xi_1)
\end{split}
\end{equation} 
and
\begin{equation}
\begin{split}
& \deg_\beta(w_1)v-\deg_\beta(v)w_1\\
=&
(r_1\deg(\gamma-\beta)+\deg_\gamma(w_1))
(r_\gamma(v)e^\gamma+a_\gamma(v)\varrho_X+\xi+(\xi,\gamma)\varrho_X)\\
&-
(r_\gamma(v)\deg(\gamma-\beta)+\deg_\gamma(v))
r_1 e^\gamma\\
=& -\deg_\gamma(v)r_1 e^\gamma+
r_1 a_\gamma(v)\deg(\gamma-\beta)\varrho_X
+r_1\xi \deg(\gamma-\beta).
\end{split}
\end{equation} 
Assume that $\deg_\gamma(v) \ne 0$.
Then 
\begin{equation}
\begin{split}
& (\deg_\gamma(v_1)r_\gamma(v)-\deg_\gamma(v)r_\gamma(v_1))
r_1 \xi \deg(\gamma-\beta)\\
&+
r_1 \deg_\gamma(\xi)\left\{
(\deg_\gamma(\xi_1)\xi-\deg_\gamma(\xi)\xi_1)+
\deg(\gamma-\beta)(r_\gamma(v_1)\xi-r_\gamma(v)\xi_1) \right\}
=0,\\
& (\deg_\gamma(v_1)r_\gamma(v)-\deg_\gamma(v)r_\gamma(v_1))
r_1 a_\gamma(v)\deg(\gamma-\beta)\\
& +\deg_\gamma(\xi)r_1
\left\{(r_\gamma(v_1)a_\gamma(v)-r_\gamma(v)a_\gamma(v_1))
\deg(\gamma-\beta)+
(\deg_\gamma(v_1)a_\gamma(v)-\deg_\gamma(v)a_\gamma(v_1))
\right\}=0.
\end{split}
\end{equation}
By the first equality,
we have
\begin{equation}
r_1(\deg_\gamma(\xi)+r_\gamma(v)\deg(\gamma-\beta))
(\deg_\gamma(v_1)\xi-\deg_\gamma(v)\xi_1)=0.
\end{equation}
Since $(\deg_\gamma(\xi)+r_\gamma(v)\deg(\gamma-\beta))=\deg_\beta(v)>0$,
we have
\begin{equation}
\deg_\gamma(v_1)\xi-\deg_\gamma(v)\xi_1=0.
\end{equation}
Then by the second equality,
we have
\begin{equation}
r_1 \deg_\beta(v)(\deg_\gamma(v_1)a_\gamma(v)-\deg_\gamma(v)a_\gamma(v_1))
=0,
\end{equation}
which implies that
\begin{equation}
\deg_\gamma(v_1)a_\gamma(v)-\deg_\gamma(v)a_\gamma(v_1)
=0.
\end{equation}
Therefore 
\begin{equation}
\xi_1 a_\gamma(v)-\xi a_\gamma(v_1)
=0.
\end{equation}
If $\deg_\gamma(v)=0$, then
$\deg_\gamma(v_1)r_\gamma(v)=0$. 
Since $\deg_\beta(v)=r_\gamma(v)\deg(\gamma-\beta)>0$,
$\deg_\gamma(v_1)=0$.
Then we have
\begin{equation}
\begin{split}
& (r_\gamma(v_1)a_\gamma(v)-r_\gamma(v)a_\gamma(v_1))
r_1 \xi \deg(\gamma-\beta)
-r_1 a_\gamma(v)\deg(\gamma-\beta)(r_\gamma(v_1)\xi-r_\gamma(v)\xi_1)\\ 
=&  r_1 r_\gamma(v)\deg(\gamma-\beta)(a_\gamma(v)\xi_1-a_\gamma(v_1)\xi)=0.
\end{split}
\end{equation}
Therefore 
\begin{equation}
\xi_1 a_\gamma(v)-\xi a_\gamma(v_1)
=0.
\end{equation}
\end{proof}
\end{NB}

\begin{cor}\label{cor:wall-general}
For the subobject $F_1$ of $F$ in Lemma \ref{lem:wall-general},
we have 
\begin{equation}
\begin{split}
r_\gamma(F_1)d_\gamma(F)-r_\gamma(F)d_\gamma(F_1)=&
 \left(\rk F_1-\rk F \frac{-r_1 a_\gamma(F_1)}{-r_1 a_\gamma(F)}
\right)d_\gamma(F).
\end{split}
\end{equation}
\end{cor}

\begin{proof}
By Lemma \ref{lem:wall-general},
$$
v(F_1)=\rk F_1 e^\gamma+a_\gamma(F_1)
\left(\varrho_X+\frac{1}{a_\gamma(v)}
(d_\gamma(v)H+D_\gamma(v)+(d_\gamma(v)H+D_\gamma(v),\gamma)\varrho_X)
\right).
$$
\end{proof}

\begin{NB}
For $E_1:=\Phi(F_1)$, we have 
$$
v(E_1)=-r_1 a_\gamma(F_1)e^{\gamma'}-\frac{\rk F_1}{r_1}\varrho_{X_1}
+\frac{r_1}{|r_1|}\frac{a_\gamma(F_1)}{a_\gamma(v)}
(d_\gamma(v)H+D_\gamma(v)+(d_\gamma(v)H+D_\gamma(v),\gamma)\varrho_X).
$$
\end{NB}

\subsection{Gieseker semi-stability and semi-stable objects
in a chamber.}\label{subsect:C+}

From now on,
we assume that $(\beta,\omega)$ belongs to exactly one wall
$W_{w_1}$ for $v$.
Then there are two chambers 
${\cal C}_\pm$ 
which are adjacent to $W_{w_1}$
in a neighborhood of $(\beta,\omega)$:
\begin{equation}
\begin{split}
{\cal C}_+:=& \{(\beta',\omega')|
\phi_{(\beta',\omega')}(v)-\phi_{(\beta',\omega')}(w_1)>0 \},\\
{\cal C}_-:=& \{(\beta',\omega')|
\phi_{(\beta',\omega')}(v)-\phi_{(\beta',\omega')}(w_1)<0 \}.
\end{split}
\end{equation}
We shall study the Bridgeland semi-stability for 
${\cal C}_\pm$.
This subsection is devoted to the study for ${\cal C}_+$.
The case ${\cal C}_-$ is treated in the next subsection.

\begin{lem}\label{lem:C_+:1}
For a $\gamma'$-twisted stable sheaf $E$ of $v(E)=\Phi(v)$
with respect to $\widehat{L}$, we set $F:=\widehat{\Phi}(E)$.
Assume that $d_\gamma(F) \ne 0$. 
We take $\omega_+$ such that $(\beta,\omega_+) \in {\cal C}_+$,
that is,
$\phi_{(\beta,\omega_+)}(w_1)<\phi_{(\beta,\omega_+)}(F)$.
Then $F$ is a $\sigma_{(\beta,\omega_+)}$-semi-stable object of
${\frak A}_{(\beta,\omega)}$. 
\end{lem}

\begin{proof}
We set $\phi_+:=\phi_{(\beta,\omega_+)}$ and $\phi:=\phi_{(\beta,\omega)}$.
Let $F_1$ be a subobject of $F$ with 
$\phi(F_1)=\phi(F)$.
Then $F_1$ and $F_2:=F/F_1$ are $\sigma_{(\beta,\omega)}$-semi-stable
objects of ${\frak A}_{(\beta,\omega)}$ with the phase
$\phi(F)$.
Assume that $F_1$ is a $\sigma_{(\beta,\omega_+)}$-semi-stable 
object with
$\phi_+(F_1)>\phi_+(F)$.
Since $\phi_+(F_1)>\phi_+(w_1)$,
$\Hom(F_1,{\bf E}_{|X \times \{x_1 \}})=0$
for all $x_1 \in X_1$.
Then 
Theorem \ref{thm:mu-semi-stable} implies that
$E_i:=\Phi(F_i)$ $(i=1,2)$ are $\mu$-semi-stable torsion free
sheaves on $X_1$ fitting in an exact sequence
\begin{equation}\label{eq:Gieseker:exact}
0 \to E_1 \to E \to E_2 \to 0.
\end{equation} 
Since $\phi_+(F)>\phi_+(w_1)$ and $\phi(F)=\phi(w_1)$,
$\Sigma_{(\beta,\omega_+)}(w_1,F)>0=\Sigma_{(\beta,\omega)}(w_1,F)$
implies that
\begin{equation}
\begin{split}
-r_1 d_\gamma(F)(\omega_+^2)=&
(r_\beta(F)d_\beta(w_1)-r_\beta(w_1)d_\beta(F))(\omega^2_+)\\
<&
2(a_\beta(F)d_\beta(w_1)-a_\beta(w_1)d_\beta(F))\\
=& (r_\beta(F)d_\beta(w_1)-r_\beta(w_1)d_\beta(F))(\omega^2)
=-r_1 d_\gamma(F)(\omega^2).
\end{split}
\end{equation}
Now we divide the argument into two cases.
\begin{enumerate}
\item[(i)]  
Assume that $r_1 d_\gamma(F)>0$.
We have $(\omega_+^2)>(\omega^2)$.
Then $\phi_+(F_1)>\phi_+(F)$ implies that
\begin{equation}
\begin{split}
&(r_\beta(F_1)d_\beta(F)-r_\beta(F)d_\beta(F_1))(\omega^2_+)\\
<&
2(a_\beta(F_1)d_\beta(F)-a_\beta(F)d_\beta(F_1)).
\end{split}
\end{equation}
Hence
$$
r_\gamma(F_1)d_\gamma(F)-r_\gamma(F)d_\gamma(F_1)=
r_\beta(F_1)d_\beta(F)-r_\beta(F)d_\beta(F_1)<0.
$$
Since $-r_1\chi(E_1(-\gamma'))=\rk F_1$ and
$-r_1\chi(E(-\gamma'))=\rk F$, we get 
\begin{equation}
\begin{split}
0> r_\gamma(F_1)d_\gamma(F)-r_\gamma(F)d_\gamma(F_1)=&
 \left(\rk F_1-\rk F \frac{-r_1 a_\gamma(F_1)}{-r_1 a_\gamma(F)}
\right)d_\gamma(F)\\
=& 
-r_1 \left(\chi(E_1(-\gamma'))-\chi(E(-\gamma')) 
\frac{\rk E_1}{\rk E}
\right)d_\gamma(F)
\end{split}
\end{equation}
by Corollary \ref{cor:wall-general}.
Hence
$$
\frac{\chi(E_1(-\gamma'))}{\rk E_1}>\frac{\chi(E(-\gamma'))}{\rk E},
$$
which is a contradiction.
\item[(ii)]
Assume that 
$r_1 d_\gamma(F)<0$.
Then
\begin{equation}
\begin{split}
0< r_\gamma(F_1)d_\gamma(F)-r_\gamma(F)d_\gamma(F_1)=&
 \left(\rk F_1-\rk F \frac{-r_1 a_\gamma(F_1)}{-r_1 a_\gamma(F)}
\right)d_\gamma(F)\\
=& 
-r_1 \left(\chi(E_1(-\gamma))-\chi(E(-\gamma')) 
\frac{\rk E_1}{\rk E}
\right)d_\gamma(F).
\end{split}
\end{equation}
Hence
$$
\frac{\chi(E_1(-\gamma'))}{\rk E_1}>\frac{\chi(E(-\gamma'))}{\rk E},
$$
which is a contradiction.
\end{enumerate}
Therefore $\phi_+(F_1) \leq \phi_+(F)$.
Thus $F$ is $\sigma_{(\beta,\omega_+)}$-semi-stable.
\end{proof}

In order to treat the case where $d_\gamma(v)=0$,
we need to choose $(\beta',\omega) \in {\cal C}_+$.
We set $\gamma-\beta':=\lambda H+\nu'$, $(\nu,H)=0$. 
We need the following claim.
\begin{lem}\label{lem:d(F)=d(F_1)=0}
If $d_\gamma(F)=d_\gamma(F_1)=0$, then
\begin{equation}
\begin{split}
& a_{\beta'}(F)d_{\beta'}(F_1)-a_{\beta'}(F_1)d_{\beta'}(F) \\
=& 
\lambda \left(
(a_\gamma(F)+(D_\gamma(F),\nu'))\rk F_1-
(a_\gamma(F_1)+(D_\gamma(F_1),\nu'))\rk F \right).
\end{split}
\end{equation}
\end{lem}

\begin{proof}
It is a consequence of Lemma \ref{lem:beta-gamma}.
\end{proof}

\begin{NB}
Then
\begin{equation}
\begin{split}
& a_{\beta'}(F)d_{\beta'}(F_1)-a_{\beta'}(F_1)d_{\beta'}(F) \\
=& \left(a_\gamma(F)+(d_\gamma(F)H+D,\gamma-\beta')
+\frac{\rk F}{2}((\beta'-\gamma)^2)\right)(d_\gamma(F_1)+\rk F_1 \lambda)\\
& -
\left(a_\gamma(F_1)+(d_\gamma(F_1)H+D_1,\gamma-\beta')
+\frac{\rk F_1}{2}((\beta'-\gamma)^2)\right)(d_\gamma(F)+\rk F \lambda)\\
=& \left(a_\gamma(F)(d_\gamma(F_1)+\rk F_1 \lambda)-
a_\gamma(F_1)(d_\gamma(F)+\rk F \lambda) \right)\\
& +\left( 
(d_\gamma(F)H+D,\gamma-\beta')(d_\gamma(F_1)+\rk F_1 \lambda)-
(d_\gamma(F_1)H+D_1,\gamma-\beta')(d_\gamma(F)+\rk F \lambda)\right)\\
& +\frac{((\beta'-\gamma)^2)}{2}
(\rk F d_\gamma(F_1)-\rk F_1 d_\gamma(F)).
\end{split}
\end{equation}
\end{NB}
We take $\beta'$ such that $(D_\gamma(F),\nu')>(D_\gamma(F),\nu)$.
Since $d_\gamma(F)=d_\gamma(w_1)=0$ and 
$\phi_{(\beta,\omega)}(F)=\phi_{(\beta,\omega)}(w_1)$, 
we have
$\phi_{(\beta',\omega)}(F)>\phi_{(\beta',\omega)}(w_1)$.

\begin{lem}\label{lem:C_+:2}
For a $\gamma'$-twisted stable sheaf $E$ of $v(E)=\Phi(v)$
with respect to $\widehat{L}$, we set $F:=\widehat{\Phi}(E)$.
Assume that $d_\gamma(F)=0$ and take $\beta'$ such that
$(\beta',\omega) \in {\cal C}_+$, that is,
$\phi_{(\beta',\omega)}(F)>\phi_{(\beta',\omega)}(w_1)$.
Then $F$ is a $\sigma_{(\beta',\omega)}$-semi-stable object of
${\frak A}_{(\beta',\omega)}$.
\end{lem}

\begin{proof}

Since $\phi_{(\beta',\omega)}(F)>\phi_{(\beta',\omega)}(w_1)$,
\begin{equation}
\begin{split}
0=(r_{\beta'}(F)d_{\beta'}(w_1)-r_{\beta'}(w_1)d_{\beta'}(F))
\frac{(\omega^2)}{2}<
r_1 \lambda (a_\gamma(F)+(D_\gamma(F),\nu')).
\end{split}
\end{equation}
Let $F_1$ be a $\sigma_{(\beta',\omega)}$-semi-stable subobject of
$F$ with $\phi_{(\beta,\omega)}(F_1)=\phi_{(\beta,\omega)}(F)$.
Assume that $\phi_{(\beta',\omega)}(F_1)>\phi_{(\beta',\omega)}(F)$.
Then we have the same exact sequence in \eqref{eq:Gieseker:exact}.
By Lemma \ref{lem:d(F)=d(F_1)=0},
\begin{equation}
\lambda \left(
(a_\gamma(F)+(D_\gamma(F),\nu'))\rk F_1-
(a_\gamma(F_1)+(D_\gamma(F_1),\nu'))\rk F \right)<0.
\end{equation}
By our choice of $(\beta,\omega)$, Lemma \ref{lem:wall-general}
implies that
$D_\gamma(F_1)=a_\gamma(F_1)D_\gamma(F)/a_\gamma(F)$.
Hence 
$$
\lambda a_\gamma(F_1)\left(a_\gamma(F)+(D_\gamma(F),\nu') \right)
\left( \frac{\rk F_1}{a_\gamma(F_1)}-\frac{\rk F}{a_\gamma(F)}
\right)<0.
$$
Since $r_1 a_\gamma(F_1)<0$ and $r_1 \lambda>0$,
we have
$$
\frac{r_1^2 \chi(E_1(-\gamma'))}{\rk E_1}=
\frac{\rk F_1}{a_\gamma(F_1)}>\frac{\rk F}{a_\gamma(F)}
=\frac{r_1^2 \chi(E(-\gamma'))}{\rk E},
$$
which is a contradiction.
Therefore $F$ is $\sigma_{(\beta',\omega)}$-semi-stable.
\end{proof}

The converse relation also holds:
\begin{lem}\label{lem:C_+:3}
We take $(\beta',\omega') \in {\cal C}_+$.
Let $F$ be a $\sigma_{(\beta',\omega')}$-semi-stable object
with $v(F)=v$.
Then $\Phi(F)$ is a $\gamma'$-twisted semi-stable sheaf on $X_1$.
\end{lem}

\begin{proof}
As in Lemma \ref{lem:C_+:1} and Lemma \ref{lem:C_+:2},
we may assume that
\begin{equation}
(\beta',\omega')=
\begin{cases}
(\beta,\omega_+), & d_\gamma(v) \ne 0, \\
(\beta',\omega), & d_\gamma(v)=0.
\end{cases}
\end{equation}
By Theorem \ref{thm:mu-semi-stable},
$E$ is a $\mu$-semi-stable sheaf on $X_1$.
If $E$ is not $\gamma'$-twisted semi-stable, then
there is an exact sequence 
\begin{equation}
0 \to E_1 \to E \to E_2 \to 0
\end{equation}
with  
\begin{equation}
\frac{(c_1(E_1),\widehat{L})}{\rk E_1}=
\frac{(c_1(E),\widehat{L})}{\rk E}.
\end{equation}
Applying Theorem \ref{thm:mu-semi-stable}
to $E_1$ and $E_2$,
we have an exact sequence 
$$
0 \to F_1 \to F \to F_2 \to 0
$$
where $F_i:=\widehat{\Phi}(E_i)$ are 
$\sigma_{(\beta,\omega)}$-semi-stable objects of 
${\frak A}_{(\beta,\omega)}$ with 
$\phi(F_i)=\phi(w_1)$.
Since $\phi_{(\beta',\omega')}(\widehat{\Phi}(E_2)) \geq 
\phi_{(\beta',\omega')}(F)> \phi_{(\beta',\omega')}(w_1)$,
in the same way as in the proof of Lemma \ref{lem:C_+:1}
and Lemma \ref{lem:C_+:2}, we see that
$$
\frac{\chi(E_2(-\gamma'))}{\rk E_2} \geq \frac{\chi(E(-\gamma'))}{\rk E}.
$$
Therefore $E$ is $\gamma'$-twisted semi-stable.
\end{proof}

\subsection{Gieseker semi-stability and semi-stable objects in 
${\cal C}_-$.}\label{subsect:C-}

\begin{NB}
\begin{lem}
Let $E$ be a $\mu$-semi semi-stable
sheaf on $X_1$ with $v(E)=v^{\vee}$ with respect to
$\widehat{L}$.
Then $F:=\Phi_{X_1 \to X}^{{\bf E}[1]}(E^{\vee})$ is an object of
${\frak A}_{(\beta,\omega)}$ such that
$F$ is a $\sigma_{(\beta,\omega)}$-semi-stable object 
of ${\frak A}_{(\beta,\omega)}$ with
$\phi_{(\beta,\omega)}(F)=\phi_{(\beta,\omega)}(w_1)$. 
\end{lem}

\begin{proof}
Let $E$ be a $\mu$-semi semi-stable
sheaf on $X_1$ with $v(E)=v^{\vee}$ with respect to
$\widehat{L}$.
Then $E^{\vee}$ fits in an exact triangle
$$
{\cal H}om_{{\cal O}_X}(E,{\cal O}_X) \to E^{\vee} \to
{\cal E}xt^1_{{\cal O}_X}(E,{\cal O}_X)[-1] \to
{\cal H}om_{{\cal O}_X}(E,{\cal O}_X)[1].
$$
Since ${\cal H}om_{{\cal O}_X}(E,{\cal O}_X)$ is $\mu$-semi-stable,
$F_1:=\Phi_{X_1 \to X}^{{\bf E}[1]}({\cal H}om_{{\cal O}_X}(E,{\cal O}_X))$
is a semi-stable object with 
$\phi_{(\beta,\omega)}(F_1)=\phi_{(\beta,\omega)}(w_1)$
with respect to
$(\beta,\omega)$.
Since ${\cal E}xt^1_{{\cal O}_X}(E,{\cal O}_X)$ is 0-dimensional,
$F_2:=
\Phi_{X_1 \to X}^{{\bf E}[1]}({\cal E}xt^1_{{\cal O}_X}(E,{\cal O}_X)[-1])$
is a semi-stable object with $\phi(F_2)=\phi(w_1)$.
Hence $F:=\Phi_{X_1 \to X}^{{\bf E}[1]}(E^{\vee})$
is a semi-stable object with 
$\phi_{(\beta,\omega)}(F)=\phi_{(\beta,\omega)}(w_1)$.
\end{proof}
\end{NB}

\begin{NB}
Assume that there is an exact triangle
$$
F_1 \to F \to F_2 \to F_1[1]
$$
such that $\phi_{\min}(F_1) \geq \phi(w_1)$
and $\phi_{\max}(F_2)<\phi(E_1)$.
Then
$\Hom(F_1,{\bf E}_{|X \times \{x_1 \}}[k])=0$
for $k<0$ and
$\Hom(F_1,{\bf E}_{|X \times \{x_1 \}})=0$
except finitely many point $x_1 \in X_1$.
We also have
$\Hom(F_2,{\bf E}_{|X \times \{x_1 \}}[k])=0$
for $k \geq 2$.
We note that
\begin{equation}
\begin{split}
\Hom(F,{\bf E}_{|X \times \{x_1 \}}[k])=&
\Hom(\Phi(F),\Phi({\bf E}_{|X \times \{x_1 \}})[k])\\
=& \Hom(E^{\vee},{\frak k}_{x_1}[k-1])\\
=& \Hom(E^{\vee},{\frak k}_{x_1}^{\vee}[k+1])\\
=& \Hom({\frak k}_{x_1},E[k+1]).
\end{split}
\end{equation}
Hence $\Hom(F,{\bf E}_{|X \times \{x_1 \}}[k])=0$ for
$k \ne 0, 1$ and
$\Hom(F,{\bf E}_{|X \times \{x_1 \}})=0$ 
except finitely many points $x_1 \in X_1$.  
Then $E_1:=\Phi(F_1)^{\vee}$ and $E_2:=\Phi(F_2)^{\vee}$
are torsion free sheaves and we have an exact sequence
$$
0 \to E_2 \to E \to E_1 \to 0.
$$
We have
$$
\frac{\deg(E_2(\gamma'))}{\rk E_2}>
\frac{\deg(E(\gamma'))}{\rk E},
$$
which is a contradiction.
Hence $F$ is a semi-stable object with respect to $(\beta,\omega)$.
\end{NB}

Let us study the relation with $(-\gamma')$-twisted semi-stability.
\begin{lem}\label{lem:C_-:1}
For a $(-\gamma')$-twisted semi-stable sheaf $E$ of $v(E)=\Phi(v)^{\vee}$
with respect to $\widehat{L}$, we set $F:=\widehat{\Phi}(E^{\vee})$.
\begin{enumerate}
\item[(1)]
Assume that $d_\gamma(F) \ne 0$. 
We take $\omega_-$ such that $(\beta,\omega_-) \in {\cal C}_-$,
that is,
$\phi_{(\beta,\omega_-)}(w_1)>\phi_{(\beta,\omega_-)}(F)$.
Then $F$ is a $\sigma_{(\beta,\omega_-)}$-semi-stable 
object of ${\frak A}_{(\beta,\omega_-)}$.
\item[(2)]
Assume that $d_\gamma(F)=0$. 
We take $\beta'$ such that $(\beta',\omega) \in {\cal C}_-$,
that is,
$\phi_{(\beta',\omega)}(w_1)>\phi_{(\beta',\omega)}(F)$.
Then $F$ is a $\sigma_{(\beta',\omega)}$-semi-stable 
object of ${\frak A}_{(\beta',\omega)}$.
\end{enumerate}
\end{lem}

\begin{proof}
(1) We set $\phi_-:=\phi_{(\beta,\omega_-)}$ and
$\phi:=\phi_{(\beta,\omega)}$.
%Assume that $\phi_-(w_1)>\phi_-(F)$.
We note that $E^{\vee}$ fits in an exact triangle
$$
{\cal H}om_{{\cal O}_X}(E,{\cal O}_X) \to E^{\vee} \to
{\cal E}xt^1_{{\cal O}_X}(E,{\cal O}_X)[-1] \to
{\cal H}om_{{\cal O}_X}(E,{\cal O}_X)[1],
$$
where ${\cal H}om_{{\cal O}_X}(E,{\cal O}_X)$ is a $\mu$-semi-stable
torsion free sheaf with \eqref{eq:slope-gamma'} and
${\cal E}xt^1_{{\cal O}_X}(E,{\cal O}_X)$ is a 0-dimensional
sheaf.
Applying Theorem \ref{thm:mu-semi-stable},
we get $F$ is a $\sigma_{(\beta,\omega)}$-semi-stable object
of ${\frak A}_{(\beta,\omega)}$
with $\phi(F)=\phi(w_1)$.

Assume that there is an exact sequence
$$
0 \to F_1 \to F \to F_2 \to 0
$$
such that $F_2$ is a semi-stable object with respect to
$(\beta,\omega_-)$ such that
$\phi_-(F_1) >\phi_-(F)>\phi_-(F_2)$
and $\phi(F_1)=\phi(F_2)=\phi(w_1)$.
Then $\Hom(F_2,{\bf E}_{|X \times \{x_1 \}}[k])=0$
for $k \geq 2$.
Since $\Hom(F_2,{\bf E}_{|X \times \{x_1 \}}[k])=0$
for $k<0$ and 
$\Hom(F_2,{\bf E}_{|X \times \{x_1 \}})=0$
except finitely many point $x_1 \in X_1$,
we find that $E_1:=\Phi(F_1)^{\vee}$ and $E_2:=\Phi(F_2)^{\vee}$
are torsion free sheaves and there exists an exact sequence
$$
0 \to E_2 \to E \to E_1 \to 0.
$$
Since $\phi_-(F_2)<\phi_-(F)<\phi_-(w_1)$,
we have
\begin{equation}
\begin{split}
-r_1 d_\gamma(F)(\omega_-^2)=&
(r_\beta(F)d_\beta(w_1)-r_\beta(w_1)d_\beta(F))(\omega^2_-)\\
>&
2(a_\beta(F)d_\beta(w_1)-a_\beta(w_1)d_\beta(F))\\
=& (r_\beta(F)d_\beta(w_1)-r_\beta(w_1)d_\beta(F))(\omega^2)
=-r_1 d_\gamma(F)(\omega^2).
\end{split}
\end{equation}
We divide the rest argument into two cases.
\begin{enumerate}
\item[(i)]
If $r_1 d_\gamma(F)>0$, then we have $(\omega_-^2)<(\omega^2)$.
Then
\begin{equation}
\begin{split}
&(r_\beta(F_2)d_\beta(F)-r_\beta(F)d_\beta(F_2))(\omega^2_-)\\
>&
2(a_\beta(F_2)d_\beta(F)-a_\beta(F)d_\beta(F_2))
\end{split}
\end{equation}
implies 
$$
r_\gamma(F_2)d_\gamma(F)-r_\gamma(F)d_\gamma(F_2)=
r_\beta(F_2)d_\beta(F)-r_\beta(F)d_\beta(F_2)<0.
$$
By Lemma \ref{lem:FM-w_1},
$\chi(E^{\vee}(-\gamma'))=-r_\gamma(F)/r_1$
and $\chi(\Phi(F_2)(-\gamma'))=-r_\gamma(F_2)/r_1$.
Hence
\begin{equation}
\begin{split}
0> & r_\gamma(F_2)d_\gamma(F)-r_\gamma(F)d_\gamma(F_2)\\
=& 
\left( \rk F_2-\rk F \frac{-r_1 a_\gamma(F_2)}{-r_1 a_\gamma(F)}\right)
d_\gamma(F)\\
=& -r_1d_\gamma(F)
\left( \chi(E_2(\gamma'))-\chi(E(\gamma)) 
\frac{-r_1 a_\gamma(F_2)}{-r_1 a_\gamma(F)}\right).
\end{split}
\end{equation} 
Hence
$$
\frac{\chi(E_2(\gamma'))}{\rk E_2}=
\frac{\chi(\Phi(F_2)(-\gamma'))}{\rk \Phi(F_2)}>
\frac{\chi(\Phi(F)(-\gamma'))}{\rk \Phi(F)}
=\frac{\chi(E(\gamma'))}{\rk E},
$$
which is a contradiction.

\item[(ii)]
If $r_1 d_\gamma(F)<0$, then
we have $(\omega_-^2)>(\omega^2)$.
Then
\begin{equation}
\begin{split}
&(r_\beta(F_2)d_\beta(F)-r_\beta(F)d_\beta(F_2))(\omega^2_-)\\
>&
2(a_\beta(F_2)d_\beta(F)-a_\beta(F)d_\beta(F_2))
\end{split}
\end{equation}
implies
$$
r_\gamma(F_2)d_\gamma(F)-r_\gamma(F)d_\gamma(F_2)=
r_\beta(F_2)d_\beta(F)-r_\beta(F)d_\beta(F_2)>0.
$$
By Lemma \ref{lem:FM-w_1}, we have
\begin{equation}
\begin{split}
0<& r_\gamma(F_2)d_\gamma(F)-r_\gamma(F)d_\gamma(F_2)\\
=& 
\left( \rk F_2-\rk F \frac{-r_1 a_\gamma(F_2)}{-r_1 a_\gamma(F)}\right)
d_\gamma(F)\\
=& -r_1d_\gamma(F)
\left( \chi(E_2(\gamma'))-\chi(E(\gamma)) 
\frac{-r_1 a_\gamma(F_2)}{-r_1 a_\gamma(F)}\right).
\end{split}
\end{equation} 
Hence
$$
\frac{\chi(E_2(\gamma'))}{\rk E_2}=
\frac{\chi(\Phi(F_2)(-\gamma'))}{\rk \Phi(F_2)}>
\frac{\chi(\Phi(F)(-\gamma'))}{\rk \Phi(F)}
=\frac{\chi(E(\gamma'))}{\rk E},
$$
which is a contradiction.
\end{enumerate}
Therefore $F$ is a $\sigma_{(\beta,\omega_-)}$-semi-stable 
object of ${\frak A}_{(\beta,\omega_-)}$.

(2)
Assume that $d_\gamma(F)=0$.
For $(\beta',\omega) \in {\cal C}_-$,
we have 
$$
0=(r_{\beta'}(F)d_{\beta'}(w_1)-r_{\beta'}(w_1)d_{\beta'}(F))
\frac{(\omega^2)}{2}
>r_1 \lambda \left(a_\gamma(F)+(D_\gamma(F),\nu') \right).
$$
Since $\phi_{(\beta',\omega)}(F_2)<\phi_{(\beta',\omega)}(F)$,
we have
$$
\lambda \left((a_\gamma(F)+(D_\gamma(F),\nu'))\rk F_2
-\left(a_\gamma(F_2)+(D_\gamma(F_2),\nu') \right)\rk F \right)>0.
$$
By the choice of $(\beta',\omega)$,
$D_\gamma(F_2)=a_\gamma(F_2)D_\gamma(F)/a_\gamma(F)$.
Hence 
$$
\lambda a_\gamma(F_2)(a_\gamma(F)+(D_\gamma(F),\nu'))
\left(\frac{\rk F_2}{a_\gamma(F_2)}-\frac{\rk F}{a_\gamma(F)} \right)>0.
$$
Therefore
$$
\frac{\rk F_2}{a_\gamma(F_2)}>\frac{\rk F}{a_\gamma(F)}.
$$
\end{proof}

The converse relation also holds:
\begin{lem}\label{lem:C_-:2}
We take $(\beta',\omega') \in {\cal C}_-$.
Let $F$ be a $\sigma_{(\beta',\omega')}$-semi-stable object
of ${\frak A}_{(\beta,\omega)}$ with $v(F)=v$.
Then $\Phi(F)^{\vee}$ is a $(-\gamma')$-twisted semi-stable
sheaf on $X_1$.
\end{lem}

\begin{proof}
We may take the same $(\beta',\omega')$ in Lemma \ref{lem:C_-:1}. 
By Theorem \ref{thm:mu-semi-stable},
$E:=\Phi(F)^{\vee}$ is a $\mu$-semi-stable torsion free sheaf on $X_1$.
Assume that $E$ is not $(-\gamma')$-twisted semi-stable.
Then there is an exact sequence
\begin{equation}
0 \to E_1 \to E \to E_2 \to 0
\end{equation}
such that $E_i$, $i=1,2$ are torsion free sheaves with
\begin{equation}
\frac{(c_1(E_i),\widehat{L})}{\rk E_i}=\frac{(c_1(E),\widehat{L})}{\rk E}.
\end{equation}
Applying Theorem \ref{thm:mu-semi-stable} to
$E_i^{\vee}$,
we have an exact sequence
\begin{equation}
0 \to F_2 \to F \to F_1 \to 0,
\end{equation}
where $F_i:=\widehat{\Phi}(E_i^{\vee})$ are 
$\sigma_{(\beta,\omega)}$-semi-stable objects of
${\frak A}_{(\beta,\omega)}$ with $\phi_{(\beta,\omega)}(F_i)=
\Phi_{(\beta,\omega)}(w_1)$.
By using the computations
in the proof of Lemma \ref{lem:C_-:1},
it is easy to see that $F$ is $\sigma_{(\beta',\omega')}$-semi-stable
 if and only if 
$E$ is $(-\gamma')$-twisted semi-stable.
\end{proof}

Summarizing our argument, we have
\begin{thm}\label{thm:isom}
Let $X$ be an abelian surface, or a $K3$ surface with
$\NS(X)={\Bbb Z}H$.
Assume that there is a smooth projective surface 
$X_1$ which is the moduli space $M_{(\beta,\omega)}(w_1)$.
Let $v$ be a Mukai vector with $d_\beta(v)>0$.  
Assume that $(\beta,\omega)$ satisfies
\begin{equation}
\langle d_\beta(v)w_1-d_\beta(w_1)v,e^{\beta+\sqrt{-1}\omega} \rangle=0.
\end{equation}
If $(\beta,\omega)$ does not belong to any wall for $v$, 
or $w_1$ defines a wall $W_{w_1}$ for $v$ and $(\beta,\omega)$ 
belongs to exactly one wall $W_{w_1}$,
then 
${\cal M}_{(\beta',\omega')}(v) \cong {\cal M}_{\widehat{L}}(u)^{ss}$,
where $u=\Phi(v)$ for $(\beta',\omega') \in {\cal C}_+$
and $u=\Phi(v)^{\vee}$ for $(\beta',\omega') \in {\cal C}_-$.
In particular, there is a coarse moduli scheme
$M_{(\beta',\omega')}(v)$ and 
$M_{(\beta',\omega')}(v) \cong \overline{M}_{\widehat{L}}(u)$. 
\end{thm}

\begin{rem}
Let $E$ be an $\alpha$-twisted sheaf on $X_1$ such that
$v(E)=w$ is a primitive Mukai vector. 
Assume that ${\bf F}$ is a family of 
$(-\alpha)$-twisted sheaves consisting of
$(-\gamma')$-twisted stable sheaves
of Mukai vector $r_1 e^{-\gamma'}$
with respect to $\widehat{L}$.
Then for a sufficiently large $n$,  
$E$ is a stable $\alpha$-twisted sheaf on $X_1$
if and only if  
$\Phi_{X_1 \to X}^{{\bf F}}(E(n \widehat{L}))$ is a stable
sheaf on $X$.  
In particular, ${\cal M}_{\widehat{L}}^{\gamma'}(w)^{ss}$ 
is isomorphic to
${\cal M}_{L}^{\gamma}(u)^{ss}$, 
where $u:=v(\Phi_{X_1 \to X}^{{\bf E}}(E(n \widehat{L})))$.
Moreover ${\cal M}_{L}^{\gamma}(u)^{ss}$ consists of $\mu$-stable
sheaves.
Then the second cohomology group
of $M_{(\beta',\omega')}(v)$ is described by $v^{\perp}$ and 
the albanese map $M_{(\beta',\omega')}(v) \to X \times \widehat{X}$.
\end{rem}

\begin{NB}
\subsection{Converse direction.}

Let $F$ be an object of
${\frak A}_{(\beta,\omega)}$
such that
$\phi_{(\beta,\omega)}(F)=\phi_{(\beta,\omega)}(w_1)$ 
and $F$ is semi-stable with respect to
$(\beta,\omega)$.
Assume that
$\Hom(F,{\bf E}_{|X \times \{x_1 \}})=0$
for all $x_1 \in X_1$.
Then $E:=\Phi_{X \to X_1}^{{\bf E}^{\vee}[1]}(F)$
is a torsion free sheaf on $X_1$.
Assume that $E$ is not $\mu$-semi-stable 
with respect to $\widehat{L}$.
Then we have an exact sequence
\begin{equation}\label{eq:destab1}
0 \to E_1 \to E \to E_2 \to 0
\end{equation}
such that $E_1$ and $E_2$ are torsion free sheaves.
Then we have an exact triangle
$$
\widehat{\Phi}(E_1) \to \widehat{\Phi}(E) \to \widehat{\Phi}(E_1)
\to \widehat{\Phi}(E_1)[1].
$$
\begin{defn}
For a complex $F \in {\bf D}(X)$,
let $^\beta H^p(F) \in {\frak A}_{(\beta,\omega)}$
denotes the $p$-th cohomology group of $F$
with respect to the $t$-structure associated to
${\frak A}_{(\beta,\omega)}$. 
\end{defn}
By Lemma \ref{lem:range-of-phi},
we get
\begin{equation}
\begin{split}
\phi(w_1)-1<
\phi_{\min}(\widehat{\Phi}(E_1)) \leq \phi_{\max}(\widehat{\Phi}(E_1))
<\phi(w_1)+1\\
\phi(w_1)-1<
\phi_{\min}(\widehat{\Phi}(E_2)) \leq \phi_{\max}(\widehat{\Phi}(E_2))
<\phi(w_1)+1.
\end{split}
\end{equation}
In particular, $^\beta H^p(\widehat{\Phi}(E_i))=0$, $i=1,2$ except for
$p=-1,0,1$.
Since $\widehat{\Phi}(E)=F \in {\frak A}_{(\beta,\omega)}$, 
we have $^\beta H^{-1}(\widehat{\Phi}(E_1))=
{^\beta H^1}(\widehat{\Phi}(E_2))=0$ and 
an exact sequence in ${\frak A}_{(\beta,\omega)}$:
$$
0 \to {^\beta H^{-1}}(\widehat{\Phi}(E_2)) \to
{^\beta H^0}(\widehat{\Phi}(E_1)) \overset{\psi}{\to} F \to
{^\beta H^0}(\widehat{\Phi}(E_2))\to {^\beta H^1}(\widehat{\Phi}(E_1))
\to 0.
$$
Then $\phi(^\beta H^{-1}(\widehat{\Phi}(E_2)))<\phi(w_1)$.
By the semi-stability of $F$, 
$\phi(\im \psi) \leq \phi(w_1)$.
Hence $\phi(^\beta H^0(\widehat{\Phi}(E_1))) \leq \phi(w_1)$.
Since $0 \geq \phi(^\beta H^1(\widehat{\Phi}(E_1))[-1])>\phi(w_1)-1$,
$\phi(w_1)-1<\phi(\widehat{\Phi}(E_1)) \leq \phi(w_1)$.
Assume that $\phi(\widehat{\Phi}(E_1))<\phi(w_1)$.
By Lemma \ref{lem:slope-phase}, we have
$$
\frac{(c_1(E_1),\widehat{L})}{\rk E_1}<
\frac{(c_1(E),\widehat{L})}{\rk E}.
$$
If $\phi(\widehat{\Phi}(E_1))=\phi(w_1)$,
then we have $^\beta H^{-1}(\widehat{\Phi}(E_2))=
{^\beta H^1}(\widehat{\Phi}(E_1))=0$.
ccIn this case, we have
\begin{equation}\label{eq:properly-E_1}
\frac{(c_1(E_1),\widehat{L})}{\rk E_1}=
\frac{(c_1(E),\widehat{L})}{\rk E}.
\end{equation}
Hence $E$ is $\mu$-semi-stable with respect to 
$\widehat{L}$, and
$E$ is $\mu$-stable, if $F$ is stable with respect to $(\beta,\omega)$. 
Therefore we get the following result.
\begin{lem}
Let $F$ be an object of
${\frak A}_{(\beta,\omega)}$
such that
$\phi_{(\beta,\omega)}(F)=\phi_{(\beta,\omega)}(w_1)$ 
and $F$ is semi-stable with respect to
$(\beta,\omega)$.
Assume that
$\Hom(F,{\bf E}_{|X \times \{x_1 \}})=0$
for all $x_1 \in X_1$.
Then $E:=\Phi(F)$ is a $\mu$-semi-stable torsion free sheaf on $X_1$.
Moreover if $F$ is stable, then $E$ is $\mu$-stable.
\end{lem}

If $E$ is not $\gamma'$-twisted semi-stable, then
there is an exact sequence \eqref{eq:destab1} with  
\eqref{eq:properly-E_1}.
Then $\widehat{\Phi}(E_1)$ is a subobject of 
$F$ with $\phi(\widehat{\Phi}(E_1))=\phi(w_1)$.
Since $\phi_{(\beta',\omega')}(\widehat{\Phi}(E_2)) \geq 
\phi_{(\beta',\omega')}(F)> \phi_{(\beta',\omega')}(w_1)$,
we see that
$$
\frac{\chi(E_2(-\gamma'))}{\rk E_2} \geq \frac{\chi(E(-\gamma'))}{\rk E}.
$$
Therefore $E$ is semi-stable.

 Let $F$ be a semi-stable object of
${\frak A}_{(\beta,\omega)}$ such that
$\Hom({\bf E}_{|X \times \{x_1 \}},F)=0$
for all $x_1 \in X_1$.
Then we have an exact sequence in ${\frak A}_{(\beta,\omega)}$
$$
0 \to F_1 \to F \to F_2 \to 0
$$
such that $\Hom(F_1,{\bf E}_{|X \times \{x_1 \}})=0$
for all $x_1 \in X_1$ and
$F_2$ is $S$-equivalent to
$\oplus_i {\bf E}_{|X \times \{x_i \}}$, $x_i \in X_1$.
Then $\Phi_{X \to X_1}^{{\bf E}^{\vee}[1]}(F_1)$ is a torsion free
$\mu$-semi-stable sheaf with respect to $\widehat{L}$,
$\Phi_{X \to X_1}^{{\bf E}^{\vee}[2]}(F_2)$ is a 0-dimensional sheaf on $X_1$
and we have an exact triangle
$$
\Phi_{X \to X_1}^{{\bf E}^{\vee}[1]}(F_1) \to 
\Phi_{X \to X_1}^{{\bf E}^{\vee}[1]}(F) \to
\Phi_{X \to X_1}^{{\bf E}^{\vee}[2]}(F_2)[-1] \to
\Phi_{X \to X_1}^{{\bf E}^{\vee}[1]}(F_1)[1].
$$
Since $\Hom({\bf E}_{|X \times \{x_1 \}},F)=0$ for all $x_1 \in X_1$,
$E:=\Phi_{X \to X_1}^{{\bf E}^{\vee}[1]}(F)^{\vee}$ is a torsion free 
sheaf.
Assume that there is an exact sequence of torsion free sheaves
$$
0 \to E_1 \to E
\to E_2 \to 0
$$
such that 
\begin{equation}\label{eq:properly-E_1:2}
\frac{(c_1(E_1),\widehat{L})}{\rk E_1}=
\frac{(c_1(E),\widehat{L})}{\rk E}.
\end{equation}
Since 
$E_i^*:={\cal H}om(E_i,{\cal O}_{X_1})$ are $\mu$-semi-stable sheaves
with 
\begin{equation}\label{eq:properly-E_1:2}
\frac{(c_1(E_i^*)-\gamma' \rk E_i^*,\widehat{L})}{\rk E_i^*}=
\frac{(c_1(E^{\vee})-\gamma' \rk E^{\vee},\widehat{L})}{\rk E^{\vee}}
=\frac{\lambda}{|r_1|},
\end{equation}
Proposition \ref{prop:mu-semi-stable} implies that
$\Phi_{X_1 \to X}^{{\bf E}[1]}({\cal H}om(E_i,{\cal O}_{X_1}))$
are semi-stable objects of ${\frak A}_{(\beta,\omega)}$. 

By the exact triangle
$$
{\cal H}om(E_i,{\cal O}_{X_1}) \to E_i^{\vee} \to 
{\cal E}xt^1(E_i,{\cal O}_{X_1})[-1] \to
{\cal H}om(E_i,{\cal O}_{X_1})[1],
$$
we have an exact sequence
$$
0 \to \Phi_{X_1 \to X}^{{\bf E}[1]}({\cal H}om(E_i,{\cal O}_{X_1}))
 \to \Phi_{X_1 \to X}^{{\bf E}[1]}(E_i^{\vee}) \to 
\Phi_{X_1 \to X}^{{\bf E}}({\cal E}xt^1(E_i,{\cal O}_{X_1})) \to
0.
$$
Thus we have an exact sequence
$$
0 \to \Phi_{X_1 \to X}^{{\bf E}[1]}(E_2^{\vee}) \to
F \to \Phi_{X_1 \to X}^{{\bf E}[1]}(E_1^{\vee})
\to 0
$$
of semi-stable objects with the same phase $\phi(w_1)$.
It is easy to see that $F$ is semi-stable with respect to
$(\beta',\omega')$ if and only if 
$E$ is $-\gamma'$-twisted semi-stable.

\end{NB}

\subsection{The projectivity of $M_{(\beta,\omega)}(w_1)$ for a $K3$ surface.}
\label{subsect:proj}

\begin{NB}
For a Mukai vector $v$,
there is $\lambda \in {\Bbb R}$ such that
$\phi_{(\beta,\omega)}(e^{\beta+\lambda H})=\phi_{(\beta,\omega)}(v)$.
Indeed 
\begin{equation}
\frac{d_\beta(e^{\beta+\lambda H}) a_\beta-
d_\beta a_\beta(e^{\beta+\lambda H})}
{d_\beta(e^{\beta+\lambda H}) r-d_\beta r_\beta(e^{\beta+\lambda})}=
\frac{\lambda\left(a_\beta-d_\beta \frac{(H^2)\lambda}{2}\right)}
{\lambda r-d_\beta}.
\end{equation}
Then we can find $\lambda \in {\Bbb R}$ satisfying 
$$
\frac{\lambda\left(a_\beta-d_\beta \frac{(H^2)\lambda}{2}\right)}
{\lambda r-d_\beta}=\frac{(\omega^2)}{2}.
$$
By the density of ${\Bbb Q}$ in ${\Bbb R}$,
we can find $\omega'$ which is sufficiently close to $\omega$
and
$\phi_{(\beta,\omega')}(e^{\beta+\lambda H})=
\phi_{(\beta,\omega')}(v)$, $\lambda \in {\Bbb Q}$.
\end{NB}
Fix a primitive isotropic Mukai
vector $w_1:=r_1 e^\gamma$ with $d_\beta(w_1)>0$.
In this subsection, let us study the assumption on $M_{(\beta,\omega)}(w_1)$
in Theorem \ref{thm:isom}.

\begin{lem}\label{lem:classify-(-2)}
Let $X$ be a $K3$ surface with $\Pic(X)={\Bbb Z}H$.
\begin{enumerate}
\item[(1)]
There is at most one Mukai vector $v$ such that
$\langle v^2 \rangle=-2$ and
$\phi_{(\beta,\omega)}(v)=\phi_{(\beta,\omega)}(w_1)$.
\item[(2)]
There is a unique primitive isotropic Mukai vector $v$ such that
$\phi_{(\beta,\omega)}(v)=\phi_{(\beta,\omega)}(w_1)$ and
$\langle v,w_1 \rangle \ne 0$.
\end{enumerate}
\end{lem}

\begin{proof}
We set $u_0:=w_1$. 
We take a primitive Mukai vector $u_1 \in H^*(X,{\Bbb Z})_{\alg}$
such that 
\begin{equation}
u_0^{\perp} \cap H^*(X,{\Bbb Z})_{\alg}={\Bbb Z}u_0
\oplus {\Bbb Z}u_1
\subset {\Bbb Q}e^\gamma \oplus {\Bbb Q}(H+(H,\gamma)\varrho_X).
\end{equation}
We next take $u_2 \in H^*(X,{\Bbb Z})_{\alg}$ such that
\begin{equation}
-\langle u_2,w_1 \rangle=\min \{-\langle u,w_1 \rangle>0|
u \in H^*(X,{\Bbb Z})_{\alg} \}.
\end{equation}
Since $H^*(X,{\Bbb Z})_{\alg}$ is generated by
$u_0^{\perp} \cap H^*(X,{\Bbb Z})_{\alg}$ and $u_2$,
$$
H^*(X,{\Bbb Z})_{\alg}={\Bbb Z}u_0
\oplus {\Bbb Z}u_1 \oplus {\Bbb Z}u_2.
$$
(1)
If $v=\sum_{i=0}^2 x_i u_i$, $x_i \in {\Bbb Z}$ is a $(-2)$-vector, then
we get 
$$
-2=\langle (x_1 u_1+x_2 u_2)^2 \rangle+2\langle x_1 u_1+x_2 u_2,x_0 u_0
\rangle,
$$
which implies that
$\gcd(x_1,x_2)=1$.
On the other hand,
if $\phi_{(\beta,\omega)}(v)=\phi_{(\beta,\omega)}(u_0)$, then
Lemma \ref{lem:Sigma-L} implies that
$-r_1(c_1(v(-\gamma)),L)+\lambda \langle v,u_0 \rangle=0$.
Hence 
$$
\lambda x_2 \langle u_2,u_0 \rangle=r_1(x_1(c_1(u_1(-\gamma)),L)
+x_2 (c_1(u_2(-\gamma)),L)),
$$ 
which implies that
$x_1/x_2$ is determined.
Thus $v$ is unique up to sign.
Since $d_\beta(v)>0$, $v$ is unique.

The proof of (2) is similar.
\end{proof}

\begin{NB}
For a given rational number $\lambda$,
there is an isotropic Mukai vector $w_1$ such that
$\phi_{(\beta,\omega)}(w_1)=\phi_{(\beta,\omega)}(e^{\beta+\lambda H})$
and $(\beta,\omega)$ is not contained in a wall for $w_1$.
If $X$ is an abelian surface, then
there is no wall for $w_1$.
Hence we assume that $X$ is a $K3$ surface.
\end{NB}

\begin{prop}\label{prop:assumption-K3}
Let $X$ be a $K3$ surface with $\Pic(X)={\Bbb Z}H$.
For $\gamma \in {\Bbb Q}H$, 
let $w_1:=r_1 e^\gamma$ be the primitive isotropic Mukai vector
with $d_\beta(w_1)>0$.  
Let $v$ be the unique primitive isotropic Mukai vector such that
$\phi_{(\beta,\omega)}(v)=\phi_{(\beta,\omega)}(w_1)$ and
$\langle v,w_1 \rangle \ne 0$.
\begin{enumerate}
\item[(1)]
Assume that ${\cal M}_{(\beta,\omega)}(w_1)$ contains a properly
$\sigma_{(\beta,\omega)}$-semi-stable object.
\begin{enumerate}
\item[(i)]
There is a $\sigma_{(\beta,\omega)}$-stable object
$E_1$ such that $\langle v(E_1)^2 \rangle=-2$ and
$\phi_{(\beta,\omega)}(v(E_1))=\phi_{(\beta,\omega)}(w_1)$.
\item[(ii)]
$\langle w_1,v(E_1) \rangle<0$ and
$v=w_1+\langle w_1,v(E_1) \rangle v(E_1)$.
\item[(iii)]
${\cal M}_{(\beta,\omega)}(v)$
consists of stable objects. 
\end{enumerate}
\item[(2)]
If there is no $(-2)$-vector with the phase
$\phi_{(\beta,\omega)}(w_1)$, then
${\cal M}_{(\beta,\omega)}(w_1)$ and ${\cal M}_{(\beta,\omega)}(v)$
consist of stable objects. 
\end{enumerate}
\end{prop}

\begin{proof}
(1)
If $(\beta,\omega)$ belongs to a wall, then
we take an object $E$ of 
${\cal M}_{(\beta,\omega)}(w_1)$
such that $E$ is not $\sigma_{(\beta,\omega')}$-semi-stable, 
where $\omega'$ is sufficiently close to
$\omega$.
We take the Harder-Narasimhan filtration of $E$ with 
respect to $(\beta,\omega')$:
$$
0 \subset F_1 \subset F_2 \subset \cdots \subset F_s=E.
$$
Then $\phi_{(\beta,\omega)}(F_i/F_{i-1})=\phi_{(\beta,\omega)}(E)$
and $\langle v(F_i/F_{i-1})^2 \rangle \leq 0$.
By Lemma \ref{lem:classify-(-2)},
there are primitive Mukai vectors
$v_1, v_2 \in H^*(X,{\Bbb Z})_{\alg}$
such that 
$$
\phi_{(\beta,\omega)}(v_i)=
\phi_{(\beta,\omega)}(w_1),\quad
\langle v_1^2 \rangle=0,\quad
\langle v_2^2 \rangle=-2.
$$
Since $v(F_i/F_{i-1}) \in {\Bbb Z}v_1$ or $v(F_i/F_{i-1}) \in {\Bbb Z}v_2$,
we have
$$
s=2,\quad
\{ v(F_1),v(F_2/F_1) \}=\{n_1 v_1,n_2 v_2 \},
\quad n_1,n_2 >0.
$$
Then $0=\langle (n_1 v_1+n_2 v_2)^2 \rangle=
2n_2(n_1 \langle v_1,v_2 \rangle-n_2)$.
Hence $n_2=n_1 \langle v_1,v_2 \rangle$.
Since $w_1=v(E)=n_1 v_1+n_2 v_2$ is primitive,
$n_1=1$ and $w_1=v_1+\langle v_1,v_2 \rangle v_2$.
Since $n_2=n_1 \langle v_1,v_2 \rangle$,
$\langle w_1,v_1 \rangle=-\langle v_2,v_1 \rangle <0$.
Let $E_1$ be the $\sigma_{(\beta,\omega)}$-stable object
with $v(E_1)=v_2$. Then (i) holds.
Since $v=v_1$, (ii) holds.
Since $\langle v,v(E_1) \rangle>0$, the assumption of (1) 
does not hold for 
${\cal M}_{(\beta,\omega)}(v)$. Thus (iii) holds.   

By the proof of (1), (2) also follows.
\end{proof}

The following is a special case of \cite[Thm. 6.12]{Ha:1}, 
if the moduli space is fine.
 
\begin{prop}\label{prop:isotropic-K3}
Let $X$ be a $K3$ surface with $\Pic(X)={\Bbb Z}H$.
Let $v$ be a primitive isotropic Mukai vector with
$d_\beta(v)>0$.
If $(\beta,\omega)$ does not belong to any wall for $v$,
then there is a coarse moduli space $M_{(\beta,\omega)}(v)$
which is a $K3$ surface. 
Moreover there is a Fourier-Mukai transform $\Psi:{\bf D}(X) \to {\bf D}(X)$
inducing an isomorphism ${\cal M}_{(\beta,\omega)}(v) \to
{\cal M}_H(\Psi(v))^{ss}$.  
\end{prop}

\begin{proof}
By \cite[Prop. 1.6.10]{MYY},
we may assume that ${\frak A}_{(\beta,\omega)}
={\frak A}_{(\beta,tH)}$, $t >1 $.
If $(\omega^2) \gg 0$, then the claim is obvious
by \cite[Cor. 2.2.9]{MYY}.
Hence it is sufficient to show that 
the claims are preserved 
under the wall-crossing. 
Assume that $(\beta,\omega)$ belongs to a wall $W$ and
the claim holds for $(\beta,\omega_+)$ with $(\omega_+^2)>(\omega^2)$.
In a neighborhood of $\omega$, we take $\omega_-$ with
$(\omega_-^2)<(\omega^2)$.
By Proposition \ref{prop:assumption-K3} (1) (i),
there is a $\sigma_{(\beta,\omega)}$-stable object
$E_1$ with $\phi_{(\beta,\omega)}(E_1)=\phi_{(\beta,\omega)}(v)$
and $\langle v(E_1)^2 \rangle=-2$.
We note that $E_1$ satisfies
$\Hom(E_1,E_1)={\frak k}$ and $\Hom(E_1,E_1[p])=0$ $(p \ne 0,2)$.
Then we have an autoequivalence
$\Phi_{E_1}:{\bf D}(X) \to {\bf D}(X)$ (\cite[sect. 1]{MYY}).
For $E \in {\cal M}_{(\beta,\omega_+)}(v)$,
by Proposition \ref{prop:assumption-K3} (1) (ii),
we have
$\langle v,v(E_1) \rangle<0$.

We first assume that 
$\phi_{(\beta,\omega_+)}(E_1)<\phi_{(\beta,\omega_+)}(v)$.
Then $\Hom(E,E_1)=0$.
By the proof of Proposition \ref{prop:assumption-K3},
we have an exact sequence
$$
0 \to E_1^{\oplus n} \to E \to F \to 0
$$
where $n= -\langle v,v(E_1) \rangle$ and $F$ is a 
$\sigma_{(\beta,\omega)}$-stable object with
isotropic Mukai vector.
%Since $-2 \leq \langle v(F)^2 \rangle=-2n(n+\langle v,v(E_1) \rangle)$,
%$n=-\langle v,v(E_1) \rangle$.
Then we see that $\Ext^1(E_1,E)=0$ and
$\Phi_{E_1}(E)=F$.
Applying $\Phi_{E_1}$ again, we have an exact sequence
$$
0 \to F \to \Phi_{E_1}(F) \to \Ext^1(E_1,F) \otimes E_1 \to 0.
$$  
Then we see that $E':=\Phi_{E_1}(F)$ is a $\sigma_{(\beta,\omega_-)}$-
semi-stable object.
Conversely for a $\sigma_{(\beta,\omega_-)}$-semi-stable
object $E'$, we have an exact sequence
$$
0 \to F \to E' \to E_1^{\oplus n} \to 0
$$
where $n= -\langle v,v(E_1) \rangle$ and $F$ is a 
$\sigma_{(\beta,\omega)}$-stable object.
Then we see that $n=-\langle v,v(E_1) \rangle$
and $\Phi_{E_1}^{-1}(E')=F$.
Moreover $\Phi_{E_1}^{-1}(F) \in {\cal M}_{(\beta,\omega_+)}(v)$.
Therefore we have a sequence of isomorphisms
$$
{\cal M}_{(\beta,\omega_+)}(v) \overset{\Phi_{E_1}}{\to}
{\cal M}_{(\beta,\omega)}(u) \overset{\Phi_{E_1}}{\to}  
{\cal M}_{(\beta,\omega_-)}(v),
$$
where $u=v+\langle v,v(E_1) \rangle v(E_1)$.
If 
$\phi_{(\beta,\omega_+)}(E_1)>\phi_{(\beta,\omega_+)}(v)$,
then by a similar argument,
we have 
a sequence of isomorphisms
$$
{\cal M}_{(\beta,\omega_+)}(v) \overset{\Phi_{E_1}^{-1}}{\to}
{\cal M}_{(\beta,\omega)}(u) \overset{\Phi_{E_1}^{-1}}{\to}  
{\cal M}_{(\beta,\omega_-)}(v).
$$
Therefore the claims hold for $(\beta,\omega_-)$. 
\end{proof}

\begin{NB}
Assume that $\widehat{L}$ is general with respect to $v$.
If ${\cal M}_{\widehat{L}}(v)^{ss}$ does not contain
locally free sheaves, then
$$
v=
\begin{cases}
(1,0,-l)e^D,\\
(l,0,-1)e^D,\\
(\rk v_0)v_0-\varrho_{X_1},
\end{cases}
$$
where $l \in {\Bbb Z}_{>0}$, $D \in \NS(X_1)$ and $\langle v_0^2 \rangle=-2$.
Moreover ${\cal M}_{\widehat{L}}(v)^{ss}$ contains
$\mu$-stable locally free sheaves,
unless $v=lv_0-a \varrho_X$, $a \rk v_0<2l$.
Hence $M_{\widehat{L}}(v)$ is birationally
equivalent to
$M_{\widehat{L}}(v^{\vee})$, unless $v=lv_0-a \varrho_X$, $a \rk v_0<2l$.

Assume that 
$v=lv_0-a \varrho_X$, $a \rk v_0<2l$.
Then $\langle v^2 \rangle=2l(-l+\rk v_0 a) \geq -2$.
Hence $l':=-l+\rk v_0 a > 0$ or $\langle v^2 \rangle=0,-2$.
Assume that $l'>0$.
Then $v_0=r_0 e^{\xi}+\frac{1}{r_0}\varrho_{X_1}$ and
$v=l r_0 e^{\xi}-b \varrho_{X_1}$,
where $r_0=\rk v_0$ and $\xi=c_1(v_0)/r_0$.
By \cite[Thm. 2.3]{Y:twist1},
we have an isomorphism
${\cal M}_{\widehat{L}}(v)^{ss} \cong 
{\cal M}_{\widehat{L}}(l' v_0^{\vee}-a \varrho_{X_1})^{ss}$.
Since $2l'<a \rk v_0$, 
${\cal M}_{\widehat{L}}(l' v_0^{\vee}-a \varrho_{X_1})^{ss}$
contains a $\mu$-stable locally free sheaf. 
Thus we have a birational map
$M_{\widehat{L}}(l' v_0^{\vee}-a \varrho_{X_1}) \cdots \to 
M_{\widehat{L}}(l' v_0-a \varrho_{X_1})$.
Combining the isomorphism
$
{\cal M}_{\widehat{L}}(l' v_0-a \varrho_{X_1})^{ss} \cong
{\cal M}_{\widehat{L}}(l v_0^{\vee}-a \varrho_{X_1})^{ss},
$
we have a birational map
$M_{\widehat{L}}(v) \cdots \to M_{\widehat{L}}(v^{\vee})$.
Therefore we get the following corollary
\begin{cor}
Assume that $(\beta,\omega)$ belongs to exactly 
one wall $W_{w_1}$. Then
the birational type of 
$M_{(\beta',\omega')}(v)$ does not depend on a general choice of
$(\beta',\omega')$.
\end{cor}
\end{NB}

\section{Applications}\label{sect:application}

\subsection{The projectivity of some moduli spaces.}
\label{subsect:projectivity}

In this subsection, we assume that $X$ is an abelian surface
or a $K3$ surface with $\Pic(X)={\Bbb Z}H$.
Let
$$
v:=r e^\beta+a_\beta \varrho_X+(d_\beta H+D)+(d_\beta H+D,\beta)\varrho_X,
\; D \in \NS(X)_{\Bbb Q} \cap H^{\perp}
$$
be a Mukai vector with $d_\beta(v)>0$ and
$\langle v^2 \rangle>0$.
Then $d_\beta>\frac{2a_\beta}{(H^2)d_\beta}r$.
Assume that 
$$
w_1:=r_1 e^{\beta+\frac{d_1}{r_1} H}=
r_1 e^\beta+d_1 (H+(H,\beta)\varrho_X)+a_1 \varrho_X.
$$
Then $a_1=\frac{d_1^2(H^2)}{2 r_1}$.
Hence 
\begin{equation}
\frac{d_1 a_\beta-d_\beta a_1}{d_1 r-d_\beta r_1}=
\frac{\frac{d_1}{r_1}\left(a_\beta-d_\beta \frac{(H^2)d_1}{2r_1}\right)}
{\frac{d_1}{r_1}r-d_\beta}.
\end{equation}

We set $f(x):=\frac{x\left(a_\beta-d_\beta \frac{(H^2)}{2}x \right)}
{xr-d_\beta}, x \in {\Bbb R}$.
Then $f(x)$ defines a bijection
$f:D \to {\Bbb R}_{>0}$, where
\begin{equation}
D:=
\begin{cases}
(x_0,\frac{d_\beta}{r}),& r > 0\\
(x_0,\infty),& r  \leq 0,
\end{cases}
\end{equation}
$x_0:=\max\{\frac{2a_\beta}{(H^2)d_\beta},0\}$. 
%and $d_\beta/r:=\infty$ for $r=0$.
For $x \in D$, 
we take a unique
element $\omega_x \in {\Bbb R}_{>0}H$ such that
$\frac{(\omega_x^2)}{2}=f(x)$.
We define an injective map
\begin{equation}
\begin{matrix}
\iota_\beta:& {\Bbb R}_{>0}H & \to & {\frak H}_{\Bbb R}\\
& \omega & \mapsto & (\eta,\omega),
\end{matrix}
\end{equation} 
where $\beta=bH+\eta, \eta \in H^\perp \cap \NS(X)_{\Bbb Q}$.
Let $I \subset {\Bbb R}_{>0}H$ be the pull-back of a chamber
in ${\frak H}_{\Bbb R}$ by $\iota_\beta$.
We set 
$$
J:=\{x \in {\Bbb R}|\omega_x \in I \}.
$$
We take a rational number $\lambda \in J$. 
%such that
%$f(\frac{p}{q})$ is sufficiently close to $\frac{(\omega^2)}{2}$.
Then $\phi_{(\beta,\omega_\lambda)}(w_1)=\phi_{(\beta,\omega_\lambda)}(v)$
and $\omega_\lambda$ belongs to the same chamber as that of $\omega$,
where $w_1=r_1 e^{\beta+\lambda H}$ is a primitive Mukai vector
with $r_1>0$.
Hence ${\cal M}_{(\beta,\omega)}(v)^{ss}=
{\cal M}_{(\beta,\omega_\lambda)}(v)^{ss}$.
Replacing $w_1$ if necessary,
we have a primitive isotropic Mukai vector
$w_1$ such that $\phi_{(\beta,\omega_\lambda)}(w_1)=
\phi_{(\beta,\omega_\lambda)}(v)$ and 
$X_1:=M_{(\beta,\omega_\lambda)}(w_1)$
is a smooth projective surface
(Proposition \ref{prop:assumption-K3}). 
Let $\Phi$ be the Fourier-Mukai transform
in \S \ref{sect:mu-semi-stable}.
Applying Corollary \ref{cor:mu-semi-stable}, we have an isomorphism
${\cal M}_{(\beta,\omega_\lambda)}(v)
\to {\cal M}_{\widehat{H}}(u)^{ss}$,
where $u=\Phi(v)$.
\begin{NB}
Old version:
Let $E_1$ be a subobject of $E \in {\cal M}_{(\beta,\omega_{p/q})}(v)$
with
$\phi_{(\beta,\omega_{p/q})}(E_1)=\phi_{(\beta,\omega_{p/q})}(E)$.
Then $E_1$ and $E/E_1$ are semi-stable objects with 
$\phi_{(\beta,\omega_{p/q})}(E_1)=\phi_{(\beta,\omega_{p/q})}(E/E_1)$.
Since $\omega_{p/q}$ is not contained in a wall,
we have $\chi(E_1(nH))=\rk E_1 \frac{\chi(E(nH))}{\rk E}$,
$\rk E \ne 0$
for all $n$ or
$\chi(E_1(nH))=(c_1(E_1),H) \frac{\chi(E(nH))}{(c_1(E),H)}$,
$\rk E=0$ for all $n$.
\begin{NB2}
Assume that $\iota_\beta({\Bbb R}_{>0}) \subset {\frak H}_{\Bbb R}$
is not contained in a wall, where $\beta=bH+\eta$, 
$\eta \in H^{\perp}$.
Then we have $v(E_1) \in {\Bbb Q}v(E)$.
\end{NB2}
We take a small $\delta \in {\Bbb Q}$ such that
$\omega_{p/q+\delta} \in I$ and 
$\phi_{(\beta,\omega_{p/q+\delta})}(w_1)<
\phi_{(\beta,\omega_{p/q+\delta})}(v)$. 
Applying Theorem \ref{thm:isom}, we have an isomorphism
${\cal M}_{(\beta,\omega_{p/q+\delta})}(v)
\to {\cal M}_{\widehat{H}}(w)^{ss}$,
where $w=\Phi(v)$.
\end{NB}

\begin{NB}
Assume that $X$ is defined over a field $k$.
Assume that there is a $k$-rational point $P$ of $X$
and $H$ is defined over $k$.
Then $X_1:=M_H(w_1)$ is defined over $k$ and there is a universal family
${\bf E}$ on $X \times X_1$ as a twisted sheaf:
Let $Q$ be an open subscheme of a quot-scheme such that 
$X_1$ is a GIT-quotient of $Q$ by a linear group $G$. 
Let ${\cal Q}$ be the universal quotient parametrized $Q$.
Then $W:=p_*({\cal Q} \otimes F)$ is a $G$-linearized vector bundle on $Q$,
and $P/G:={\Bbb P}(W)/G$ is a projective bundle over $X_1$.
Then ${\cal Q}(-\lambda)$ descend to $X \times P/G$.
\end{NB}

\begin{NB}
The following is replaced by Corollary \ref{cor:mu-semi-stable}.
\begin{rem}
Under the assumption, 
all $\mu$-semi-stable sheaves $F$ with $v(F)=w$ are locally free and
are contained in 
${\cal M}_{\widehat{H}}^{\widehat{\gamma}}(w)^{ss}$:

Let $F_1$ be a subsheaf of $F$ such that $F_2:=F/F_1$ is torsion free
sheaf with $\frac{(c_1(F_2),\widehat{H})}{\rk F_2}=
\frac{(c_1(F),\widehat{H})}{\rk F}$.
Then 
$Z_{(\beta,\omega)}(\widehat{\Phi}(F_1)), 
Z_{(\beta,\omega)}(\widehat{\Phi}(F_2)) 
\in {\Bbb R}Z_{(\beta,\omega)}(w_1)$.
Since 
\begin{equation}
\begin{split}
d_\gamma(v(\widehat{\Phi}(F_i)))(H^2)=&
\langle v(\widehat{\Phi}(F_i)),H+(H,\gamma)\varrho_X \rangle \\
=& \langle v(F_i),\widehat{H}+(H,\widehat{\gamma})\varrho_{X_1} \rangle,
\end{split}
\end{equation}
the signatures of 
$\langle v(\widehat{\Phi}(F_i)),H+(H,\gamma)\varrho_X \rangle$, $i=1,2$
are the same. Since $d_\gamma(v)=d_\beta-r \frac{d_1}{r_1}>0$,
we get $d_\gamma(v(\widehat{\Phi}(F_i)))>0$ for $i=1,2$.
Thus 
$$
Z_{(\beta,\omega)}(\widehat{\Phi}(F_1)), 
Z_{(\beta,\omega)}(\widehat{\Phi}(F_2)) 
\in {\Bbb R}_{>0} Z_{(\beta,\omega)}(v)=
{\Bbb R}_{>0} Z_{(\beta,\omega)}(w_1).
$$
Then $E_1:=\widehat{\Phi}(F_1)$ and $E_2:=\widehat{\Phi}(F_2)$ are
semi-stable objects with respect to $(\beta,\omega)$
and we get an exact sequence
\begin{equation}
0 \to E_1 \to \widehat{\Phi}(F) \to E_2 \to 0.
\end{equation}
Since $(\beta,\omega)$ does not belong to any wall,
we have $v(E_1), v(E_2) \in {\Bbb Q}v$.
Then we have $v(F_1), v(F_2) \in {\Bbb Q}w$.
Therefore $F$ is $\widehat{\gamma}$-twisted semi-stable.
\end{rem}
\end{NB}

\begin{thm}\label{thm:projective}
Let $X$ be an abelian surface or a $K3$ surface with $\Pic(X)={\Bbb Z}H$.
Assume that $(\beta,\omega)$ is general.
There is a coarse moduli scheme $M_{(\beta,\omega)}(v)$
which is isomorphic to the projective scheme 
$\overline{M}_{\widehat{H}}(u)$, where $u=\Phi(v)$.
\end{thm}
In particular, the moduli spaces in \cite{AB} are projective, 
if $\omega$ is general. 
\begin{rem}
Maciocia and Meachan showed the claim for 
$v=1+2H+n\varrho_X$, where $X$ is an abelian surface with
$\NS(X)={\Bbb Z}H$ in
\cite[Thm. 3.1]{MM}.
It is easy to see that
their proof also works for any $v$ and get the same result
for abelian surfaces. 
\end{rem}

\subsection{The dependence of walls on $\beta$.}
\label{subsect:dependence}

We shall study the structure of walls for stabilities
under the deformation of
$\beta$. In this subsection, we assume that $X$ is an abelian surface.
Let us start with the following lemma.
\begin{lem}\label{lem:positivity-of-d}
Assume that non-zero vectors
$$
v_i:=r_i e^\beta+a_i \varrho_X+d_i H+D_i+
(d_i H+D_i,\beta)\varrho_X,\quad D_i \in H^{\perp} \cap \NS(X)_{\Bbb Q}
\quad (i=1,2)
$$
satisfy (1) $\langle v_i^2 \rangle \geq 0$
and
(2) $Z_{(\beta,\omega)}(v_1)$
and $Z_{(\beta,\omega)}(v_2)$ are linearly dependent over ${\Bbb R}$.
Then $d_1 d_2 \langle v_1,v_2 \rangle> 0$ or $d_1=d_2=0$.
\end{lem}

\begin{proof}
Since
$Z_{(\beta,\omega)}(v_1)$ and
$Z_{(\beta,\omega)}(v_2)$ are linearly dependent, we have
$$
(d_1 r_2-d_2 r_1)\frac{(\omega^2)}{2}=(d_1 a_2-d_2 a_1).
$$
By \cite[Lem. 3.1]{MYY},
we have
\begin{equation}
\begin{split}
\langle v_1,v_2 \rangle (d_1 d_2)=& 
-\frac{1}{2}((d_1 D_2-d_2 D_1)^2)+
\frac{d_2^2 \langle v_1^2 \rangle}{2}+
\frac{d_1^2 \langle v_2^2 \rangle}{2}+
(d_1 r_2-d_2 r_1)(d_1 a_2-d_2 a_1)\\
=&
-\frac{1}{2}((d_1 D_2-d_2 D_1)^2)+
\frac{d_2^2 \langle v_1^2 \rangle}{2}+
\frac{d_1^2 \langle v_2^2 \rangle}{2}+
(d_1 r_2-d_2 r_1)^2 \frac{(\omega^2)}{2} \geq 0. 
\end{split}
\end{equation}
If the equality holds, then
$d_1 r_2-d_2 r_1=d_1 a_2-d_2 a_1=
d_1 D_2-d_2 D_1=0$.
Thus $d_1 v_2=d_2 v_1$.
If $d_1 \ne 0$ or $d_2 \ne 0$, then
$v_1=0$ or $v_2=0$, which is a contradiction.
Therefore the claim holds.
\begin{NB}
Old argument:
Assume that $d_1=0$ and $d_2 \ne 0$.
Then $Z_{(\beta,\omega)}(v_2) \not \in {\Bbb R}$ and
(2) imply that $Z_{(\beta,\omega)}(v_1)=0$.
Thus we have $a_1=r_1 \frac{(\omega^2)}{2}$. Then 
$\langle v_1^2 \rangle=-2r_1^2 \frac{(\omega^2)}{2}+(D_1^2) \leq 0$.
By (1), we get $r_1=0$ and $D_1=0$, which implies that $v_1=0$.
Therefore the claim holds.
\end{NB}
\end{proof}

\begin{NB}
If $r_1 d_2-r_2 d_1 \ne 0$, then $d_1 d_2 \ne 0$.
\end{NB}

The following characterization of the walls for stabilities
is a consequence of the Bogomolov inequality.
\begin{prop}
Assume that $\langle v^2 \rangle>0$.
For a Mukai vector
$v_1$, we set
$v_2:=v-v_1$.
Then $v_1 \not \in {\Bbb Q}v$ defines a wall in ${\frak H}_{\Bbb R}$, 
if and only if
(1) $\langle v_1^2 \rangle, \langle v_2^2 \rangle \geq 0$
and (2) $\langle v_1,v_2 \rangle>0$.
%Moreover if $(r_1,d_1,a_1) \not \in {\Bbb Q}(r,d,a)$,
%then $v_1$ defines a wall in $\iota_\beta({\Bbb R}_{>0} H)$.
\end{prop}

\begin{proof}
We write
%$$
%v_i=r_i e^\beta+a_i \varrho_X+d_i H+D_i
%+(d_i H+D_i,\beta)\varrho_X, D_i \in H^{\perp} \cap \NS(X)_{\Bbb Q}.
%$$
\begin{align*}
v&=r e^\beta+a \varrho_X+d H+D
   +(d H+D,\beta)\varrho_X,
\quad D \in H^{\perp} \cap \NS(X)_{\Bbb Q},
\\
v_i&=r_i e^\beta+a_i \varrho_X+d_i H+D_i+
      (d_i H+D_i,\beta)\varrho_X,
\quad D_i \in H^{\perp} \cap \NS(X)_{\Bbb Q}
\quad (i=1,2).
\end{align*}

(I) Assume that $v_1$ defines a wall,
that is, there are $\sigma_{(\beta,\omega)}$-semi-stable
objects $E_i$ $(i=1,2)$ of ${\frak A}_{(\beta,\omega)}$
with $v(E_i)=v_i$,
and $Z_{(\beta,\omega)}(E_i)$ $(i=1,2)$ 
are linearly dependent over ${\Bbb R}$. 
Then $d_1,d_2 \geq 0$. 
By Lemma \ref{lem:positivity-of-d},
we have (i) $d_1 d_2 \langle v_1,v_2 \rangle>0$
or (ii) $d_1=d_2=0$.
In the first case, $d \geq 0$ implies that
$d_1,d_2>0$. Hence $\langle v_1,v_2 \rangle>0$. 
So it is enough to consider the second case.
In this case,
$d_1=d_2=0$ implies that
\begin{equation}\label{eq:r}
r_1 r_2 \langle v_1,v_2 \rangle=
\frac{r_2^2 \langle v_1^2 \rangle}{2}
+\frac{r_1^2 \langle v_2^2 \rangle}{2}-
\frac{1}{2}((r_2 D_1-r_1 D_2)^2) 
\geq 0.
\end{equation}
If the equality holds, then
$r_2 D_1-r_1 D_2=0$.
Since $Z_{(\beta,\omega)}(v_1), Z_{(\beta,\omega)}(v_2) \in
{\Bbb R}_{<0}$
by \eqref{eq:stab_func}, the condition $d_1=d_2=0$ 
and the definition of semi-stable object, we have
$a_i-r_i \frac{(\omega^2)}{2}>0$
$(i=1,2)$.
By $0 \leq \langle v_i^2 \rangle=-2r_i a_i+(D_i^2) \leq -2r_i a_i $,
$-r_i \geq 0$ and $a_i \geq 0$ $(i=1,2$).
If $r_1=r_2=0$, then
$r=0$ and $\langle v^2 \rangle \leq 0$.
Therefore 
 $r_1 \ne 0$ or $r_2 \ne 0$.
If $r_1 \ne 0$ and $r_2 =0$,
then
$v_2=a_2 \varrho_X \ne 0$ and
$\langle v_1,v_2 \rangle=-r_1 a_2>0$.
If $r_1, r_2 \ne 0$ and
$r_1 r_2 \langle v_1,v_2 \rangle=0$, then
$\langle v_1^2 \rangle=\langle v_2^2 \rangle=0$
and $D_1/r_1=D_2/r_2$.
Hence $v_i=r_i e^{D_i/r_i}$ $(i=1,2)$, which implies
$v$ is isotropic. 
Therefore if $r_1, r_2 \ne 0$, then
$r_1 r_2 \langle v_1,v_2 \rangle>0$.
Hence $r_1,r_2 <0$ and $\langle v_1,v_2 \rangle>0$.

(II)
Conversely assume that 
$\langle v_1^2 \rangle, \langle v_2^2 \rangle \geq 0$ and
$\langle v_1,v_2 \rangle>0$.
For $(\beta,\omega)$ with $Z_{(\beta,\omega)}(v) 
\in {\Bbb H} \cup {\Bbb R}_{<0}$, assume that
$Z_{(\beta,\omega)}(v_i)$ $(i=1,2)$ 
are linearly dependent over ${\Bbb R}$.
We shall show that $Z_{(\beta,\omega)}(v_1),
Z_{(\beta,\omega)}(v_2) \in {\Bbb H} \cup {\Bbb R}_{<0}$.
Then there are $\sigma_{(\beta,\omega)}$-semi-stable
objects $E_i$ $(i=1,2)$ of ${\frak A}_{(\beta,\omega)}$
with $v(E_i)=v_i$. 
This means that $v_1$ defines a wall for $v$.

We first assume that $d> 0$.
In this case, Lemma \ref{lem:positivity-of-d} implies that
$d_1,d_2>0$.
Hence we get $Z_{(\beta,\omega)}(v_1),
Z_{(\beta,\omega)}(v_2) \in {\Bbb H} \cup {\Bbb R}_{<0}$.

If $d=0$, then we have
$0< \langle v^2 \rangle=-2ra+(D^2) \leq -2ra$
and 
$-Z_{(\beta,\omega)}(v)=-r\frac{(\omega^2)}{2}+a>0$.
Hence $-r,a>0$.
By Lemma \ref{lem:positivity-of-d}, $d=d_1+d_2$ and 
$\langle v_1,v_2 \rangle>0$, we have
$d_1=d_2=0$.
Then by \eqref{eq:r}, we have
$r_1 r_2 \langle v_1,v_2 \rangle>0$ 
or $r_1 r_2=0$.
For the first case $r_1 r_2 \langle v_1,v_2 \rangle>0$, 
our assumption implies that
$r_1 r_2>0$.
Since $-r_1 a_1 \geq 0$ and 
$-r_2 a_2 \geq 0$,
$-a_1/r_1,-a_2/r_2 \geq 0$.
Since 
\begin{equation}
\begin{split}
0< & -r\frac{(\omega^2)}{2}+a \\
=& -r_1 \left(\frac{(\omega^2)}{2}-\frac{a_1}{r_1} \right)
-r_2 \left(\frac{(\omega^2)}{2}-\frac{a_2}{r_2} \right),
\end{split}
\end{equation}
we have $-r_1,-r_2>0$.
Therefore $Z_{(\beta,\omega)}(v_1),Z_{(\beta,\omega)}(v_2)
\in {\Bbb R}_{<0}$.
For the second case $r_1 r_2=0$, 
we may assume that $r_1=r<0$ and $r_2=0$.
Then we see that $v_2=a_2 \varrho_X$.
Since $\langle v_1,v_2 \rangle=-r_1 a_2>0$ and $r_1<0$,
$a_2>0$.
Since $0 \leq \langle v_1^2 \rangle \leq -2r_1 a_1$,
$a_1 \geq 0$.
Therefore $Z_{(\beta,\omega)}(v_1),Z_{(\beta,\omega)}(v_2)
\in {\Bbb R}_{<0}$.
\end{proof}

\begin{NB}
Old remark:
\begin{rem}
For 
$v_1=r_1 e^\beta+(d_1 H+D_1)+(d_1 H+D_1,\beta)\varrho_X
+a_1 \varrho_X$, we set
$v_2:=v-v_1$.
Assume that (1) $\langle v_1^2 \rangle, \langle v_2^2 \rangle \geq 0$
and (2) $0<d_1<d$.
If $r d_1-r_1 d=0$, then for a general $\beta$,
we have $a d_1-a_1 d \ne 0$.
 
\end{rem}
\end{NB}

We set $\beta:=\beta_0+sH$, $s \leq d_{\beta_0}/r$.
Then $d(s):=d_\beta(v)$ and $d_i(s):=d_\beta(v_i)$ 
$(i=1,2)$ are function of $s$.
We note that the conditions in Lemma \ref{lem:positivity-of-d}  
are independent of $s$.
By Lemma \ref{lem:positivity-of-d},
we have the following. 
\begin{lem} 
Assume that $\langle v_1^2 \rangle, \langle v_2^2 \rangle \geq 0$.
We take $(s,t)$ in 
%$C:=\{(s,t)|\Sigma_{(\beta,tH)}(v,v_1)=0, t>0, d(s) \geq 0 \}$.
\begin{align*}
C:=\{(s,t) \mid \Sigma_{(\beta,tH)}(v,v_1)=0, \  t>0, \  d(s) \geq 0 \}.
\end{align*}
Then the following conditions are equivalent:
\begin{enumerate}
\item[(1)]
$0<d_1(s)<d(s)$ at a point $(s,t) \in C$.
\item[(2)]
$0<d_1(s)<d(s)$ for all $(s,t) \in C$ with $d(s)>0$.
\item[(3)]
$\langle v_1,v_2 \rangle>0$.
\end{enumerate}
\end{lem}

\begin{rem}
Keep the notation as above.
We note that $r_1 d_2(s)-r_2 d_1(s)$ does not depend on the 
choice of $s$.
Assume that the constant $r_1 d_2(s)-r_2 d_1(s) \ne 0$ 
and $\langle v_1,v_2 \rangle>0$.
Then $d_1(s) d_2(s) \langle v_1,v_2 \rangle>0$ and $d(s) \ne 0$
for all $(s,t) \in C$.

Indeed
if $d_1(s)=d_2(s)=0$, then $r_1 d_2(s)-r_2 d_1(s)=0$.
By Lemma \ref{lem:positivity-of-d},
we have $d_1(s) d_2(s) \langle v_1,v_2 \rangle>0$.
Since $\langle v_1,v_2 \rangle>0$,
$d_1(s) d_2(s)>0$, which implies that $d(s) \ne 0$.
\end{rem}

For an isotropic Mukai vector
$$
w_1=r_1 e^{\beta_0+xH}=r_1 e^{\beta+(x-s)H},
$$
we see that
\begin{equation}
\begin{split}
& \frac{(x-s)
\left(a_{\beta_0+sH}-d_{\beta_0+sH}(x-s)\frac{(H^2)}{2}\right)}
{(x-s)r-(d_{\beta_0}-rs)}\\
=& \frac{(x-s)\left(a_{\beta_0}-d_{\beta_0}s(H^2)+rs^2 \frac{(H^2)}{2}
-(d_{\beta_0}-rs)(x-s)\frac{(H^2)}{2}\right)}
{xr-d_{\beta_0}}\\
=& \frac{(x-s)\left(a_{\beta_0}-d_{\beta_0}x \frac{(H^2)}{2}+s
(rx -d_{\beta_0})\frac{(H^2)}{2}\right)}{xr-d_{\beta_0}}.
\end{split}
\end{equation}

Hence we define $\omega_{s,x}$ by 
\begin{equation}
\frac{(\omega_{s,x}^2)}{2}=
\frac{(x-s)\left(a_{\beta_0}-d_{\beta_0}x \frac{(H^2)}{2}+s
(rx -d_{\beta_0})\frac{(H^2)}{2}\right)}{xr-d_{\beta_0}}.
\end{equation}

\begin{cor}
Assume that $\omega_{s,x}$ is general with respect to $v$.
Then
${\cal M}_{(\beta,\omega_{s,x})}(v)$ does not depend on the choice of 
$s$.
\end{cor}

\begin{proof}
For the Mukai vector $w_1$, we set $X_1:=M_H(w_1)$ and
let ${\bf E}$ be a universal family on $X \times X_1$.
The isomorphism ${\cal M}_{(\beta,\omega_{s,x})}(v) \to
{\cal M}_{\widehat{H}}^{\gamma'}(w)^{ss}$
is defined by the Fourier-Mukai transform
$\Phi_{X \to X_1}^{{\bf E}[1]}$, 
which is independent of the choice of $s$.
Since ${\cal M}_{\widehat{H}}^{\gamma'}(w)^{ss}$
is independent of the choice of $s$,
we get the claim.
\end{proof}

\subsection{Ample line bundles on $M_{(\beta,\omega)}(v)$.}
We fix $\beta$ in this subsection.
We set 
\begin{equation}
\varphi_\omega:=
\frac{r\frac{(\omega^2)}{2}-a_\beta}{d_\beta}=
\frac{\mathrm{Re} Z_{(\beta,\omega)}(v)}{\mathrm{Im} Z_{(\beta,\omega)}(v)}
(H,\omega).
\end{equation}
Then 
$$
\{\varphi_\omega| \omega \in {\Bbb R}_{>0} H \}
=
\left(-\frac{a_\beta}{d_\beta},\infty \right).
$$
We set
\begin{equation}
\begin{split}
\xi_\omega:=& \frac{(\omega^2)}{2 d_\beta}(r(H+(H,\beta)\varrho_X)
+d_\beta(H^2)\varrho_X)-
\frac{1}{d_\beta}(a_\beta(H+(H,\beta)\varrho_X)+d_\beta(H^2)e^\beta) \\
=& \varphi_\omega 
\left(H+\left(H,\frac{c_1(v)}{r} \right)\varrho_X \right)
-(H^2) \left(e^\beta-\frac{a_\beta}{r} \varrho_X \right), \;(r \ne 0).
\end{split}
\end{equation}

For $\omega=\omega_\lambda$, $\lambda \in {\Bbb Q}$,
we set $\gamma:=\beta+\lambda H$. 
Let $w_\lambda=r_1 e^{\beta+\lambda H}$ be a primitive isotropic Mukai vector
with $r_1 \lambda>0$.
We set $X_1:=M_H^\beta(w_\lambda)$.
\begin{lem}\label{lem:FM(xi)}
By the Fourier-Mukai transform $\Phi:{\bf D}(X) \to 
{\bf D}^{\alpha}(X_1)$, we have
\begin{equation}
\Phi(\xi_\omega)=\frac{1}{|r_1|(d_\beta-\lambda r)}
(\rk w \widehat{H}+(\widehat{H},c_1(w))\varrho_{X_1}),
\end{equation}
where $w=\Phi(v)$.
We also have
\begin{equation}
-\frac{\rk w}{r}\widehat{\Phi}
\left(e^{\gamma'}+
\frac{\langle e^{\gamma'},w \rangle}{\rk w} \varrho_{X_1}\right)
=\left(e^\beta-\frac{a_\beta}{r}\varrho_X \right)+
\lambda \left(H+\left(H,\frac{c_1(v)}{r}\right)\varrho_X \right). 
\end{equation}

\end{lem}

\begin{proof}
We note that
\begin{equation}
\begin{split}
e^\gamma-\frac{a_\gamma}{r}\varrho_X 
=& e^\beta+\lambda(H+(H,\beta)\varrho_X)+
\frac{(H^2)}{2}\lambda^2 \varrho_X-\frac{a_\gamma}{r}\varrho_X\\
=&
e^\beta-\frac{a_\beta}{r}\varrho_X+
\lambda \left(H+(H,\beta)\varrho_X+
\frac{d_\beta}{r}(H^2)\varrho_X \right)\\
=& \left(e^\beta-\frac{a_\beta}{r}\varrho_X \right)+
\lambda \left(H+\left(H,\frac{c_1(v)}{r}\right)\varrho_X \right) 
\end{split}
\end{equation}
and
\begin{equation}
\varphi_{\omega_\lambda}=\frac{a_\beta-r \lambda^2 \frac{(H^2)}{2}}
{r\lambda-d_\beta}
=\frac{a_\gamma+d_\gamma \lambda (H^2)}{-d_\gamma}.
\end{equation}
Then we see that 
\begin{equation}
\begin{split}
\xi_{\omega_\lambda}=& 
\frac{a_\gamma+d_\gamma \lambda (H^2)}{-d_\gamma}
\left(H+\left(H,\frac{c_1(v)}{r}\right)\varrho_X \right)-
\left(e^\gamma-\frac{a_\gamma}{r}\varrho_X \right)(H^2)
+(H^2)\lambda \left(H+\left(H,\frac{c_1(v)}{r} \right)\varrho_X \right)\\
=& \frac{a_\gamma}{-d_\gamma}
\left(H+\left(H,\frac{c_1(v)}{r}\right)\varrho_X \right)
-\left(e^\gamma-\frac{a_\gamma}{r}\varrho_X \right)(H^2)\\
=& \frac{a_\gamma}{-d_\gamma}(H+(H,\gamma)\varrho_X)
-(H^2)e^\gamma.
\end{split}
\end{equation}

Hence
\begin{equation}
\begin{split}
\Phi(\xi_{\omega_\lambda})=& 
\frac{-r_1 a_\gamma}{|r_1| d_\gamma}(\widehat{H}+
(\widehat{H},\gamma')\varrho_{X_1})
+\frac{(H^2)}{r_1}\varrho_{X_1}\\
=& \frac{\rk w}{|r_1| d_\gamma}(\widehat{H}+
(\widehat{H},\gamma')\varrho_{X_1})
+\frac{(H,d_\gamma H)}{ r_1 d_\gamma }\varrho_{X_1}\\
=& \frac{1}{|r_1| d_\gamma}( \rk w \widehat{H}+
(\widehat{H},c_1(w))\varrho_{X_1}),
\end{split}
\end{equation}
where we used the equality
$$
c_1(w)=\frac{r_1}{|r_1|}d_\gamma \widehat{H}
+\rk w \gamma'+D_{\gamma'}(w). 
$$

We also have
\begin{equation}
\begin{split}
-\frac{\rk w}{r}
\widehat{\Phi} \left(e^{\gamma'}+
\frac{\langle e^{\gamma'},w \rangle}{\rk w} \varrho_{X_1}\right)
=& e^\gamma-\frac{a_\gamma}{r}\varrho_X\\
=& \left(e^\beta-\frac{a_\beta}{r}\varrho_X \right)+
\lambda \left(H+\left(H,\frac{c_1(v)}{r}\right)\varrho_X \right). 
\end{split}
\end{equation}

\end{proof}

Assume that 
$v_2:=r_2 e^\beta+a_2 \varrho_X+(d_2 H+D_2)+(d_2 H+D_2,\beta)\varrho_X$
satisfies
\begin{equation}
\frac{(\omega^2)}{2}=\frac{d_2 a_\beta-d_\beta a_2}{d_2 r-d_\beta r_2}.
\end{equation}
Then 
\begin{equation}
\varphi_\omega=\frac{r_2 a_\beta-r a_2}{d_2 r- d_\beta r_2}.
\end{equation}
Hence
\begin{equation}\label{eq:trivial}
\langle v_2,\xi_\omega \rangle
=\varphi_\omega \left(d_2-\frac{r_2}{r}d_\beta \right)(H^2)
-\left(\frac{a_\beta}{r}r_2-a_2 \right)(H^2)
=0.
\end{equation}

From now on, we assume that ${\frak k}={\Bbb C}$.
Then we have a homomorphism
$$
\theta_v:v^{\perp} \to H^2(M_{(\beta,\omega)}(v),{\Bbb Z})
$$
which preserves the Hodge structures.
If $X$ is a $K3$ surface and $v$ is a primitive Mukai
vector with $\langle v^2 \rangle \geq 2$,
then $M_{(\beta,\omega)}(v)$ is an irreducible symplectic
manifold deformation equivalent to
$\Hilb_X^{\langle v^2 \rangle/2+1}$
by Theorem \ref{thm:projective} and
\cite{Y:7}.
We regard $H^2(M_{(\beta,\omega)}(v),{\Bbb Z})$
as a lattice by the Beuville's bilinear form (\cite{B:1}). Then 
$\theta_v$ is an isometry. 
If $X$ is an abelian surface, 
then we have the albanese morphism 
${\frak a}:M_{(\beta,\omega)}(v) \to X \times \widehat{X}$, which is
an \'{e}tale locally trivial fibration.
Let $K_{(\beta,\omega)}(v)$ be the albanese fiber.
Assume that $v$ is primitive and
$\langle v^2 \rangle \geq 6$.
If $\omega$ is general,
then $K_{(\beta,\omega)}(v)$ is an irreducible symplectic manifold
which is deformation equivalent to the generalized Kummer variety
constructed by Beauville \cite{B:1}.
We have an isomorphism
\begin{equation}
v^{\perp}  \to
 H^2(M_{(\beta,\omega)}(v),{\Bbb Z})  \to 
H^2(K_{(\beta,\omega)}(v),{\Bbb Z})
\end{equation}
which preserves the Hodge structures
(cf. \cite{Y:7}).
We also denote this map by $\theta_v$. 
Thus we have an isomorphism
\begin{equation}
\theta_v':v^{\perp} \cap H^*(X,{\Bbb Z})_{\alg} \to 
\NS(M_{(\beta,\omega)}(v)) \to
\NS(K_{(\beta,\omega)}(v))
\end{equation}
as the restriction of $\theta_v$.

\begin{prop}\label{prop:ample}
Assume that $\omega$ belongs to a chamber $I$.
Then $\theta_v(\xi_\omega) \in \NS(M_{(\beta,\omega)}(v))_{\Bbb R}$ 
belongs to the ample cone
of $M_{(\beta,\omega)}(v)$.
\end{prop}

\begin{proof}
For $\omega_\lambda \in I$, $\lambda \in {\Bbb Q}$,
we take the isomorphism
$\Phi:M_{(\beta,\omega_\lambda)}(v) \to M_{\widehat{H}}^{\gamma'}(w)$.
By Lemma \ref{lem:FM(xi)} and Lemma \ref{lem:GIT-polarization},
\begin{equation}
\begin{split}
\Phi \left(\theta_v(\xi_{\omega_\lambda}))\right)=&
\frac{\rk w}{|r_1| (d_\beta-\lambda r)}
\theta_w \left(\widehat{H}+
\left(\widehat{H},\frac{c_1(w)}{\rk w} \right)\varrho_{X_1}\right)\\
=& \frac{\rk w}{|r_1| (d_\beta-\lambda r)}{\cal L}(\xi_1)
\end{split}
\end{equation}
is a nef divisor on $M_{\widehat{H}}^{\gamma'}(w)$. 
Hence $\theta_v(\xi_{\omega_\lambda})$ 
belongs to the nef cone of $M_{(\beta,\omega)}(v)$. 
Moreover by Lemma \ref{lem:GIT-polarization},
$\Phi(\theta_v(\xi_\omega))$, $\omega \in I$
spans a 2-plane containing an ample divisor.   
Hence $\theta_v(\xi_{\omega})$, $\omega \in I$ 
belongs to the ample cone of $M_{(\beta,\omega)}(v)$.
%\begin{NB}
%Let $\theta_v(\xi_{\omega_{p/q}}+(\varrhi_{\omega
%=\theta_v(\xi$, $x_1 \leq t \leq x_2$
%be a curve in the nef cone such that .
%Let $C$ be a curve in $M_{(\beta,\omega)}(v)$.
%Then $(C,
%\end{NB} 
\end{proof}

\begin{cor}\label{cor:ample}
Let $v$ be a primitive Mukai vector.
\begin{enumerate}
\item[(1)]
Assume that $X$ is a $K3$ surface with $\NS(X)={\Bbb Z}H$
and $\langle v^2 \rangle \geq 2$.
For a chamber $I=(\omega_1 ,\omega_2) \subset {\Bbb R}_{>0}H$ such that 
$\omega_1,\omega_2$ belong to walls,
\begin{equation}\label{eq:ample:K3}
\Amp(M_{(\beta,\omega)}(v))_{\Bbb R} 
\supset \theta_v(\{ {\Bbb R}_{>0} \xi_\omega| \omega \in I \}).
\end{equation}
\item[(2)]
Assume that $X$ is an abelian surface and $\langle v^2 \rangle \geq 6$.
For a chamber $I=(\omega_1 ,\omega_2) \subset {\Bbb R}_{>0}H$ such that 
$\omega_1,\omega_2$ belong to walls,
$$
\Amp(M_{(\beta,\omega)}(v))_{\Bbb R} \cap L
=\theta_v(\{ {\Bbb R}_{>0} \xi_\omega| \omega \in I \}),
$$
where
$$
L:=\{ \theta_v(x)| \langle x,v \rangle=0, 
x \in {\Bbb R}e^\beta+{\Bbb R}(H+(H,\beta)\varrho_X)+
{\Bbb R}\varrho_X \}.
$$ 
In particular, if $\NS(X)={\Bbb Z}H$, then
$$
\Amp(K_{(\beta,\omega)}(v))_{\Bbb R}=\theta_v(
\{ {\Bbb R}_{>0} \xi_\omega| \omega \in I \}).
$$
\end{enumerate}
\end{cor}

\begin{proof}
(1) is obvious by Proposition \ref{prop:ample}. 

(2)
Assume that $\omega$ belongs to a boundary of $I$ 
defined by a wall $W_{v_1}$.
We set $v_2:=v-v_1$.
We may assume that 
$\phi_{(\beta,\omega)}(v_1)<\phi_{(\beta,\omega)}(v)$.
Assume that $\langle v_1,v_2 \rangle \geq 2$ and there are
$\sigma_{(\beta,\omega)}$-stable objects $E_i$, $i=1,2$
with $v(E_i)=v_i$.
Let $P$ be the projective space associated to
$\Ext^1(E_2,E_1)$. 
We take the associated extension 
$$
0 \to E_1(\lambda) \to E \to E_2 \to 0
$$ 
on $P \times X$,
where ${\cal O}_P(\lambda)$ is the tautological 
line bundle. 
Then by using \eqref{eq:trivial},
we see that $\theta_v(\xi_\omega)_{|P}={\cal O}_P$.
Thus $\theta_v(\xi_\omega)$ is not ample.

We next treat the remaining case.
By Proposition \ref{prop:properly-semi-stable},
replacing $v_1$ by another $v_1$, we may assume that
(i) $v=v_1+n v_2$, $\langle v_1^2 \rangle=\langle v_2^2 \rangle=0$ 
and $\langle v_1,v_2 \rangle=1$
or (ii) $v=v_1+v_2+v_3$, $\langle v_i^2 \rangle=0$,
$\langle v_i,v_j \rangle=1$, $(i \ne j)$. 
In the first case, we note that
$n=\langle v^2 \rangle/2 \geq 2$.
We take $E_1 \in M_{(\beta,\omega)}(v_2)$ 
and $E_2 \in M_{(\beta,\omega)}(v-2v_2)$. 
Since $\Ext^1(E_1,E_1) \cong {\Bbb C}$,
we have a family of non-trivial extensions $F$ of $E_1$ by $E_1$
parametrized by a projective line $P$.
Then $F \oplus E_2$ is a family of semi-stable objects and
we have $\theta_v(\xi_\omega)_{|P}={\cal O}_P$.
Thus $\theta_v(\xi_\omega)$ is not ample.
In the second case, for the ${\Bbb P}^1$-bundles in Corollary \ref{cor:A_2},
we can easily see that
$\theta_v(\xi_\omega)_{|D_i}$ is trivial along the fibers of the 
${\Bbb P}^1$-bundles.
Thus $\theta_v(\xi_\omega)$ is not ample. 
\end{proof}

\begin{rem}
The homomorphism
$$
\theta_v' : v^{\perp} \cap H^*(X,{\Bbb Z})_{\alg}
\to \NS(M_{(\beta,\omega)}(v))_{\Bbb Q}
$$
is defined over any field
${\frak k}$.
Replacing $\theta_v$ by $\theta_v'$, 
Proposition \ref{prop:ample} holds over any field ${\frak k}$.
\end{rem}

%\begin{rem}
%If $d_\beta(v)=2d_{\beta,\min}$, then
%the same proof of Corollary \ref{cor:ample} (2) implies that
%the inclusion in \eqref{eq:ample:K3} is the equality. 
%\end{rem}

\begin{NB}
$$
\langle e^{\beta+xH},v \rangle>0,\; x \in (x_0,\frac{d_\beta}{r}).
$$
Assume that 
$a_\beta>0$.
Then 
\begin{equation}
\begin{split}
\langle e^{\beta+xH},v \rangle
= & -a_\beta-r x^2 \frac{(H^2)}{2}+d_\beta x (H^2)\\
> & -a_\beta-\frac{a_\beta }{d_\beta}
\left(\frac{2 r a_\beta }{(H^2)d_\beta}-2 d_\beta \right)\\
= & - \frac{a_\beta }{d_\beta}
\left(\frac{2 r a_\beta }{(H^2)d_\beta}-d_\beta \right)\\
= & a_\beta \frac{\langle v^2 \rangle-(D^2)}{d_\beta^2}>0.
\end{split}
\end{equation}
If $a_\beta \leq 0$, then $\langle e^{\beta+xH},v \rangle >
-a_\beta \geq 0$.

\end{NB}

\begin{NB}
We set $n:=(H^2)/2$.
We note that
$|\frac{q}{p}-\frac{qnk}{pnk+1}|=\frac{q}{p(pnk+1)} \ll 1$
for $k \gg 0$.
Replacing $q/p$ by $\frac{qnk}{pnk+1}$,
we may assume that $(p,n)=1$.
Under this assumption, 
we may also assume that $(p,q)=1$.
Then $w_1:=p^2+pqH+q^2 n \varrho_X$ is a primitive and isotropic
Mukai vector with $\langle w_1,u \rangle =1$,
$u \in H^*(X,{\Bbb Z})_{\alg}$. 
Thus $M_H(w_1)$ defines an untwisted Fourier-Mukai transform.
This fact can be applied, if $\beta \in {\Bbb Q}H$.  
\end{NB}

\section{Appendix}\label{sect:appendix}

\subsection{Another proof of Theorem \ref{thm:mu-semi-stable}}
\label{subsect:another}
Assume that $X_1$ is a fine moduli scheme,
that is, ${\bf E}$ is an untwisted object.
Then Theorem \ref{thm:mu-semi-stable} directly follows from
\cite[Prop. 10.3]{Br:3}, as we explain below.

We note that
\begin{equation}
\begin{split}
e^{\beta+\sqrt{-1} \omega}=&
e^{\gamma}e^{(\beta-\gamma)+\sqrt{-1} \omega}\\
=& e^\gamma+\left(\frac{((\beta-\gamma)^2)-(\omega^2)}{2}+
\sqrt{-1}(\beta-\gamma,\omega) \right)\varrho_X+
\left(\beta-\gamma+\sqrt{-1}\omega+
(\beta-\gamma+\sqrt{-1}\omega,\gamma)\varrho_X \right).
\end{split}
\end{equation}
Hence 

\begin{equation}
\begin{split}
\Phi(e^{\beta+\sqrt{-1} \omega})=&
-r_1 \left(\frac{((\beta-\gamma)^2)-(\omega^2)}{2}+
\sqrt{-1}(\beta-\gamma,\omega) \right)e^{\gamma'}-
\frac{1}{r_1}\varrho_{X_1}\\
& +\frac{r_1}{|r_1|}
\left(\widehat{\beta}-\widehat{\gamma}+\sqrt{-1}\widehat{\omega}+
(\beta-\gamma+\sqrt{-1}\omega,\gamma)\varrho_{X_1} \right)\\
=& -r_1 \left(\frac{((\beta-\gamma)^2)-(\omega^2)}{2}+
\sqrt{-1}(\beta-\gamma,\omega) \right)
e^{\gamma'+\widehat{\xi}+\sqrt{-1}\widehat{\eta}},
\end{split}
\end{equation}
where
\begin{equation}
\begin{split}
\xi=& -\frac{1}{|r_1|}
\frac{1}{\left(\frac{((\beta-\gamma)^2)-(\omega^2)}{2}\right)^2+
(\beta-\gamma,\omega)^2}
\left(\frac{((\beta-\gamma)^2)-(\omega^2)}{2}(\beta-\gamma)
+(\beta-\gamma,\omega)\omega
\right),\\
\eta=& -\frac{1}{|r_1|}
\frac{1}{\left(\frac{((\beta-\gamma)^2)-(\omega^2)}{2}\right)^2+
(\beta-\gamma,\omega)^2}
\left(\frac{((\beta-\gamma)^2)-(\omega^2)}{2}\omega
-(\beta-\gamma,\omega)(\beta-\gamma)
\right). 
\end{split}
\end{equation}

Since $\langle \Phi(e^{\beta+\sqrt{-1}\omega}),x \rangle=
\langle e^{\beta+\sqrt{-1}\omega},\widehat{\Phi}(x) \rangle=
Z_{(\beta,\omega)}(\widehat{\Phi}(x))$,
we get the following commutative diagram:
\begin{equation*}
%\begin{CD}
%{\bf D}(X) @>>> {\bf D}(X_1)\\
%@V{Z_{(\beta,\omega)}}VV @VV{Z_{(\gamma'+\widehat{\xi},\widehat{\eta})}}V\\
%{\Bbb C} @>{\zeta^{-1}}>> {\Bbb C}
%\end{CD}
\xymatrix{
    {\bf D}(X) 
    \ar[r]^{\Phi} 
    \ar[d]_{Z_{(\beta,\omega)}}
    \ar@{}[dr]|\circlearrowleft
  & {\bf D}(X_1) 
    \ar[d]^{Z_{(\gamma'+\widehat{\xi},\widehat{\eta})} } 
  \\
    {\Bbb C} \ar[r]_{ \times \zeta^{-1}} 
  & {\Bbb C}                        
 }
\end{equation*}
with 
$$
\zeta=-r_1 \left(\frac{((\beta-\gamma)^2)-(\omega^2)}{2}+
\sqrt{-1}(\beta-\gamma,\omega) \right).
$$

We set $\beta=\gamma-\lambda H-\nu$ with $\nu \in H^{\perp}$.
Then we see that
\begin{equation}
\begin{split}
& \frac{((\beta-\gamma)^2)-(\omega^2)}{2}\omega
-(\beta-\gamma,\omega)(\beta-\gamma)\\
= &\frac{\lambda^2(H^2)+(\nu^2)-(\omega^2)}{2}\frac{(H,\omega)}{(H^2)}H-
\lambda(H,\omega)(\lambda H+\nu)\\
=& (H,\omega) \left(
\frac{-\lambda^2(H^2)+(\nu^2)-(\omega^2)}{2(H^2)}H-\lambda \nu \right)\\
=& -(H,\omega)L
\end{split}
\end{equation}
where $L$ is defined in Definition \ref{defn:L},
and
\begin{equation}
\begin{split}
& \frac{((\beta-\gamma)^2)-(\omega^2)}{2}(\beta-\gamma)
+(\beta-\gamma,\omega)\omega\\
=& 
-\frac{\lambda^2(H^2)+(\nu^2)-(\omega^2)}{2}(\lambda H+\nu)+
\lambda(H,\omega)\frac{(H,\omega)}{(H^2)}H\\
=& -\frac{\lambda^2(H^2)+(\nu^2)-(\omega^2)}{2}(\lambda H+\nu)+
\lambda(\omega^2)H\\
=& -\frac{\lambda^2(H^2)+(\nu^2)+(\omega^2)}{2}\lambda H
-\frac{\lambda^2(H^2)+(\nu^2)-(\omega^2)}{2}\nu.
\end{split}
\end{equation}

%\begin{equation}
%(d_{\gamma'}\widehat{H}+D_{\gamma'},\widehat{\eta})
%-r(\widehat{\xi},\widehat{\eta})
%\end{equation}

\begin{prop}\cite[Prop. 10.3]{Br:3}
Assume that 
$X_1$ is a fine moduli space of
$\sigma_{(\beta,\omega)}$-stable objects, that is,
${\bf E}$ is an untwisted object.  
Then $\widehat{\eta}$ is ample and 
$\Phi$ preserves the Bridgeland stability condition.
Thus $F$ is $\sigma_{(\beta,\omega)}$-semi-stable
if and only if $\Phi(F)$ is 
$\sigma_{(\gamma'+\widehat{\xi},\widehat{\eta})}$-semi-stable. 
\end{prop}

\begin{NB}
Assume that ${\bf E}$ parametrizes 
$\sigma_{(\beta,\omega)}$-stable objects.
Then the ampleness of $\widehat{\eta}$ implies that
$\Phi$ preserves the orientation. 
\end{NB}

If 
$\phi_{(\beta,\omega)}(F)=\phi_{(\beta,\omega)}(r_1 e^\gamma)$,
then $\phi_{(\gamma'+\widehat{\xi},\widehat{\eta})}(\Phi(F))
=\phi_{(\gamma'+\widehat{\xi},\widehat{\eta})}({\frak k}_{x_1}[-1])
\equiv 0
 \pmod {2{\Bbb Z}}$.
Hence $F$ is a $\sigma_{(\beta,\omega)}$-semi-stable object
with $\phi_{(\beta,\omega)}(F)=\phi_{(\beta,\omega)}(r_1 e^\gamma)$
if and only if
$\Phi(F)[1]$ is a $\sigma_{(\gamma'+\widehat{\xi},\widehat{\eta})}$-
semi-stable object
with $\phi_{(\gamma'+\widehat{\xi},\widehat{\eta})}(\Phi(F)[1])=1$.
This is equivalent to the condition that  
 $H^{-1}(\Phi(F)[1])$ is a $\mu$-semi-stable torsion free sheaf
of $(c_1(H^{-1}(\Phi(F)[1])(-\widehat{\xi})),
\widehat{\eta})=0$ with respect to
$\widehat{\eta}$ and
that $H^0(\Phi(F)[1])$ is a 0-dimensional sheaf.
Thus we get another proof of Theorem \ref{thm:mu-semi-stable}.

\begin{rem}
For $E \in K(X_1)$,
\eqref{eq:slope} is equivalent to
$(c_1(E(-\gamma'-\widehat{\xi})),\widehat{L})=0$:

Indeed we first note that
\begin{equation}
\begin{split}
& r_1^2\left\{ \left(\frac{((\beta-\gamma)^2)-(\omega^2)}{2}\right)^2+
(\beta-\gamma,\omega)^2 \right\}^2
\frac{(\xi,\eta)}{(H,\omega)}\\
=&
\frac{\lambda(\lambda^2(H^2)+(\omega^2)+(\nu^2))
(\lambda^2(H^2)+(\omega^2)-(\nu^2))}{4}+
\lambda(\nu^2)\frac{\lambda^2 (H^2)-(\omega^2)+(\nu^2)}{2}\\
=&\lambda \left\{ 
\left(\frac{\lambda^2 (H^2)-(\omega^2)-(\nu^2)}{2} \right)^2+
\lambda^2 (H^2) (\omega^2) \right \}\\
=&\lambda 
\left\{\left(\frac{((\beta-\gamma)^2)-(\omega^2)}{2}\right)^2+
(\beta-\gamma,\omega)^2 \right\}. 
\end{split}
\end{equation}
Since 
$$
L=\frac{\left(\frac{((\beta-\gamma)^2)-(\omega^2)}{2}\right)^2+
(\beta-\gamma,\omega)^2 }{(\omega,H)}|r_1|\eta,
$$
we have
\begin{equation}
\begin{split}
\lambda=&
 r_1^2\left\{ \left(\frac{((\beta-\gamma)^2)-(\omega^2)}{2}\right)^2+
(\beta-\gamma,\omega)^2 \right\}
\frac{(\xi,\eta)}{(\omega,H)}
= |r_1|(L,\xi).
\end{split}
\end{equation}

Thus
we get
\begin{equation}
\begin{split}
\frac{\left(\frac{((\beta-\gamma)^2)-(\omega^2)}{2}\right)^2+
(\beta-\gamma,\omega)^2 }{(\omega,H)}|r_1|
\left((c_1(E(-\gamma')),\widehat{\eta})
-\rk E(\widehat{\xi},\widehat{\eta}) \right)=&
(c_1(E(-\gamma')),\widehat{L})
-\frac{\rk E}{|r_1|} \lambda \\
=& (c_1(E(-\gamma'-\widehat{\xi})),\widehat{L}).
\end{split}
\end{equation}
\end{rem}

\subsection{Polarizations on the moduli spaces of stable sheaves.}

Let $X$ be a smooth projective surface with an ample divisor $H$.
In this section, we study Simpson's polarization \cite{S:1} of the
moduli spaces of $\beta$-twisted stable sheaves.
For a topological invariant $v$ (e.g. Chern character or
the equivalence class in the Grothendieck group $K(X)_{\mathrm{top}}$
of topological vector bundles),
$\overline{M}_H^\beta(v)$ denotes the moduli space of
$\beta$-twisted semi-stable sheaves.
We take a locally free sheaf $G$ with $\frac{c_1(G)}{\rk G}=\beta$.
Let $Q^{ss}$ be the open subscheme of
$\Quot_{G(-nH) \otimes V/X}$ such that
$\overline{M}_H^\beta(v)=Q^{ss} \dslash GL(V)$,
where $V$ is a vector space of dimension $\chi(G,E(nH))$, 
$E \in \overline{M}_H^\beta(v)$ and the action of $GL(V)$ 
is the natural one coming from the action on
$G(-nH) \otimes V$.
Let ${\cal Q}$ be the universal quotient on $Q^{ss} \times X$. Then
${\cal Q}_{|\{q \} \times X}$ is $G$-twisted semi-stable
for all $q \in Q^{ss}$
and ${\cal Q}$ is $GL(V)$-linearized.
By the construction of the moduli space,
we have a $GL(V)$-equivariant isomorphism
$V \to p_{Q^{ss}*}(G^{\vee} \otimes {\cal Q}(nH))$.
We set
\begin{equation}
\begin{split}
{\cal L}_{m,n}:=&
\det p_{Q^{ss} !}(G^{\vee} \otimes {\cal Q}((n+m)H))^{\otimes P(n)} \otimes
\det p_{Q^{ss} !}(G^{\vee} \otimes {\cal Q}(nH))^{\otimes (-P(m+n))}\\
=& \det p_{Q^{ss} !}(G^{\vee} \otimes {\cal Q}((n+m)H))^{\otimes P(n)} \otimes
\det V^{\otimes (-P(m+n))},
\end{split}
\end{equation}
where $P(n):=\chi(G,E(n))$ is the $G$-twisted Hilbert polynomial
of $E \in \overline{M}_H^\beta(v)$. 
It is a $GL(V)$-linearized line bundle on $Q^{ss}$, i.e.,
${\cal L}_{m,n} \in \Pic^{GL(V)}(Q^{ss})$.
By the construction of the moduli space, we get the following.

\begin{lem}\label{lem:polarization}
${\cal L}_{m,n}$, $m \gg n \gg 0$ is the pull-back of a relatively
ample line bundle on $\overline{M}_H^\beta(v)$.
\end{lem}

We set
\begin{equation}
\begin{split}
\xi_1:= & H+\left(H,\frac{c_1(v)}{r}-\frac{K_X}{2} \right)\varrho_X\\
\xi_2:= & -\left(e^\beta-\frac{\chi(e^{\beta},v)}{r}\varrho_X \right).
\end{split}
\end{equation}
For $n_1,n_2 \in {\Bbb Q}$, let 
${\cal L}(n_1 \xi_1+n_2 \xi_2)
\in \NS(\overline{M}_H^\beta(v))_{\Bbb Q}$ 
be an algebraic equivalence
class of a ${\Bbb Q}$-line bundle
such that 
$$
q^*({\cal L}(n_1 \xi_1+n_2 \xi_2)^{\otimes N})=
\det p_{Q^{ss}!}({\cal Q} \otimes F^{\vee})
\in \Pic^{GL(V)}(Q^{ss}),
$$
where $F \in {\bf D}(X)$ satisfies 
$\ch (F)=N(n_1 \xi_1+n_2 \xi_2)$, $N \gg 0$.
 Then we have a homomorphism
\begin{equation}
\begin{matrix}
{\Bbb Q}^{\oplus 2} & \to & \NS(\overline{M}_H^\beta(v))_{\Bbb Q} \\
(n_1,n_2) & \mapsto & {\cal L}(n_1 \xi_1+n_2 \xi_2).\\
\end{matrix}
\end{equation}
\begin{lem}\label{lem:GIT-polarization}
${\cal L}(\xi_1+\varepsilon  \xi_2)$
is a ${\Bbb Q}$-ample divisor on $\overline{M}_H^\beta(v)$ 
for $0<\varepsilon \ll 1$. 
In particular, ${\cal L}(\xi_1)$ defines a nef divisor.
\end{lem}

\begin{proof}
We note that
\begin{equation}
\begin{split}
& \frac{\chi(G^{\vee} \otimes E(nH))}{\rk G \rk E}
\frac{\ch (G^{\vee}((n+m)H))}{\rk G}-
\frac{\chi(G^{\vee} \otimes E((n+m)H))}{\rk G \rk E}
\frac{\ch(G^{\vee}(nH))}{\rk G}\\
=& 
\left( \frac{\chi(G^{\vee} \otimes E)}{\rk G \rk E}+
n \left(H,\frac{c_1(E)}{\rk E}-\frac{c_1(G)}{\rk G}-\frac{K_X}{2}\right)
+\frac{(H^2)}{2}n^2 \right)\\
& \times 
\left(\frac{\ch(G^{\vee})}{\rk G}+
(n+m)\left(H-\left(\frac{c_1(G)}{\rk G},H \right)\varrho_X \right)
+\frac{(H^2)}{2}(n+m)^2 \varrho_X \right)\\
& -\left( \frac{\chi(G^{\vee} \otimes E)}{\rk G \rk E}+
(n+m)\left(H,\frac{c_1(E)}{\rk E}-\frac{c_1(G)}{\rk G}-\frac{K_X}{2}\right)
+\frac{(H^2)}{2}(n+m)^2 \right)\\
& \times 
\left(\frac{\ch(G^{\vee})}{\rk G}+
n \left(H-\left(\frac{c_1(G)}{\rk G},H \right)\varrho_X \right)
+\frac{(H^2)}{2}n^2 \varrho_X \right)\\
=& m\left(n(n+m)\frac{(H^2)}{2}-
\frac{\chi(G^{\vee} \otimes E)}{\rk G \rk E} \right)
\left(-H+\left(H,\frac{c_1(E)}{\rk E}-\frac{K_X}{2}\right) 
\varrho_X \right) \\
& +m\left((2n+m)\frac{(H^2)}{2}+
\left(H,\frac{c_1(E)}{\rk E}-
\frac{c_1(G)}{\rk G}-\frac{K_X}{2} \right) \right)
\left(\frac{\chi(G^{\vee} \otimes E)}{\rk G \rk E}  \varrho_X-
\frac{\ch G^{\vee}}{\rk G} \right)
\end{split}
\end{equation}
and 
$$
\frac{\chi(G^{\vee} \otimes E)}{\rk G \rk E}  \varrho_X-
\frac{\ch G^{\vee}}{\rk G}=\xi_2^{\vee}.
$$
Since
$$
\lim_{m \to \infty}
\frac{(2n+m)\frac{(H^2)}{2}+(H,\frac{c_1(E)}{\rk E}-
\frac{c_1(G)}{\rk G}-\frac{K_X}{2}) 
}{n(n+m)\frac{(H^2)}{2}-
\frac{\chi(G^{\vee} \otimes E)}{\rk G \rk E}}
=\frac{1}{n}
$$
and $n$ is an arbitrary large integer, we get the claim. 
\end{proof}

\subsection{Some results to study the nef cone.}

In this subsection,
we shall give some results which are used in the proof of 
Corollary \ref{cor:ample} (2).
So assume that $X$ is an abelian surface.

We start with the following lemma.
\begin{lem}\label{lem:<,>}
Let $v_1,v_2$ be Mukai vectors such that
$Z_{(\beta,\omega)}(v_2) \in {\Bbb R} Z_{(\beta,\omega)}(v_1)$.
\begin{enumerate}
\item[(1)]
If $d_\beta(v_1),d_\beta(v_2)>0$, then
$\langle v_1,v_2 \rangle \geq 0$ and the equality holds only if
$\langle v_1^2 \rangle=0$ and $v_2 \in {\Bbb Q}v_1$.
\item[(2)]
If $\langle v_1,v_2 \rangle>0$ and $d_\beta(v_1)>0$, then
$d_\beta(v_2)>0$.
\item[(3)]
Assume that $\langle v_1, v_2 \rangle=1$, $d_\beta(v_1)>0$ and
$\langle v_1^2 \rangle \leq \langle v_2^2 \rangle$, then
$\langle v_1^2 \rangle=0$ and $v_2=v_2'+n v_1$, where 
$v_2'$ is an isotropic Mukai vector with
$d_\beta(v_2')>0$. 
\end{enumerate}
\end{lem}

\begin{proof}
(1) is a consequence of 
the formula
\begin{align*}
\frac{\langle v_1,v_2 \rangle}{d_\beta(v_1) d_\beta(v_2)}
=-&\frac{1}{2}
  \Bigl(\Bigl(\frac{D_{\beta}(v_1)}{d_\beta(v_1)}
               -\frac{D_{\beta}(v_2)}{d_\beta(v_2)}\Bigr)^2\Bigr)
 +\frac{\langle v_1^2 \rangle}{2 d_\beta(v_1)^2}+
  \frac{\langle v_2^2 \rangle}{2 d_\beta(v_2)^2}
\\
&+\left(
    \frac{r_\beta(v_1)}{d_\beta(v_1)}
   -\frac{r_\beta(v_2)}{d_\beta(v_2)} 
  \right) 
  \left(
      \frac{a_\beta(v_1)}{d_\beta(v_1)}
     -\frac{a_\beta(v_2)}{d_\beta(v_2)}
  \right)
\end{align*}
for $d_\beta(v_1),d_\beta(v_2)>0$ 
(see also \cite[Lem. 3.1.1]{MYY})
and the assumption 
$Z_{(\beta,\omega)}(v_2) \in {\Bbb R} Z_{(\beta,\omega)}(v_1)$.

(2)
If $d_\beta(v_2) < 0$, then
applying (1) to $v_1$ and $-v_2$, $\langle v_1,v_2 \rangle<0$,
which is a contradiction.
Therefore $d_\beta(v_2) \geq 0$. By
 $Z_{(\beta,\omega)}(v_2)\in
{\Bbb R}Z_{(\beta,\omega)}(v_1)$, $d_\beta(v_2) \ne 0$.
Thus the claim holds.

(3) If $\langle v_1^2 \rangle \geq 2$, then 
$\langle v_1,v_2 \rangle \geq 3$ by \cite[Lem. 4.2.4 (1)]{MYY}.
Hence $\langle v_1^2 \rangle=0$. 
We set $v_2':=v_2-\frac{\langle v_2^2 \rangle}{2}v_1$.
Since $\langle v_1,v_2' \rangle=1$ and $d_\beta(v_1)>0$,
(2) implies that $d_\beta(v_2')>0$.
\end{proof}

\begin{prop}\label{prop:properly-semi-stable}
Assume that $v=\sum_{i=1}^s n_i v_i$, where
$v_i$ are primitive Mukai vectors such that
$v_i \ne v_j$ $(i \ne j)$ and  
$\phi_{(\beta,\omega)}(v_i)=\phi_{(\beta,\omega)}(v)$.
\begin{enumerate}
\item[(1)]
If $s \geq 4$, then there are $\sigma_{(\beta,\omega)}$-stable
objects $E_1, E_2$ such that
$v=v(E_1)+v(E_2)$ and $\langle v(E_1),v(E_2) \rangle \geq 2$. 
\item[(2)]
Assume that $s=3$.
Then 
there are $\sigma_{(\beta,\omega)}$-stable
objects $E_1, E_2$ such that
$v=v(E_1)+v(E_2)$ and $\langle v(E_1),v(E_2) \rangle \geq 2$
unless 
$\langle v_i^2 \rangle=0$, $\langle v_i,v_j \rangle=1$, $i \ne j$
and $n_1=n_2=n_3=1$.
\item[(3)]
Assume that $s=2$.
%If $\langle v_1^2 \rangle,\langle v_2^2 \rangle>0$,
%then there are $\sigma_{(\beta,\omega)}$-stable
%objects $E_1, E_2$ such that
%$v=v(E_1)+v(E_2)$ and $\langle v(E_1),v(E_2) \rangle \geq 2$.
\begin{NB}
If $\langle v_1, v_2 \rangle \geq 2$ or $n_1,n_2 \geq 2$,
then there are $\sigma_{(\beta,\omega)}$-stable
objects $E_1, E_2$ such that
$v=v(E_1)+v(E_2)$ and $\langle v(E_1),v(E_2) \rangle \geq 2$
\end{NB}
Then there are $\sigma_{(\beta,\omega)}$-stable
objects $E_1, E_2$ such that
$v=v(E_1)+v(E_2)$ and $\langle v(E_1),v(E_2) \rangle \geq 2$
unless 
(a) there is an isotropic Mukai vector $w_1$
such that $\langle v,w_1 \rangle=1$ and $\phi_{(\beta,\omega)}(w_1)
=\phi_{(\beta,\omega)}(v)$
or (b) $n_1=n_2=1$, 
$\langle v_1^2 \rangle=\langle v_2^2 \rangle=0$
and $\langle v_1,v_2 \rangle=1$.
\end{enumerate}
\end{prop}

\begin{rem}
In the case of (a) in (3),
we have $w_1=v_1$, $n_2=1$ or
$w_1=v_2$ and $n_1=1$. We also have $\langle v_1,v_2 \rangle=1$. 
\end{rem}

\begin{proof}
(1)
If $s \geq 4$, then Lemma \ref{lem:existence-of-stable}
below implies that there is a 
$\sigma_{(\beta,\omega)}$-stable object
$E_1$ with $v(E_1)=v-v_1$.
Let $E_2$ be a $\sigma_{(\beta,\omega)}$-stable object
with $v(E_2)=v_1$.
Then the claim holds.

(2)
We first assume that $(n_1,n_2,n_3) \ne (1,1,1)$.
We may assume that $n_1>1$.
Then there is a $\sigma_{(\beta,\omega)}$-stable object
$E_1$ with $v(E_1)=v-v_1$.
Let $E_2$ be a $\sigma_{(\beta,\omega)}$-stable object
with $v(E_2)=v_1$.
Then the claim holds.

We next assume that $n_1=n_2=n_3=1$.
We shall classify the Mukai vectors
for which the claim does not hold.
We divide the argument into two cases.

(2-1) 
If $\langle v_1,v_2 \rangle>1$, 
then Lemma \ref{lem:existence-of-stable} (2) below
implies that there is a 
$\sigma_{(\beta,\omega)}$-stable object
$E_1$ with $v(E_1)=v_1+v_2$.
For a $\sigma_{(\beta,\omega)}$-stable object
$E_2$ with $v(E_2)=v_3$, $\langle v(E_1),v(E_2) \rangle \geq 2$.
Thus the claim holds.

(2-2)
We next assume that $\langle v_i,v_j \rangle=1$ for $i \ne j$.
If $n:=\langle v_1^2 \rangle/2>0$, then
$\langle v_1,v_i \rangle=1$ $(i=2,3)$ implies that
$\langle v_2^2 \rangle=\langle v_3^2 \rangle=0$ 
by Lemma \ref{lem:<,>} (3).
Since $v_1':=v_1-n v_2$ satisfies
$\langle {v_1'}^2 \rangle =0$
and $\langle v_1',v_2 \rangle=1$,
we get $d_\beta(v_1')>0$ by Lemma \ref{lem:<,>} (2). 
Then $0<\langle v_1',v_3 \rangle=\langle v_1,v_3 \rangle-n \leq 0$.
Therefore $\langle v_1^2 \rangle=0$.
In this case, $\langle v^2 \rangle=6$.

(3)
We may assume that $\langle v_1^2 \rangle \leq \langle v_2^2 \rangle$.
(3-1) We first assume that $\langle v_1,v_2 \rangle \geq 2$ and divide
the argument into two cases.
\begin{enumerate}
\item Assume that $\langle v_1^2 \rangle>0$.
If there is a primitive isotropic Mukai vector
$w_1$ such that $\langle v_1,w_1 \rangle=1$,
then
$v_1=u_1+n w_1$, $\langle u_1^2 \rangle=0$.
Then $v=n_1 u_1+n_1 n w_1+n_2 v_2$ and $u_1,w_1$ and $v_2$ are 
different primitive vectors. 
By (2), we get the claim.

Otherwise, by Lemma \ref{lem:existence-of-stable},
there are $\sigma_{(\beta,\omega)}$-stable
objects $E_1, E_2$ with $v(E_1)=v_1,v(E_2)=v_2$.
Since $\langle v_1,v_2 \rangle \geq 3$, we get the claim.

\item
Assume that $\langle v_1^2 \rangle=0$.
By Lemma \ref{lem:existence-of-stable},
if there is an integer $i$ with $n_i \geq 2$, then
there is a $\sigma_{(\beta,\omega)}$-stable
object $E_1$ with $v(E_1)=v-v_i$.
For a $\sigma_{(\beta,\omega)}$-stable
object $E_2$ with $v(E_2)=v_i$,
$\langle v(E_1), v(E_2) \rangle \geq 2$.
Thus the claim holds.

If $n_1=n_2=1$, then for 
$\sigma_{(\beta,\omega)}$-stable
objects $E_i$ $(i=1,2)$ with $v(E_i)=v_i$,
the claim holds.
\end{enumerate}

(3-2) 
Assume that $\langle v_1,v_2 \rangle =1$.
By Lemma \ref{lem:<,>} (3), $\langle v_1^2 \rangle=0$. 
We set $v_2':=v_2-\frac{\langle v_2^2 \rangle}{2}v_1$
 and set $n_1':=n_1+n_2 \frac{\langle v_2^2 \rangle}{2}$.
Then $v=n_1' v_1+n_2 v_2'$, $\langle v_1,v_2' \rangle=1$ and
$\langle {v_2'}^2 \rangle=0$.

If $n_1'=1$ or $n_2=1$, then $v_1$ or $v_2'$ is an isotropic
Mukai vector $w_1$ with $\langle v,w_1 \rangle=1$.
Hence we assume that
$n_1',n_2 \geq 2$.

If $n_1'=2$, then (i) $n_1=2$ and $\langle v_2^2 \rangle=0$, or 
(ii) $n_1=1$ and $n_2 \langle v_2^2 \rangle/2=1$.
Since $n_2 \geq 2$, (ii) does not occur. 
The case (i) corresponds to (b).
If $n_1' \geq 3$, then
Lemma \ref{lem:existence-of-stable} implies that   
there is a $\sigma_{(\beta,\omega)}$-stable
object $E_1$ with $v(E_1)=(n_1'-1) v_1+n_2 v_2'$.
For a $\sigma_{(\beta,\omega)}$-stable
object $E_2$ with $v(E_2)=v_1$,
$\langle v(E_1),v(E_2) \rangle= n_2 \geq 2$.
Therefore the claim holds.
\end{proof}

\begin{lem}\label{lem:existence-of-stable}
Let $v$ be a Mukai vector such that
$\langle v^2 \rangle>0$ and $d_\beta(v)>0$.
\begin{enumerate}
\item[(1)]
If
there is no isotropic Mukai vector
$w_1$ with $\langle v,w_1 \rangle=1$
and $\phi_{(\beta,\omega)}(w_1)=\phi_{(\beta,\omega)}(v)$, 
then
there is a $\sigma_{(\beta,\omega)}$-stable
object $E$ with
$v=v(E)$.
\item[(2)]
Assume that $v$ has a decomposition
$v=\sum_{i=1}^s n_i v_i$, where 
$v_i$ are primitive Mukai vectors such that
$v_i \ne v_j$ $(i \ne j)$, 
$\langle v_i^2 \rangle \geq 0$ and
$\phi_{(\beta,\omega)}(v_i)=\phi_{(\beta,\omega)}(v)$.
Then there is a $\sigma_{(\beta,\omega)}$-stable
object $E$ with $v=v(E)$,
unless there is an isotropic Mukai vector
$w_1$ with $\phi_{(\beta,\omega)}(w_1)=
\phi_{(\beta,\omega)}(v)$ and $\langle v,w_1 \rangle=1$. 
In particular if $s \geq 3$ or $s=2$ and $\langle v_1,v_2 \rangle \geq 2$, 
then there is a 
$\sigma_{(\beta,\omega)}$-stable
object $E$ with $v=v(E)$.
\end{enumerate}
%Moreover the same claim holds 
%if $n_1 \langle v_1^2 \rangle+n_2 \langle v_1,v_2 \rangle, 
%n_2 \langle v_2^2 \rangle+n_2 \langle v_1,v_2 \rangle>0$,
%then
\end{lem}

\begin{proof}
(1)
By \cite[Prop. 4.2.5]{MYY}, 
${\cal M}_{(\beta,\omega_+)}(v) \cap {\cal M}_{(\beta,\omega_-)}(v)
\ne \emptyset$, if
there is no isotropic Mukai vector
$w_1$ with $\langle v,w_1 \rangle=1$.
Since $\langle v^2 \rangle>0$ and 
${\cal M}_{(\beta,\omega_+)}(v)$ is isomorphic to 
a moduli stack of semi-stable sheaves (Theorem \ref{thm:main}),
\cite[Lem. 3.2]{tori} implies that 
there is a $\sigma_{(\beta,\omega_+)}$-stable object $E$ 
in ${\cal M}_{(\beta,\omega_+)}(v) \cap {\cal M}_{(\beta,\omega_-)}(v)$.
Then $E$ is $\sigma_{(\beta,\omega)}$-stable.
Indeed for a subobject $E_1$ of $E$ with 
$\phi_{(\beta,\omega)}(E_1)=\phi_{(\beta,\omega)}(E)$,
$\phi_{(\beta,\omega_\pm)}(E_1) \leq \phi_{(\beta,\omega_\pm)}(E)$.
Then $\phi_{(\beta,\omega_\pm)}(E_1)=\phi_{(\beta,\omega_\pm)}(E)$,
which implies that $E$ is properly $\sigma_{(\beta,\omega_+)}$-
semi-stable. 
Therefore $E$ is $\sigma_{(\beta,\omega)}$-stable.

(2) We first assume that $s \geq 3$.
Then for any isotropic Mukai vector $w_1$ with
$\phi_{(\beta,\omega)}(w)=\phi_{(\beta,\omega)}(v)$,
$\langle v,w_1 \rangle \geq 2$
since there are two different indices $1 \le i,j \le s$
such that 
$$\langle v,w \rangle 
 \geq  n_i \langle v_i,w \rangle 
     + n_j \langle v_j,w \rangle 
 \geq n_i +n_j 
 \geq2.
$$
Here we used Lemma~\ref{lem:<,>} (1). 
By (1), the claim holds. 
 
We next assume that $s=2$.
%If $\langle v_1^2 \rangle,\langle v_2^2 \rangle>0$,
%then for any isotropic Mukai vector $w$ with 
%$\phi_{(\beta,\omega)}(w)=\phi_{(\beta,\omega)}(v)$,
%$\langle v_1,w \rangle,\langle v_2,w \rangle>0$.
%Hence the assumption of (1) holds.
Let $w_1$ be an isotropic Mukai vector with
$\phi_{(\beta,\omega)}(w_1)=\phi_{(\beta,\omega)}(v)$. 
We have $\langle w_1,v \rangle \geq n_1+n_2 \geq 2$,
if $w \ne v_1,v_2$.
Hence if $\langle v_1,v_2 \rangle \geq 2$, then
the assumption of (1) also holds.
So the conclusion follows from (1).

If $\langle v_1,v_2 \rangle =1$, 
then we may assume that
$0=\langle v_1^2 \rangle \leq \langle v_2^2 \rangle$
by Lemma \ref{lem:<,>} (3).
In this case, $v=n_1' v_1+n_2 v_2'$, where
$v_2'=v_2-\frac{\langle v_2^2 \rangle}{2}v_1$
and $n_1'=n_1+n_2 \frac{\langle v_2^2 \rangle}{2}$.
If $n_1',n_2  \geq 2$, then   
the assumption of (1) holds. 
So the conclusion follows from (1).
If $n_1'=1$ or $n_2=1$, then
there is an isotropic Mukai vector
$w_1$ with $\langle v,w_1 \rangle=1$. 
\end{proof}

%We shall study $\sigma_{(\beta,\omega)}$-semi-stable objects
%whose Mukai vector has the decomposition in
%Proposition \ref{prop:properly-semi-stable} (2).

\begin{lem}\label{lem:A_2}
Let $v_i$ $(i=1,2,3)$ be primitive isotropic Mukai vectors
such that $\langle v_i,v_j \rangle=1$ $(i \ne j)$,
$\phi_{(\beta,\omega)}(v_i)=v$ and $v=\sum_{i=1}^3 v_i$.
Let $(\beta',\omega')$ be a general pair
which is close to $(\beta,\omega)$.
\begin{enumerate}
\item[(1)]
Assume that $\phi_{(\beta',\omega')}(v_1),\phi_{(\beta',\omega')}(v_2)
<\phi_{(\beta',\omega')}(v)$.
For $E_i \in {\cal M}_{(\beta,\omega)}(v_i)$,
we take non-trivial extensions
$$
0 \to E_1 \to E_{13} \to E_3 \to 0
$$
and 
$$
0 \to E_2 \to E_{23} \to E_3 \to 0.
$$
Then $E_{13}$ and $E_{23}$ are $\sigma_{(\beta',\omega')}$-stable
objects.
We take non-trivial extensions
$$
0 \to E_1 \to E_{123} \to E_{23} \to 0
$$
and 
$$
0 \to E_{2} \to E_{213} \to E_{13} \to 0.
$$
Then $E_{123}$ and $E_{213}$ are $\sigma_{(\beta',\omega')}$-stable
objects.
\item[(2)]
Assume that $\phi_{(\beta',\omega')}(v_1),\phi_{(\beta',\omega')}(v_2)
>\phi_{(\beta',\omega')}(v)$.
For $E_i \in {\cal M}_{(\beta,\omega)}(v_i)$,
we take non-trivial extensions
$$
0 \to E_3 \to E_{31} \to E_1 \to 0
$$
and 
$$
0 \to E_3 \to E_{32} \to E_2 \to 0.
$$
Then $E_{31}$ and $E_{32}$ are $\sigma_{(\beta',\omega')}$-stable
objects.
We take non-trivial extensions
$$
0 \to E_{31} \to E_{312} \to E_2 \to 0
$$
and 
$$
0 \to E_{32} \to E_{321} \to E_{1} \to 0.
$$
Then $E_{312}$ and $E_{321}$ are $\sigma_{(\beta',\omega')}$-stable
objects.
\end{enumerate}

\end{lem}

\begin{proof}
We only prove the $\sigma_{(\beta',\omega')}$-stability of $E_{123}$
in (1).
Assume that $E_{123}$ contains an object $F$ with
$\phi_{(\beta',\omega')}(F)>\phi_{(\beta',\omega')}(v)$.
We may assume that 
$\phi_{(\beta,\omega)}(F)=\phi_{(\beta,\omega)}(v)$.
Then we see that (i) $F \cong E_3$ or (ii) $E_{123}/F \cong E_1$ or (iii)
$E_{123}/F \cong E_2$.
If $F \cong E_3$, then $E_{23} \cong E_2 \oplus E_3$.
If $E_{123}/F \cong E_1$, then $E_{123} \cong E_{23} \oplus E_1$.
If $E_{123}/F \cong E_2$, then $E_1 \subset F$ and 
$E_{23}$ contains $F/E_2 \cong E_3$.
Therefore each case does not occur.
Hence $E_{123}$ is $\sigma_{(\beta',\omega')}$-stable.
\end{proof}

Since $\langle v_1+v_2,v_3 \rangle=2$, we get the following result.
\begin{cor}\label{cor:A_2}
Under the same assumptions of Lemma \ref{lem:A_2},
$M_{(\beta',\omega')}(v)$ contains two ${\Bbb P}^1$-bundles $D_1$, $D_2$ 
over
$\prod_{i=1}^3 M_{(\beta,\omega)}(v_i)$.
\end{cor}

\begin{rem}
It seems that
they intersect transversely.
Thus $D_1$ and $D_2$ give an $A_2$-type configuration.
\end{rem}

\begin{NB}

\section{}

\begin{lem}
Assume that $\langle w_1^2 \rangle=-2$. 
Assume that $(\beta,\omega)$ belongs to \begin{NB2}exactly\end{NB2} one wall
$W_{w_1}$.
For a subobject $F_1$ of $F$ with $v(E)=v$ and $\phi(F_1)=\phi(F)$,
$$
\frac{c_1(F_1(-\gamma))}{\chi(G,F_1)}=\frac{c_1(F(-\gamma))}{\chi(G,F)},
$$
where $G$ be an object of ${\frak A}$ with
$v(G)=r_0 e^\gamma-\frac{1}{r_0}\varrho_X$, 
$w_1=r_0 e^\xi+\frac{1}{r_0}\varrho_X$.
\end{lem}

\begin{proof}
We note that $w_1=r_0 e^\gamma+\frac{1}{r_0}\varrho_{X_1}$.
We set
\begin{equation}
\begin{split}
v=& v(F)=r_\gamma(v)e^\gamma+a_\gamma(v) \varrho_X+
(\xi+(\xi,\gamma)\varrho_X),\\
v_1=& v(F_1)=r_\gamma(v_1)e^\gamma+a_\gamma(v_1) \varrho_X+
(\xi_1+(\xi_1,\gamma)\varrho_X).
\end{split}
\end{equation}
By the assumption, 
$d_\beta(v_1)v-d_\beta(v)v_1$ and
$d_\beta(w_1)v-d_\beta(v)w_1$ are linearly dependent.

We note that
\begin{equation}
\begin{split}
& \deg_\beta(v_1)v-\deg_\beta(v)v_1\\
=&
(r_\gamma(v_1)\deg(\gamma-\beta)+\deg_\gamma(v_1))
(r_\gamma(v)e^\gamma+a_\gamma(v)\varrho_X+\xi+(\xi,\gamma)\varrho_X)\\
&-
(r_\gamma(v)\deg(\gamma-\beta)+\deg_\gamma(v))
(r_\gamma(v_1)e^\gamma+a_\gamma(v_1)\varrho_X+\xi_1+(\xi_1,\gamma)\varrho_X)\\
=& (\deg_\gamma(v_1)r_\gamma(v)-\deg_\gamma(v)r_\gamma(v_1))e^\gamma+
(r_\gamma(v_1)a_\gamma(v)-r_\gamma(v)a_\gamma(v_1))
\deg(\gamma-\beta)\varrho_X\\
& +
(\deg_\gamma(v_1)a_\gamma(v)-\deg_\gamma(v)a_\gamma(v_1))\varrho_X\\
& +(r_\gamma(v_1)\xi-r_\gamma(v)\xi_1)\deg(\gamma-\beta)
+(\deg_\gamma(v_1)\xi-\deg_\gamma(v)\xi_1)
\end{split}
\end{equation} 
and
\begin{equation}
\begin{split}
& \deg_\beta(w_1)v-\deg_\beta(v)w_1\\
=&
(r_0 \deg(\gamma-\beta)+\deg_\gamma(w_1))
(r_\gamma(v)e^\gamma+a_\gamma(v)\varrho_X+\xi+(\xi,\gamma)\varrho_X)\\
&-
(r_\gamma(v)\deg(\gamma-\beta)+\deg_\gamma(v))
(r_0 e^\gamma+\frac{1}{r_0}\varrho_{X})\\
=& -\deg_\gamma(v)r_0 e^\gamma+
r_0 a_\gamma(v)\deg(\gamma-\beta)\varrho_X
+r_0 \xi \deg(\gamma-\beta)
-\frac{\deg_\beta(v)}{r_0}\varrho_{X}.
\end{split}
\end{equation} 
Assume that $\deg_\gamma(v) \ne 0$.
Then 
\begin{equation}
\begin{split}
& (\deg_\gamma(v_1)r_\gamma(v)-\deg_\gamma(v)r_\gamma(v_1))
r_0 \xi \deg(\gamma-\beta)\\
&+
r_0 \deg_\gamma(\xi)(
(\deg_\gamma(\xi_1)\xi-\deg_\gamma(\xi)\xi_1)+
\deg(\gamma-\beta)(r_\gamma(v_1)\xi-r_\gamma(v)\xi_1))
=0,\\
& (\deg_\gamma(v_1)r_\gamma(v)-\deg_\gamma(v)r_\gamma(v_1))
(r_0 a_\gamma(v)\deg(\gamma-\beta)-\frac{\deg_\beta(v)}{r_0})\\
& +\deg_\gamma(\xi)r_1((r_\gamma(v_1)a_\gamma(v)-r_\gamma(v)a_\gamma(v_1))
\deg(\gamma-\beta)+
(\deg_\gamma(v_1)a_\gamma(v)-\deg_\gamma(v)a_\gamma(v_1)))=0.
\end{split}
\end{equation}
By the first equality,
we have
\begin{equation}
r_0(\deg_\gamma(\xi)+r_\gamma(v)\deg(\gamma-\beta))
(\deg_\gamma(v_1)\xi-\deg_\gamma(v)\xi_1)=0.
\end{equation}
Since $(\deg_\gamma(\xi)+r_\gamma(v)\deg(\gamma-\beta))=\deg_\beta(v)>0$,
we have
\begin{equation}
\deg_\gamma(v_1)\xi-\deg_\gamma(v)\xi_1=0.
\end{equation}
Then by the second equality,
we have
\begin{equation}
r_0 \deg_\beta(v)
\left(\deg_\gamma(v_1)a_\gamma(v)-\deg_\gamma(v)a_\gamma(v_1))
-\frac{\deg_\beta(v)}{r_0}
(\deg_\gamma(v_1)r_\gamma(v)-\deg_\gamma(v)r_\gamma(v_1)\right)
=0,
\end{equation}
which implies that
\begin{equation}
\deg_\gamma(v_1)\left(a_\gamma(v)-\frac{r_\gamma(v)}{r_0^2} \right)-
\deg_\gamma(v) \left(a_\gamma(v_1)-\frac{r_\gamma(v_1)}{r_0^2} \right)
=0.
\end{equation}
Therefore 
\begin{equation}
\xi_1 \left(a_\gamma(v)-\frac{r_\gamma(v)}{r_0^2} \right)
-\xi \left(a_\gamma(v_1)-\frac{r_\gamma(v_1)}{r_0^2} \right)
=0,
\end{equation}
which implies the claim.
\end{proof}

\section{}
Assume that $d|\langle w_1,u \rangle$ for all $u \in H^*(X,{\Bbb Z})_{\alg}$
and $\langle v, w_1 \rangle =d$.

Assume that $\phi_\pm(w_1)<\phi_\pm(w)$.
Let $F$ be a semi-stable object with respect to $\phi_\pm$
such that $v(F)=w$.
Then $E:=\Phi_{X \to X_1}^{{\bf E}^{\vee}[1]}(F)$
is a torsion free $\alpha$-twisted sheaf on $X_1$.

Conversely for a torsion free $\alpha$-twisted sheaf $E$ with $v(E)=v$,
we set $F:=\Phi_{X_1 \to X}^{{\bf E}[1]}(E)$.
If $E$ is not semi-stable with respect to $\phi$, then
we have an exact triangle
\begin{equation}
F_1 \to F \to F_2 \to F_1[1]
\end{equation} 
such that $\phi_{\max}(F_1)>\phi(w_1)$ and
$\phi_{\min}(F_2) \leq \phi(w_1)$.
Then we have $\Hom({\bf E}_{|X \times \{ x_1 \}},F_1[k])=0$
for $k \geq 2$ and
$\Hom({\bf E}_{|X \times \{ x_1 \}},F_2[k])=0$ for $k<0$.
Moreover $\Hom({\bf E}_{|X \times \{ x_1 \}},F_2)=0$ except
finitely many $x_1 \in X_1$.
Since $\Hom({\bf E}_{|X \times \{ x_1 \}},F[k])=
\Hom({\frak k}_{x_1}[-2],E[k-1])$,
$\Hom({\bf E}_{|X \times \{ x_1 \}},F[k])=0$ for
$k \ne 0,1$.
Since $E$ is torsion free, we also have 
$\Hom({\bf E}_{|X \times \{ x_1 \}},F[k])=0$
except for finitely many point $x_1 \in X_1$.
We set $E_i:=\Phi_{X \to X_1}^{{\bf E}^{\vee}[1]}(F_i)$.
Then $E_i$ are torsion free sheaves fitting in an exact sequence
$$
0 \to E_1 \to E \to E_2 \to 0.
$$ 
This is a contradiction. Thus $F$ is semi-stable with respect to
$\phi$.
Assume that there is an exact sequence  
\begin{equation}
0 \to F_1 \to F \to F_2 \to 0
\end{equation} 
such that $\phi_\pm(F_1) \geq \phi_\pm(F)$ and $F_1$ is $\phi_\pm$-stable.
Then $\phi(F_1)=\phi(F)=\phi(w_1)$.
Since $\phi_\pm(F_1) \geq \phi_\pm(F) >\phi_\pm(w_1)$,
we have
$\Hom({\bf E}_{|X \times \{ x_1 \}},F_1[2])=0$.
Since $\Hom({\bf E}_{|X \times \{ x_1 \}},F_2)=0$
except finitely many point $x_1 \in X_1$,
we have an exact sequence 
$$
0 \to E_1 \to E \to E_2 \to 0,
$$ 
where $E_1$ and $E_2$ are torsion free sheaves,
which is a contradiction.
Therefore $E$ is $\phi_\pm$-stable.

Assume that there is an exact sequence  
\begin{equation}
0 \to F_1 \to F \to F_2 \to 0
\end{equation} 
such that $\phi_\pm(F_1) \geq \phi_\pm(F)$ and $F_1$ is $\phi_\pm$-stable.
Then $\phi(F_1)=\phi(F)=\phi(w_1)$.
Since $\phi_\pm(F_1) \geq \phi_\pm(F) >\phi_\pm(w_1)$,
we have
$\Hom({\bf E}_{|X \times \{ x_1 \}},F_1[2])=0$.
Since $\Hom({\bf E}_{|X \times \{ x_1 \}},F_2)=0$
except finitely many point $x_1 \in X_1$,
we have an exact sequence 
$$
0 \to E_1 \to E \to E_2 \to 0,
$$ 
where $E_1$ and $E_2$ are torsion free sheaves.

\begin{NB2}
\begin{proof}
Then
we have $\Sigma_{(\beta,\omega)}(F,F_1)> 0$
and
$\Sigma_{(\beta,\omega)}(F_2,F)> 0$.
Since $\phi_{\min}(F_1)>\phi(w_1)$ and
$\phi_{\max}(F_2) \leq \phi(w_1)$,
we have $\Hom({\bf E}_{|X \times \{ x_1 \}},F_1[k])=0$
for $k \geq 2$ and
$\Hom({\bf E}_{|X \times \{ x_1 \}},F_2[k])=0$ for $k<0$.
Moreover $\Hom({\bf E}_{|X \times \{ x_1 \}},F_2)=0$ except
finitely many $x_1 \in X_1$.
We set $E_i:=\Phi_{X \to X_1}^{{\bf E}^{\vee}[1]}(F_i)$.
Then $E_i$ are torsion free sheaves fitting in an exact sequence
$$
0 \to E_1 \to E \to E_2 \to 0.
$$ 
This is a contradiction. Thus $F$ is semi-stable with respect to
$(\beta,\omega)$.
\end{proof}
\end{NB2}

\section{A special case}

Assume that $\NS(X)={\Bbb Z}H$ and set
$\beta=bH$ and $\gamma=cH$.

\begin{equation}
\begin{split}
w_1:=& r_1 e^{cH}\\
=& r_1 \left(e^{bH}+\frac{(b-c)^2}{2}(H^2) \varrho_X
+(b-c)(H+b(H^2)\varrho_X)\right).
\end{split}
\end{equation}

For $v$, we have the following relation.
\begin{equation}
\begin{split}
v=& r_\gamma(v) e^{cH}+a_\gamma(v) \varrho_X +
d_\gamma(v)(H+c(H^2)\varrho_X)\\
=& r_\gamma(v) e^{bH}+
\left(a_\gamma(v)+d_\gamma(v)(c-b)(H^2)+\frac{r_\gamma(v)}{2}(b-c)^2
\right)\varrho_X
+(r_\gamma(v)(c-b)+d_\gamma(v))(H+b(H^2)\varrho_X).
\end{split}
\end{equation}

Then
\begin{equation}
\begin{split}
& a_\beta(v)d_\beta(w_1)-a_\beta(w_1)d_\beta(v)\\
=& \left(\frac{r_\gamma(v)}{2}(b-c)^2 (H^2)+d_\gamma(v)(c-b)(H^2)
+a_\gamma(v) \right)r_1(c-b)\\
& -\frac{r_1}{2}(b-c)^2(H^2)(r_\gamma(v)(c-b)+d_\gamma(v))\\
=& d_\gamma(v)(c-b)^2 r_1 (H^2)+a_\gamma(v)r_1(c-b)-
\frac{r_1}{2}(b-c)^2 d_\gamma(v) (H^2)\\
=& r_1(c-b)\left(a_\gamma(v)+\frac{d_\gamma(v)}{2}(c-b)(H^2) \right)
\end{split}
\end{equation}
and
\begin{equation}
\begin{split}
r_\beta(v) d_\beta(w_1)-r_1 d_\beta(v)=&
r_\gamma(v) r_1(c-b)-r_\beta(w_1)(r_\gamma(v)(c-b)+d_\gamma(v))\\
=& -r_1 d_\gamma(v).
\end{split}
\end{equation}

We note that $-r_1 a_\gamma(v)>0$ and
$r_1(c-b)>0$.

Assume that
\begin{equation}
(r_\beta(v)d_\beta(w_1)-r_1 d_\beta(v))(\omega^2)=
2(a_\beta(v)d_\beta(w_1)-a_\beta(w_1)d_\beta(v)).
\end{equation}
If $r_1>0$, then 
$a_\gamma(v)<0$ and we have
\begin{equation}
0>r_1(c-b)a_\gamma(v)=-r_1 \left(\frac{(\omega^2)}{2}+(c-b)^2\frac{(H^2)}{2}
\right)d_\gamma(v).
\end{equation}
Hence $d_\gamma(v)>0$.

If $r_1<0$, then $a_\gamma(v)>0$ and we have
\begin{equation}
0<r_1(c-b)a_\gamma(v)=-r_1 \left(\frac{(\omega^2)}{2}+(c-b)^2\frac{(H^2)}{2}
\right)d_\gamma(v).
\end{equation}
Hence $d_\gamma(v)>0$.

We note that $\phi(F_i)<\phi(w_1)$
if and only if
\begin{equation}
(r(F_i)d_\beta(w_1)-r_1 d_\beta(F_i))(\omega^2)>
2(a_\beta(F_i)d_\beta(w_1)-a_\beta(w_1)d_\beta(F_i)),
\end{equation}
which is equivalent to
\begin{equation}
r_1(c-b)a_\gamma(F_i)<(-r_1)
\left(\frac{(\omega^2)}{2}+(c-b)^2\frac{(H^2)}{2}
\right)d_\gamma(F_i).
\end{equation}
If $r_1<0$, then $a_\gamma(F_i)>0$, which implies that
$d_\gamma(F_i)>0$.
If $\phi(F_i)>\phi(w_1)$, then $r_1>0$ implies that
$d_\gamma(F_i)>0$.

$\phi(F)=\phi(w_1)>\phi(F_2)>\phi(w_1)-1$ implies that
\begin{equation}
\begin{split}
(r_\beta(F)d_\beta(w_1)-r_\beta(w_1)d_\beta(F))(\omega^2)=&
2(a_\beta(F)d_\beta(w_1)-a_\beta(w_1)d_\beta(F))\\
(r_\beta(F_2)d_\beta(w_1)-r_\beta(w_1)d_\beta(F_2))(\omega^2)>&
2(a_\beta(F_2)d_\beta(w_1)-a_\beta(w_1)d_\beta(F_2)).
\end{split}
\end{equation}
Then
\begin{equation}
\begin{split}
r_1(c-b)a_\gamma(F)=&
-r_1\left(\frac{(\omega^2)}{2}+(c-b)^2\frac{(H^2)}{2} \right)d_\gamma(F)\\
r_1(c-b)a_\gamma(F_2)<&
-r_1\left(\frac{(\omega^2)}{2}+(c-b)^2\frac{(H^2)}{2} \right)d_\gamma(F_2).
\end{split}
\end{equation}
We first assume that $r_1>0$.
Since $-r_1 a_\gamma(F)>0,-r_1 a_\gamma(F_2)>0$, we have
\begin{equation}
\frac{d_\gamma(F)}{-r_1 a_\gamma(F)}>
\frac{d_\gamma(F_2)}{-r_1 a_\gamma(F_2)}.
\end{equation}
By Lemma \ref{lem:FM-w_1},
\begin{equation}
\frac{d_{\gamma'}(\Phi(F))}{\rk \Phi(F)}>
\frac{d_{\gamma'}(\Phi(F_2))}{\rk F_2},
\end{equation}
which is a contradiction.

If $r_1<0$, then
\begin{equation}
\frac{d_\gamma(F)}{-r_1 a_\gamma(F)}<
\frac{d_\gamma(F_2)}{-r_1 a_\gamma(F_2)}.
\end{equation}
By Lemma \ref{lem:FM-w_1},
\begin{equation}
\frac{d_{\gamma'}(\Phi(F))}{\rk \Phi(F)}>
\frac{d_{\gamma'}(\Phi(F_2))}{\rk F_2},
\end{equation}
which is a contradiction.

\subsection{A special case}

In this subsection, we assume that $\widehat{H}$ is ample.
\begin{prop}
Assume that $\gamma-\beta \in {\Bbb Q}H$ and
$\phi_\pm(w_1)<\phi_\pm(\widehat{\Phi}(E))$.
If $E$ is $\gamma'$-twisted semi-stable, 
then
$F=\widehat{\Phi}(E)$ is semi-stable with respect to
$(\beta,\omega_\pm)$.
\end{prop}

\begin{proof}
We note that $F$ is semi-stable with respect to
$(\beta,\omega)$.
Assume that there is an exact sequence
\begin{equation}
0 \to F_1 \to F \to F_2 \to 0
 \end{equation}
such that $F_1$ is a stable object with
$\phi(F_1)=\phi(E)$
and $\phi_\pm(F_1)>\phi_\pm(F)$.
Since $\phi_\pm(w_1)<\phi_\pm(F_1)$,
we see that $\Phi(F_1)$ and $\Phi(F_2)$ are 
torsion free sheaves.
\begin{NB2}
$\Hom({\cal E}_{|X \times \{ x_1 \}},F_1[2])=
\Hom(F_1,{\cal E}_{|X \times \{ x_1 \}})^{\vee}=0$
and $0 =\Hom({\cal E}_{|X \times \{ x_1 \}},F[2])
\to \Hom({\cal E}_{|X \times \{ x_1 \}},F_2[2])$ is surjective.
\end{NB2}
By the twisted semi-stability of $E$,
we see that
\begin{equation}
\frac{d_{\gamma'}(F_1)}{\rk F_1}=\frac{d_{\gamma'}(F)}{\rk F}.
\end{equation}
We have
\begin{equation}
\begin{split}
-r_1 d_\gamma(F)(\omega_\pm^2)=&
(r_\beta(F)d_\beta(w_1)-r_\beta(w_1)d_\beta(F))(\omega^2_\pm)\\
<&
2(a_\beta(F)d_\beta(w_1)-a_\beta(w_1)d_\beta(F))\\
=& (r_\beta(F)d_\beta(w_1)-r_\beta(w_1)d_\beta(F))(\omega^2)\\
=&-r_1 d_\gamma(F)(\omega^2).
\end{split}
\end{equation}
If $r_1>0$, then we have $(\omega_\pm^2)>(\omega^2)$.
Then
\begin{equation}
\begin{split}
&(r_\beta(F_1)d_\beta(F)-r_\beta(F)d_\beta(F_1))(\omega^2_+)\\
<&
2(a_\beta(F_1)d_\beta(F)-a_\beta(F)d_\beta(F_1))
\end{split}
\end{equation}
\begin{NB2}implies\end{NB2} that 
$r_\gamma(F_1)d_\gamma(F)-r_\gamma(F)d_\gamma(F_1)=
r_\beta(F_1)d_\beta(F)-r_\beta(F)d_\beta(F_1)<0$.
By Lemma \ref{lem:FM-w_1},
$\chi(E(-\gamma'))=-r_\gamma(F)/r_1$
and $\chi(\Phi(F_1)(-\gamma'))=-r_\gamma(F_1)/r_1$.
Hence
\begin{equation}
0>r_\gamma(F_1)d_\gamma(F)-r_\gamma(F)d_\gamma(F_1)
=-r_1 \chi(\Phi(F_1)(-\gamma'))d_\gamma(F)+
r_1 \chi(\Phi(F)(-\gamma'))d_\gamma(F_1).
\end{equation} 
Hence
$$
\frac{\chi(\Phi(F_1)(-\gamma'))}{\rk \Phi(F_1)}>
\frac{\chi(\Phi(F)(-\gamma'))}{\rk \Phi(F)},
$$
which is a contradiction.

If $r_1<0$, then
we have $(\omega_\pm^2)<(\omega^2)$.
Then
\begin{equation}
\begin{split}
&(r_\beta(F_1)d_\beta(F)-r_\beta(F)d_\beta(F_1))(\omega^2_-)\\
<&
2(a_\beta(F_1)d_\beta(F)-a_\beta(F)d_\beta(F_1))
\end{split}
\end{equation}
\begin{NB2}implies\end{NB2} that 
$r_\gamma(F_1)d_\gamma(F)-r_\gamma(F)d_\gamma(F_1)=
r_\beta(F_1)d_\beta(F)-r_\beta(F)d_\beta(F_1)>0$.
By Lemma \ref{lem:FM-w_1}, we have
\begin{equation}
0<r_\gamma(F_1)d_\gamma(F)-r_\gamma(F)d_\gamma(F_1)
=-r_1 \chi(\Phi(F_1)(-\gamma'))d_\gamma(F)+
r_1 \chi(\Phi(F)(-\gamma'))d_\gamma(F_1).
\end{equation} 
Hence
$$
\frac{\chi(\Phi(F_1)(-\gamma'))}{\rk \Phi(F_1)}>
\frac{\chi(\Phi(F)(-\gamma'))}{\rk \Phi(F)},
$$
which is a contradiction.
Therefore $F$ is semi-stable with respect to
$(\beta,\omega_\pm)$.
\end{proof}

We assume that $\gamma-\beta=\lambda H$.
Then

\begin{equation}
\begin{split}
& (a_\beta(v)d_\beta(w_1)-a_\beta(w_1)d_\beta(v))(H^2)\\
=& r_1 a_\gamma(v)\deg(\gamma-\beta)+r_1(d_\gamma(v)H+D,\gamma-\beta)
\deg(\gamma-\beta)-r_1 d_\gamma(v) (H^2)\frac{((\gamma-\beta)^2)}{2}\\
=& r_1 a_\gamma(v)\lambda(H^2)+r_1 \lambda^2 \frac{(H^2)^2}{2}d_\gamma(v)
\end{split}
\end{equation}
and
\begin{equation}
r_\beta(v)d_\beta(w_1)-r_\beta(w_1)d_\beta(v)=
r_\gamma(v)d_\gamma(w_1)-r_\gamma(w_1)d_\gamma(v)=-r_\gamma(w_1)d_\gamma(v).
\end{equation}

\subsubsection{}

Assume that $\gamma-\beta \in {\Bbb Q}H$.
Let $E$ be a $-\gamma'$-twisted semi-stable
sheaf on $X_1$ with $v(E)=v^{\vee}$.
Then $E^{\vee}$ fits in an exact triangle
$$
{\cal H}om_{{\cal O}_X}(E,{\cal O}_X) \to E^{\vee} \to
{\cal E}xt^1_{{\cal O}_X}(E,{\cal O}_X)[-1] \to
{\cal H}om_{{\cal O}_X}(E,{\cal O}_X)[1].
$$
Since ${\cal H}om_{{\cal O}_X}(E,{\cal O}_X)$ is $\mu$-semi-stable,
$F_1:=\Phi_{X_1 \to X}^{{\bf E}[1]}({\cal H}om_{{\cal O}_X}(E,{\cal O}_X))$
is a semi-stable object with $\phi(F_1)=\phi(w_1)$
with respect to
$(\beta,\omega)$.
Since ${\cal E}xt^1_{{\cal O}_X}(E,{\cal O}_X)$ is 0-dimensional,
$F_2:=
\Phi_{X_1 \to X}^{{\bf E}[1]}({\cal E}xt^1_{{\cal O}_X}(E,{\cal O}_X)[-1])$
is a semi-stable object with $\phi(F_2)=\phi(w_1)$.
Hence $F:=\Phi_{X_1 \to X}^{{\bf E}[1]}(E^{\vee})$
is a semi-stable object with $\phi(F)=\phi(w_1)$.

\begin{NB2}
Assume that there is an exact triangle
$$
F_1 \to F \to F_2 \to F_1[1]
$$
such that $\phi_{\min}(F_1) \geq \phi(w_1)$
and $\phi_{\max}(F_2)<\phi(E_1)$.
Then
$\Hom(F_1,{\bf E}_{|X \times \{x_1 \}}[k])=0$
for $k<0$ and
$\Hom(F_1,{\bf E}_{|X \times \{x_1 \}})=0$
except finitely many point $x_1 \in X_1$.
We also have
$\Hom(F_2,{\bf E}_{|X \times \{x_1 \}}[k])=0$
for $k \geq 2$.
We note that
\begin{equation}
\begin{split}
\Hom(F,{\bf E}_{|X \times \{x_1 \}}[k])=&
\Hom(\Phi(F),\Phi({\bf E}_{|X \times \{x_1 \}})[k])\\
=& \Hom(E^{\vee},{\frak k}_{x_1}[k-1])\\
=& \Hom(E^{\vee},{\frak k}_{x_1}^{\vee}[k+1])\\
=& \Hom({\frak k}_{x_1},E[k+1]).
\end{split}
\end{equation}
Hence $\Hom(F,{\bf E}_{|X \times \{x_1 \}}[k])=0$ for
$k \ne 0, 1$ and
$\Hom(F,{\bf E}_{|X \times \{x_1 \}})=0$ 
except finitely many points $x_1 \in X_1$.  
Then $E_1:=\Phi(F_1)^{\vee}$ and $E_2:=\Phi(F_2)^{\vee}$
are torsion free sheaves and we have an exact sequence
$$
0 \to E_2 \to E \to E_1 \to 0.
$$
We have
$$
\frac{\deg(E_2(\gamma'))}{\rk E_2}>
\frac{\deg(E(\gamma'))}{\rk E},
$$
which is a contradiction.
Hence $F$ is a semi-stable object with respect to $(\beta,\omega)$.
\end{NB2}
Assume that $\phi_\pm(w_1)>\phi_\pm(F)$.
Assume that there is an exact sequence
$$
0 \to F_1 \to F \to F_2 \to 0
$$
such that $F_2$ is a semi-stable object with respect to
$(\beta,\omega_\pm)$ such that
$\phi_\pm(F_1) >\phi_\pm(F)>\phi_\pm(F_2)$
and $\phi(F_1)=\phi(F_2)=\phi(w_1)$.
Then $\Hom(F_2,{\bf E}_{|X \times \{x_1 \}}[k])=0$
for $k \geq 2$.
Since $\Hom(F_2,{\bf E}_{|X \times \{x_1 \}}[k])=0$
for $k<0$ and 
$\Hom(F_2,{\bf E}_{|X \times \{x_1 \}})=0$
except finitely many point $x_1 \in X_1$,
Then $E_1:=\Phi(F_1)^{\vee}$ and $E_2:=\Phi(F_2)^{\vee}$
are torsion free sheaves and we have an exact sequence
$$
0 \to E_2 \to E \to E_1 \to 0.
$$
Since $\phi_\pm(F_2)<\phi_\pm(F)<\phi_\pm(w_1)$,
we have
\begin{equation}
\begin{split}
-r_1 d_\gamma(F)(\omega_\pm^2)=&
(r_\beta(F)d_\beta(w_1)-r_\beta(w_1)d_\beta(F))(\omega^2_\pm)\\
>&
2(a_\beta(F)d_\beta(w_1)-a_\beta(w_1)d_\beta(F))\\
=& (r_\beta(F)d_\beta(w_1)-r_\beta(w_1)d_\beta(F))(\omega^2)\\
=&-r_1 d_\gamma(F)(\omega^2).
\end{split}
\end{equation}
If $r_1>0$, then we have $(\omega_\pm^2)<(\omega^2)$.
Then
\begin{equation}
\begin{split}
&(r_\beta(F_2)d_\beta(F)-r_\beta(F)d_\beta(F_2))(\omega^2_-)\\
>&
2(a_\beta(F_2)d_\beta(F)-a_\beta(F)d_\beta(F_2))
\end{split}
\end{equation}
\begin{NB2}implies\end{NB2} that 
$r_\gamma(F_2)d_\gamma(F)-r_\gamma(F)d_\gamma(F_2)=
r_\beta(F_2)d_\beta(F)-r_\beta(F)d_\beta(F_2)<0$.
By Lemma \ref{lem:FM-w_1},
$\chi(E(-\gamma'))=-r_\gamma(F)/r_1$
and $\chi(\Phi(F_2)(-\gamma'))=-r_\gamma(F_2)/r_1$.
Hence
\begin{equation}
0>r_\gamma(F_2)d_\gamma(F)-r_\gamma(F)d_\gamma(F_2)
=-r_1 \chi(\Phi(F_2)(-\gamma'))d_\gamma(F)+
r_1 \chi(\Phi(F)(-\gamma'))d_\gamma(F_2).
\end{equation} 
Hence
$$
\frac{\chi(E_2(\gamma'))}{\rk E_2}=
\frac{\chi(\Phi(F_2)(-\gamma'))}{\rk \Phi(F_2)}>
\frac{\chi(\Phi(F)(-\gamma'))}{\rk \Phi(F)}
=\frac{\chi(E(\gamma'))}{\rk E},
$$
which is a contradiction.

If $r_1<0$, then
we have $(\omega_\pm^2)>(\omega^2)$.
Then
\begin{equation}
\begin{split}
&(r_\beta(F_2)d_\beta(F)-r_\beta(F)d_\beta(F_2))(\omega^2_+)\\
>&
2(a_\beta(F_2)d_\beta(F)-a_\beta(F)d_\beta(F_2))
\end{split}
\end{equation}
\begin{NB2}implies\end{NB2} that 
$r_\gamma(F_2)d_\gamma(F)-r_\gamma(F)d_\gamma(F_2)=
r_\beta(F_2)d_\beta(F)-r_\beta(F)d_\beta(F_2)>0$.
By Lemma \ref{lem:FM-w_1}, we have
\begin{equation}
0<r_\gamma(F_2)d_\gamma(F)-r_\gamma(F)d_\gamma(F_2)
=-r_1 \chi(\Phi(F_2)(-\gamma'))d_\gamma(F)+
r_1 \chi(\Phi(F)(-\gamma'))d_\gamma(F_2).
\end{equation} 
Hence
$$
\frac{\chi(E_2(\gamma'))}{\rk E_2}=
\frac{\chi(\Phi(F_2)(-\gamma'))}{\rk \Phi(F_2)}>
\frac{\chi(\Phi(F)(-\gamma'))}{\rk \Phi(F)}
=\frac{\chi(E(\gamma'))}{\rk E},
$$
which is a contradiction.
Therefore $F$ is semi-stable with respect to
$(\beta,\omega_\pm)$.
\end{NB}

{\it Acknowledgement.}
The third author would like to thank 
Professor A. Maciocia for valuable discussions on Bridgeland's 
stability during the program ``Moduli spaces''
at the Newton Institute.


\begin{thebibliography}{00}
\bibitem{AB}
Arcara, D., Bertram, A.,
{\it Bridgeland-stable moduli spaces for $K$-trivial surfaces,}
arXiv:0708.2247  

\bibitem{B:1}
Beauville, A.,
{\it Vari\'{e}t\'{e}s K\"{a}hleriennes dont la premi\`{e}re classe de Chern
est nulle,}
J. Diff. Geom. {\bf 18} (1983), 755--782

\bibitem{BBD}
Beilinson, A. A., Bernstein, J., Deligne, P.,
{\it Faisceaux pervers,} 
Analysis and topology on singular spaces, I (Luminy, 1981),  5--171, 
Ast\'{e}risque, {\bf 100}, Soc. Math. France, Paris, 1982 

\bibitem{Br:3}
Bridgeland, T.,
{\it Stability conditions on K3 surfaces,}
Duke Math. J.  {\bf 141} (2008), 241--291

\bibitem{Ha:1}
Hartmann, H.,
{\it Cusps of the K\"{a}hler moduli space and stability conditions 
on K3 surfaces,}
arXiv:1012.3121

\bibitem{HMS}
Huybrechts, D., Macri, E., Stellari, P.,
{\it Derived equivalences of $K3$ surfaces and orientation,} 
Duke Math. J. {\bf 149} (2009), 461--507

\bibitem{tori}
Kurihara, K.,  Yoshioka, K.,
{\it Holomorphic vector bundles on non-algebraic tori of dimension 2,}
manuscripta mathematica, {\bf 126} (2008), 143--166 

\bibitem{MM}
Maciocia, A., Meachan, C.,
{\it Rank One Bridgeland Stable Moduli Spaces on 
A Principally Polarized Abelian Surface,}
arXiv:1107.5304

\bibitem{MYY}
Minamide, H., Yanagida, S., Yoshioka, K.,
\emph{Fourier-Mukai transforms and 
the wall-crossing behavior for Bridgeland's stability conditions},
arXiv:1106.5217. 

\bibitem{S:1}
Simpson, C.,
{\it Moduli of representations of the fundamental group
of a smooth projective variety I,}
Publ. Math. I.H.E.S. {\bf 79} (1994), 47--129       

\bibitem{Y:7}
Yoshioka, K.,
{\it Moduli spaces of stable sheaves on abelian surfaces,}
Math. Ann. {\bf 321} (2001), 817--884, math.AG/0009001

%\bibitem[Y1]{Y:8}
%Yoshioka,~K.,
%\emph{Irreducibility of moduli spaces of vector bundles on $K3$ surfaces},
%math.AG/9907001.

%\bibitem[Y2]{Y:twist1}
%Yoshioka,~K.,
%\emph{Twisted stability and Fourier-Mukai transform I},
%Compositio Math. {\bf 138} (2003), 261--288.


\bibitem{Y:birational}
Yoshioka,~K.,
\emph{Fourier-Mukai transform on abelian surfaces},
Math. Ann. {\bf 345} (2009), 493--524.

\end{thebibliography}
\end{document}